\documentclass{amsart}

\usepackage{epigraph}

\usepackage{amssymb}
\usepackage{amsmath}
\usepackage{graphicx}

\newtheorem{theorem}{Theorem}
\newtheorem{lemma}[theorem]{Lemma}

\newtheorem{corollary}[theorem]{Corollary}
\newtheorem{definition}[theorem]{Definition}
\newtheorem{remark}[theorem]{Remark}

\DeclareMathOperator{\Var}{\mathop{Var}}

\title{Random Matrix Products: Universality and Least Singular Values}

\author{Phil Kopel}
\address{Department of Mathematics, University of Colorado at Boulder, Boulder CO 80309}
\email{philip.kopel@colorado.edu}

\author{Sean O'Rourke}
\address{Department of Mathematics, University of Colorado at Boulder, Boulder CO 80309}
\email{sean.d.orourke@colorado.edu}

\author{Van Vu}
\address{Department of Mathematics, Yale University, New Haven CT 06520
}
\email{van.vu@yale.edu}

\thanks{S. O'Rourke has been supported by NSF grant ECCS-1610003.  V. Vu is supported
	by NSF grant DMS-1307797 and AFORS grant FA9550-12-1-0083}

\date{}

\begin{document}
	\maketitle
	\begin{abstract}		
		We establish, under a moment matching hypothesis, the local universality of the correlation functions associated with products of
		$M$  independent iid random matrices, as $M$ is fixed, and  the sizes of the matrices tend to infinity.
		This generalizes an earlier result of Tao and the third author for the case $M=1$.

		We also prove Gaussian limits for the centered linear spectral statistics of products of $M$ independent iid random matrices. This is done in two steps. First, we establish the result for 
		product random matrices with Gaussian entries, and  then extend to the general case of non-Gaussian entries  by  another moment matching argument.  Prior to our result,  Gaussian limits were known only for the case $M=1$.   In a similar fashion, we establish Gaussian limits for the centered linear spectral statistics of products of independent truncated random unitary matrices. In both cases, we are able to obtain explicit expressions for the limiting variances. 
		
		The main difficulty in our study is that the entries of the product matrix are no longer independent. 
		Our key technical lemma  is a lower bound on the least singular value of the translated linearization matrix associated with the product of $M$ normalized independent random matrices with independent and identically distributed subgaussian entries. This lemma is of independent interest. 
	\end{abstract}
		
		\section{Introduction and Statement of Results}
		
	Random matrices with independent entries  have been among the most central objects of study in random matrix theory since the pioneering work of Ginibre in 1965 \cite{Ginibre}, and many important advances and applications have been established since that time (for a partial overview see \cite{BD}, and the excellent collection of references contained therein). 
		
		\begin{definition} \label{def:iidmatrix}
			An \emph{iid random matrix (with subgaussian decay)} is an $n \times n$ random matrix $M_n=(\xi_{i,j})_{1\leq i,j\leq n}$ where each entry $\xi_{i,j}$ is given by an independent copy of a subgaussian random variable $\xi$ with mean zero and unit variance, and whose real and imaginary parts are independent. 
		\end{definition}
		The distribution of a single entry of an iid random matrix is commonly called the \textit{atom distribution}.
		The requirement that atom distributions be \textit{subgaussian} means that there exists some constants $C,v>0$ (independent of $n$) such that for each $t>0$:
		\begin{align}
		\mathbf{P}\left(|\xi_{1,1}|>t \right)\leq Ce^{-vt^2}.
		\end{align}
		If each entry $\xi_{i,j}$ is given by a complex standard Gaussian random variable, the iid matrix is said to belong to the complex Ginibre ensemble, or more briefly GinUE. If the entries are real Gaussians, then the matrix is said to belong to the real Ginibre ensemble, or GinOE. The ensemble of matrices with independent centered Bernoulli entries provides another natural example of an iid random matrix.
		
		One of the most salient and important features of iid matrices is that they are non-Hermitian, and the methods used to study them are therefore frequently different from those employed in Hermitian random matrix theory. Notice that unlike in the Hermitian case we no longer have strictly real eigenvalues, which introduces a number of difficulties (for instance, limiting the usefulness of polynomials in studying general functions of the eigenvalues, which in turn limits the usefulness of combinatorial methods based on moment methods \cite{BS} \cite{BD} \cite{PK}). On the other hand, the lack of symmetry is occasionally useful and the mutual independence of all rows (or all columns) can simplify matters considerably in certain situations \cite{V}.

		In this paper, we study the product of $M$ independent iid random matrices. 
		The spectrum of such a product  has been an object of keen interest in random matrix theory, and many important advances have been made in recent years.  We refer the reader to \cite{AKR, AGT, AGT2, AB, AB2, AI, BC, BJW, Freal, GNT, GT, Ipsen, IK, Nemish, Nemish2, OS, OSVR} and references therein; however, since the subject has been so intensively studied, the collection of references above as well as our discussion below if far from complete.  This subfield of random matrix theory is not only motivated by its own intrinsic interest, as  the study of products of matrices is a very natural generalization of the study of a single random matrix, but is also motivated by connections to several of the applied sciences (the interested reader is directed to the impressive survey of applications assembled in \cite{Ipsen}). 
		
		The first important  result concerning  random matrix products is the $M$-fold circular law, which occupies the same position as either the semicircular law for Wigner matrices or circular law for independent entry matrices, and governs the limiting distribution of eigenvalues. This remarkable result says that as $n\to\infty$ the limiting distribution of the spectrum of the product of $M$ normalized $n$ by $n$ independent iid matrices (obeying certain moment assumptions), $X=n^{-M/2}X^{(1)}\dotsm X^{(M)}$, is supported on the unit disc $|z|<1$, and has the following density \cite{GT} \cite{OS}:
		\begin{align}
		\frac{1}{M\pi}|z|^{2/M-2}.
		\end{align}
		Notice that, although for even somewhat sizable $M$ the distribution of a single entry of the product matrix $X$ may be quite complicated, the limiting distribution is miraculously simple.
		
		Once the global distribution has been established, the next natural questions are the limiting laws of 
		linear statistics and universality of local statistics. In this paper, we  attack both problems and solve them under 
		certain moment conditions.  Throughout the paper, we always assume all iid random matrices have atom distribution with subgaussian decay as  described in Definition \ref{def:iidmatrix}.

		Our first result concerns the \textit{linear statistics} of random product matrices, which describe \textit{fluctuations} of the spectrum of a random matrix around its limiting distribution -- in this case, the $M$-fold circular law. The centered linear statistic associated with the random matrix $X$ (with eigenvalues denoted by $\lambda_1,\ldots,\lambda_n$) and some test function $f$ (whose regularity depends on the ensemble under consideration) is defined by the following formula:
		\begin{align}
			N_n[f]=\sum_{j=1}^n f(\lambda_j) -\mathbf{E}\left[\sum_{j=1}^n f(\lambda_j) \right].
		\end{align}
		The behavior of linear statistics is well understood for Wigner matrices (the unnormalized statistic converges to a normal distribution, which implies fluctuations comparable to a constant, in striking opposition to the fluctuations seen in the case of the sum of iid random variables, which are comparable to $\sqrt{n}$) up to some delicate questions about the optimal regularity constraints on the test function $f$ \cite{Sch} \cite{SW} \cite{LP} \cite{SSosh}. Much progress has been made in recent years on the linear statistics of independent entry ensembles \cite{OR} \cite{PK} \cite{RS}, even though due to difficulties associated with the failure of the Hermitian condition there is still much work to be done in that area (the failure of Hermiticity and the complex eigenvalues it implies means that analytic test functions are no longer dense in the space of smooth test functions, for instance, which limits the efficacy of trace methods), and substantial new ideas are likely to be required. To the best of the authors' knowledge, the linear statistics of products of iid matrices have not yet been investigated.
		
		Before we can state our result, we will need to introduce the concept of \textit{moment matching}. Two iid random matrices $X$ and $X^\prime$ with entries $\xi_{i,j}$ and $\xi_{i,j}^\prime$ respectively are said to match moments to order $k$ if, for all $1\leq i,j\leq n$, $a,b\geq 0$, $a+b\leq k$:
		\begin{align*}
			\mathbf{E} \left[\mbox{Re}(\xi_{i,j})^a\mbox{Im}(\xi_{i,j})^b\right]
			=
			\mathbf{E}\left[ \mbox{Re}(\xi^\prime_{i,j})^a\mbox{Im}(\xi^\prime_{i,j})^b
			\right].
		\end{align*}
		We will say that two product matrices, $X^{(1,1)}\dotsm X^{(1,M)}$ and $X^{(2,1)}\dotsm X^{(2,M)}$, match moments to order $k$ if each individual pair of factor matrices (meaning $X^{(1,j)}$ and $X^{(2,j)}$ for $1\leq j\leq M$) do. We will also need to define a relevant Sobolev norm.
		
		\begin{definition}
			Let $f$ be a real valued function defined on the complex plane, and let $\hat{f}(k)$ denote the $k$-th Fourier coefficient of the restriction of a function $f$ to the circle $|z|=R$:\footnote{Here, we use $\sqrt{-1}$ to denote the imaginary unit and reserve $i$ as an index.}
			\begin{align}
				\hat{f}(k)=\frac{1}{2\pi}\int_0^{2\pi} f(Re^{\sqrt{-1}\theta})
				e^{-\sqrt{-1}k\theta}d\theta.
			\end{align}
			The inner product $\left<f,g\right>_{H^{1/2}(|z|=R)}$ is then defined to be:
			\begin{align}
				\left<f,g\right>_{H^{1/2}(|z|=R)}=\sum_{k\in\mathbb{Z}} |k| \hat{f}(k) \overline{\hat{g}(k)}.
			\end{align}
			Finally, we set the Sobolev norm $\|\cdot\|_{H^{1/2}(|z|=R)}$ to be the norm induced by this inner product.  
		\end{definition}
		
		This Sobolev norm makes an appearance in the expression of the limiting variance of linear statistics of products of matrices matching the GinUE ensemble to four moments:
		
		\begin{theorem}
			\label{ginibre4mom}
			Let $f:\mathbb{C}\to\mathbb{R}$ be a test function with at least two continuous derivatives, supported in the spectral bulk $\left\{z\in\mathbb{C}:\tau_0<|z|<1-\tau_0\right\}$ for some fixed $\tau_0>0$. Fix an integer $M \geq 1$, and let $n^{-M/2}X^{(1)}\dotsm X^{(M)}$ be a matrix product such that each factor $X^{(i)}$ is an $n$ by $n$ iid random matrix (which are all jointly independent) with an  atom distribution matching the standard complex Gaussian distribution to four moments.
			Then the centered linear statistic associated with the test function $f$ and the product matrix $n^{-M/2}X^{(1)}\dotsm X^{(M)}$, denoted $N_n[f]$, converges in distribution as $n \to \infty$ to the mean zero normal distribution with the following limiting variance:
			\begin{align} \label{eq:variancels}
				\frac{1}{4\pi}\int_{|z|<1}\left|\nabla f(z)\right|^2 d^2z
				+\frac{1}{2}\|f\|^2_{H^{1/2}(|z|=1)}.
			\end{align}
			Notice that the limiting variance does not depend on $M$.
		\end{theorem}
		
		We will prove this result by way of a more general result establishing four moment universality for linear statistics of products of iid matrices, along with an entirely separate result establishing the appropriate central limit theorem for products of Ginibre matrices (the proof of which is based on the rotary flow combinatorial approach pioneered by Rider and Vir\'ag for a single Ginibre matrix \cite{RV}). If a central limit theorem for linear statistics was established for products of matrices with some known subgaussian distribution besides the complex Gaussian distribution (for instance, the real Gaussian), then our machinery (specifically, Theorem \ref{fmlthm} appearing below) would establish the same central limit theorem for all matrices which match that base case to four moments.
		
		There are four remarks which must be made about the technical assumptions appearing in the statement of Theorem \ref{ginibre4mom}. The first is that the condition requiring the entries of each matrix to be identically distributed is needed only to provide an appropriate least singular value estimate (specifically, we will need this assumption to be able to apply Theorem \ref{linsmallsig} below, otherwise the condition is not needed -- the substitution of an alternate condition that preserves the statement of Theorem \ref{linsmallsig} will not affect the proof of this result).  The second remark concerns the subgaussian decay condition on the atom distributions from Definition \ref{def:iidmatrix}.  Naturally, we expect our main results to hold under weaker assumptions (such as subexponential tails or the existence of a high enough number of finite moments).   The subgaussian decay condition is mostly used to control certain events with high enough probability, and it seems likely that more technical methods could be utilized to relax this assumption.  Due to the already technical nature of our proofs, we have not pursued this direction.  
		
		The third remark is that the assumption of two derivatives on $f$, as well as the condition requiring support in the spectral bulk, is necessary only for the moment matching argument. If one is only interested in the linear statistics in the specific case of complex Gaussian atom distribution, then, as will be clear from the proof, both of these conditions can be relaxed (see Theorem \ref{cltforginibre} for details).  
		
		The fourth remark concerns the assumption on the support of $f$, specifically that $f$ is supported away from both the origin and the spectral edge $|z|=1$.  Under this assumption the term $\|f\|^2_{H^{1/2}(|z|=1)}$ appearing in \eqref{eq:variancels} is zero; we include this term in the limiting variance since it matches the analogous result obtained for products of complex Ginibre matrices (Theorem \ref{cltforginibre} below), which does not feature this restriction on $f$.  For Theorem \ref{ginibre4mom}, the condition on the support of $f$ appears to be an artifact of the proof, and is necessary mainly to ensure good behavior of the entries of the resolvent of the Hermitian linearization of the product matrix. Lifting this requirement is a direction for future work (a similar condition encountered in the course of the proof of the local $M$-fold circular law \cite{Nemish} has recently been lifted \cite{GNT}, at least at the origin, but the methods employed there do not appear to directly extend to our case), and we suspect \eqref{eq:variancels} to still be the limiting variance once this condition is lifted.  
		
		The same technique that establishes Theorem \ref{ginibre4mom} in the base case of complex Ginibre matrices can also be applied to other product matrix models, as long as the eigenvalues constitute a rotationally invariant determinantal point process and the ensuing combinatorics work out reasonably. As an illustration, we obtain a  limit theorem for linear statistics in the case of \textit{truncated unitary random matrices}. A truncated unitary random matrix is a random matrix produced by taking the top-left $n$ by $n$ submatrix of a random unitary $K$ by $K$ matrix (distributed according to the relevant Haar measure), where $\tau=(K-n)/n$ is a fixed parameter. The study of these matrices is motivated in part by connections to topics such as time evolution in quantum mechanics or chaotic scattering on mesoscopic devices; the interested reader is directed to \cite{SZ}, \cite{AB}, \cite{AI} and the references within.
		
		\begin{theorem}
			\label{unit4mom}
			Let $f:\mathbb{C}\to\mathbb{R}$ be a real valued polynomial test function, and let $U^{(1)}\dotsm U^{(M)}$ be the product of $M$ $n$ by $n$ jointly independent truncated unitary random matrices, with fixed truncation ratio $\tau\in(1/2,1)$. Then, for fixed $M$ and $\tau$, in the large dimensional limit $n \to \infty$ the associated centered linear statistic, $N_n[f]$, converges in distribution to the normal distribution with mean zero and the following limiting variance:
			\begin{align}
				\frac{1}{4\pi}\int_{|z|\leq (1/(1+\tau))^{M/2}} \left|\nabla f(z)\right|^2
				d^2z 
				+\frac{1}{2}
				\|f\|^2_{H^{1/2}(|z| = |1+\tau|^{-M/2})}.
			\end{align}
		\end{theorem}
		
		Theorem \ref{unit4mom} deals only with polynomial test functions to simplify technical aspects of the argument; we believe a similar result should hold for more general test functions but this is outside the scope of the paper. The restriction to the disc $|z|\leq (1/(1+\tau))^{M/2}$ is explained by the limiting distribution of eigenvalues of products of truncated unitary matrices, which is defined by the following formula:
		\begin{align*}
			\phi^{M,\tau}(z)=\frac{1}{\pi M}\frac{\tau}{|z|^{2(1-1/M)}(1-|z|^{2/M})^2}
			\mbox{  }
			\mbox{on}
			\mbox{  }
			|z|\leq \left(\frac{1}{\tau+1}\right)^{M/2},
		\end{align*}
		where $\phi^{M,\tau}(z)$ vanishes on $|z|> \left(\frac{1}{\tau+1}\right)^{M/2}$. This formula is already in the literature \cite{AB}, but we encounter it during the course of the proof of Theorem \ref{unit4mom}.
		
		Our next result will be the local universality of the $k$-point correlation functions throughout the spectral bulk for products of iid factor matrices matching to four moments. Define the $k$-point correlation function $p^{(k)}$, for $1\leq k\leq n$, of the random points $\lambda_1,
	\ldots,\lambda_n \in \mathbb{C}$ by way of the following formula (with $\phi : \mathbb{C}^k \to \mathbb{C}$ denoting any arbitrary continuous, compactly supported test function):
		\begin{align*}
			\mathbf{E}\left[\sum_{i_1,\ldots,i_k \mbox{ distinct }}
			\phi(\lambda_{i_1},\ldots,\lambda_{i_k})\right]=
			\int_{\mathbb{C}^k}\phi(z_1,\ldots,z_k)p^{(k)}(z_1,\ldots,z_k)
			d^2z_1\ldots d^2z_k.
		\end{align*}
		We are abusing notation here by writing $p^{(k)}(z_1,\ldots,z_k)
		d^2z_1\ldots d^2z_k$, as $p^{(k)}$ is in general a measure and not a function (see \cite{TVPol} for a complete discussion of this issue). 
		
		The following universality result holds in the case of products $Z_n = X_n^{(1)} \cdots X_n^{(M)}$ of $n \times n$ iid matrices with complex atom distributions.  Before discussing the correlation functions for the eigenvalues of $Z_n$, we note that the largest eigenvalue of $Z_n$ is on the order of $n^{M/2}$, and we will need to first rescale the eigenvalues so that the typical spacing in the bulk is on the order of a constant.  There are several possible ways to rescale the eigenvalues.  We adopt the convention of taking the $M$-th root of each eigenvalue of $Z_n$ to create a new point process with $Mn$ points.  Formally, given the product $Z_n = X_n^{(1)} \cdots X_n^{(M)}$, we define the \emph{$M$-th root eigenvalue process associated with $Z_n$} as the point process on $\mathbb{C}$ consisting of the $Mn$ roots (counted with algebraic multiplicity) of the polynomial $z \mapsto \det(z^M I - Z_n)$, where $I$ denotes the identity matrix.  Clearly the eigenvalues of $Z_n$ can be recovered from this process by raising each root to the $M$-th power.  Our next result proves universality of the correlation functions corresponding to the $M$-th root eigenvalue process.  
		
		We define the open ball of radius $r>0$ centered at $z_0 \in \mathbb{C}$ as
		\[ B(z_0, r) := \{ z \in \mathbb{C} : |z - z_0| < r\}. \]
		\begin{theorem}
			\label{corruni}
			For $\beta=1,2$, let $Z^\beta_n=X_n^{(\beta,1)}\dotsm X_n^{(\beta,M)}$ be the product of $M$ independent iid $n$ by $n$ matrices with complex-valued entries and subgaussian decay. Assume that the factor matrices $X_n^{(1,i)}$ and $X_n^{(2,i)}$ (for $1\leq i\leq M$) match to four moments. Let $z_1,\ldots,z_k$ be complex numbers (which may depend on $n$) located in the spectral bulk $\tau_0\leq n^{-1/2}|z_i|\leq 1-\tau_0$ for some fixed positive constant $\tau_0>0$. 
			
			Let $p^{(k)}_\beta$ be the $k$-point correlation function of the $M$-th root eigenvalue process associated with $Z^\beta_n$. If $G:\mathbb{C}^k\to\mathbb{C}$ is a smooth function supported on $B(0,r_0)^k$ (for a small $r_0>0$ which depends on $\tau_0$) then for any sufficiently small positive constant $c_0>0$ the following estimate holds:
			\begin{align*}
				\Bigg| \int_{\mathbb{C}^k} G(w_1 ,\ldots, w_k)p_1^{(k)}(z_1+w_1,\ldots,
				z_k+w_k)d^2w_1\ldots d^2w_k\\
				-\int_{\mathbb{C}^k} G(w_1 ,\ldots, w_k)p_{2}^{(k)}(z_1+w_1,\ldots,
				z_k+w_k)d^2w_1\ldots d^2w_k
				\Bigg|\\
				\leq C_{[\tau_0,k, G]}n^{-c_0}.
			\end{align*}
			The constant $C_{[\tau_0,k,G]}$ depends on $\tau_0$, $k$, $M$, and (linearly) on the maximal $L^{\infty}$-norm bound on the first $2k+6$ derivatives of $G$.
		\end{theorem}
		
		While this result focuses on the $M$-th root eigenvalue process, our methods can be generalized to other scaling conventions for the eigenvalues as well; see Remark \ref{rem:scaling} for further details.  
		
		The proof of Theorem \ref{corruni} is based on an application of Theorem 2.1 in \cite{TVPol}, and may be particularly of interest in the case of complex Ginibre factor matrices.  Indeed, the derivation of the correlation functions for the product of independent complex Ginibre matrices can be found in \cite{AB2}, which finds that the limiting correlation functions match those of a single Ginibre matrix away from the origin.  Specifically, the eigenvalues of the product of $M$ independent $n \times n$ complex Ginibre matrices form a determinantal point process with kernel $K_n^{(M)}$ defined by the formula 
		\begin{align}
		K_n^{(M)}(z, w)=\sqrt{w_M(z)w_M(w)}\sum_{k=0}^{n-1}
		\frac{1}{(\pi k!)^M}(z\bar{w})^k. 
		\end{align}
		The weight function $w_M(z)$ appearing here depends only on the modulus of $z$ and is defined through the Meijer G-function, see Equation 2.3 in \cite{AB2}. As shown in \cite{AB2}, the kernel $K^{M}_n$ obeys the following scaling limit in the spectral bulk, away from the origin: 
		\begin{align*}
		\lim_{n\to\infty}M^{2-M}|\xi_i\xi_j|^{M-1}K_n^{(M)}&
		((\xi_i/\sqrt{M})^{M},(\xi_j/\sqrt{M})^M)\\
		&=\frac{1}{\pi} \left( \frac{\xi_i\overline{\xi_j}}{|\xi_i\xi_j|} \right)^{(1-M)/2}
		\exp\left(
		-\frac{1}{2}\left(|\xi_i|^2+|\xi_j|^2+\xi_i\overline{\xi_j} \right)
		\right).
		\end{align*} 
Up to the phase factor in front (which is irrelevant after taking determinants), this limiting kernel is the same as the limiting bulk behavior of a single Ginibre matrix (when $M=1$).  
		Theorem \ref{corruni} implies that this behavior should be universal for other matrix products matching the Ginibre ensemble to four moments.
		Both the spectral edge and the origin are dealt with in \cite{AB2} as well, but our universality result does not extend to these cases. It would be interesting to see if it can be generalized to include these cases as well.
		
		In the case of real entries the spectrum and therefore the correlation functions split into real and imaginary components.  For a thorough exposition on the spectrum in the real case, see for instance  \cite{TVPol}.  For such a point process, we define $p^{k,\mathit{l}}$ to be the correlation function for $k$ real  (which we will denote $\zeta_{i,\mathbb{R}}$) and $\mathit{l}$ complex points in the upper half place (which we will denote $\zeta_{j,\mathbb{C}_+}$). To be precise, for continuous test functions $\phi:\mathbb{R}^k\times\mathbb{C}^\mathit{l}\to\mathbb{C}$ with compact support we require:
		\begin{align*}
			&\mathbf{E}\left[\sum_{i_1,\ldots,i_k \mbox{ distinct }}
			\sum_{j_1,\ldots,j_\mathit{l} \mbox{ distinct }}
			\phi(\zeta_{i_1,\mathbb{R}},\ldots,\zeta_{i_k,\mathbb{R}},
			\zeta_{j_1,\mathbb{C}_+},\ldots,\zeta_{j_\mathit{l},\mathbb{C}_+})
			\right]\\ &=
			\int_{\mathbb{R}^k}\int_{\mathbb{C}^\mathit{l}_+}
			\phi(x_1,\ldots,x_k, z_1,\ldots,z_\mathit{l})	
			p^{k,\mathit{l}}(x_1,\ldots,x_k, z_1,\ldots,z_\mathit{l})	
			d^2 z_1\ldots d^2 z_\mathit{l}dx_1\ldots dx_k.
		\end{align*}
		This defines the mixed correlation functions $p^{k,\mathit{l}}$ over regions of the form $\mathbb{R}^k\times\mathbb{C}^\mathit{l}_+$. These function are then extended to include $z_i\in\mathbb{C}_{-}$ by requiring that they be symmetric under conjugation of any complex argument.
		
		The following universality result for the $M$-th root eigenvalue process holds in the case of products of matrices whose entries have real distributions. 
		\begin{theorem}
			\label{corruni2}
			For $\beta=1,2$, let $Z^\beta_n=X_n^{(\beta,1)} \dotsm X_n^{(\beta,M)}$ be the product of $M$ independent iid $n$ by $n$ matrices with identically distributed, real-valued entries and subgaussian decay, and let $x_1,\ldots,x_k \in\mathbb{R}$ and $z_1,\ldots,z_\mathit{l}\in\mathbb{C}$ be numbers (which are allowed to depend on $n$) located in the spectral bulk (meaning $\tau_0\leq n^{-1/2}|z_i|\leq 1-\tau_0$ and  $\tau_0\leq n^{-1/2}|x_i|\leq 1-\tau_0$ for some fixed $\tau_0 > 0$).
			
			Let $p^{k, \mathit{l}}_\beta$ be the mixed $(k,\mathit{l})$-correlation function for the $M$-th root eigenvalue process associated with $Z^\beta_n$, and assume that the factor matrices $X_n^{(1,i)}$ and $X_n^{(2,i)}$ (for $1\leq i\leq M$) match to four moments, and also that the correlation functions for one of the two processes satisfy the following estimate in the spectral bulk (which is the annulus defined by the equations $\tau_0\leq n^{-1/2}|z_i|\leq 1-\tau_0$ and $\tau_0\leq n^{-1/2}|x_i|\leq 1-\tau_0$):
			\begin{align}
				p_{\beta}^{2,0}(x_1,x_2)<C,\\
				p_{\beta}^{0,1}(z_1)<C.
			\end{align}
			If $G:\mathbb{R}^{k}\times\mathbb{C}^\mathit{l}\to\mathbb{C}$ is a smooth function supported on $[-r_0, r_0]^k \times B(0,r_0)^\mathit{l}$ (for some small $r_0>0$ depending on $\tau_0$), then for any sufficiently small $c_0>0$ the following estimate holds (where we write $d\mu=dy_1\cdots dy_kd^2w_1\cdots d^2w_\mathit{l}$):
			\begin{align*}
				\Bigg|\int_{\mathbb{R}^k}\int_{\mathbb{C}^\mathit{l}} G(y_1 ,\ldots, y_k, w_1,\ldots,w_{\mathit{l}})
				p_1^{k,\mathit{l}}(x_1+y_1,\ldots,x_k+y_k,z_1+w_1,\ldots
				,z_\mathit{l}+w_\mathit{l})d\mu\\
				-\int_{\mathbb{R}^k}\int_{\mathbb{C}^\mathit{l}}
				G(y_1 ,\ldots, y_k, w_1,\ldots,w_{\mathit{l}})
				p_2^{k,\mathit{l}}(x_1+y_1,\ldots,x_k+y_k,z_1+w_1,\ldots,
				z_\mathit{l}+w_\mathit{l})d\mu
				\Bigg|\\
				\leq C_{[\tau_0,k, \mathit{l}, G]}n^{-c_0}.
			\end{align*}
			The constant $C_{[\tau_0,k,\mathit{l},G]}$ depends on $C$, $k$, $\mathit{l}$, $\tau_0$, $M$, and (linearly) on the maximal $L^\infty$-norm bound on the first $2(k+l)+6$ derivatives of $G$.
		\end{theorem}
		
		Analogously to the complex case, this result may be of particular interest when comparing to the product of $M$ real Ginibre factor matrices; concrete expressions for the correlation kernels (which are expressed in terms of Pfaffians) as well as the correlations between real eigenvalues and the correlations between complex eigenvalues are derived and can be found in \cite{FI, IK}.
		
		As was the case with Theorem \ref{ginibre4mom}, the spectral bulk requirement (namely, that $z_i$ is chosen such that the estimate $\tau_0 \leq n^{-1/2}|z_i| \leq 1-\tau_0$ holds) which appears in both Theorem \ref{corruni} and Theorem \ref{corruni2} is necessary to invoke aspects of Nemish's proof of the local M-fold circular law for the spectral bulk \cite{Nemish}. In particular, this restriction appears to be an artifact of the method of proof and the authors believe that it should be possible to lift this requirement; an inspection of the proofs of Theorems \ref{corruni} and \ref{corruni2} should make clear exactly where the difficulties lie.
		
		The proofs of the preceding results (with the exception of Theorem \ref{unit4mom}) rely in large part on a least singular value estimate, Theorem \ref{linsmallsig} below, which is the main technical advance in the present volume and may be useful in a variety of other contexts. To present this result properly we first must discuss \textit{linearization}, one of the most useful techniques in the study of products of random matrices, which allows one to bypass many of the difficulties associated with the product structure. In a linearization argument one studies the product $X = X^{(1)}X^{(2)}\dotsm X^{(M)}$ by considering instead an associated $Mn$ by $Mn$ linearization block matrix $Y$, defined below:
			\begin{align*}
				Y=\left(
				\begin{array}{cccccc}
					0 & X^{(1)} & 0 & \cdots & \cdots & 0 \\
					0 & 0 & X^{(2)} & \ldots & \cdots & 0 \\
					0 & 0 & 0   & X^{(3)} & \cdots & 0 \\
					\vdots & \vdots & \vdots & \vdots & \ddots & \vdots \\
					0 & 0 & 0 & 0 & \cdots & X^{(M-1)} \\
					X^{(M)} & 0 & 0 & 0 & \cdots & 0 \\
				\end{array}\right).\end{align*}	
		As observed by Burda, Janik and Waclaw \cite{BJW}, if $\lambda_{1},\ldots,\lambda_{n}$ are the eigenvalues of $X$, then each $\lambda_{k}$ is an eigenvalue of $Y^{M}$ with multiplicity $M$. Indeed, to see this it suffices to raise $Y$ to the $M$-th power and notice that the diagonal block entries are cyclic permutations of $X^{(1)}\dotsm X^{(M)}$. It is this link that, in some situations, allows one to study the linearization matrix in lieu of the actual product matrix (which may be substantially more cumbersome). 
			
		Frequently it happens the matrix of interest is not $Y$ but is actually $Y-zI$, where $I$ is the identity matrix and $z$ is some parameter in the complex plane. Our next result deals primarily with the smallest singular value of such matrices:
			
			\begin{theorem}
			\label{linsmallsig}
			Let $X^{(i)}$, for $1\leq i \leq M$, denote a family of jointly independent $n$ by $n$ iid random matrices with subgaussian decay (with potentially different atom distributions) and, for any $z\in\mathbb{C}$, define the $Mn$ by $Mn$ matrix $Y(z)=Y-zI$ as below:
			\[Y(z)= \left( \begin{array}{ccccc}
			-zI & X^{(1)} &  0 & \cdots & 0 \\
			0 & -zI & X^{(2)} & \cdots & 0 \\
			\vdots & \vdots & \vdots & \ddots & \vdots \\
			0 & \cdots & 0 & -zI & X^{(M-1)} \\
			X^{(M)} & 0 & \cdots & 0 & -zI 
			\end{array} \right). \] 
			Let $\sigma_{1}(Y(z))$ denote the least singular value of $Y(z)$. For any sufficiently small constant $A > 0$, there exists $t_0 > 0$ (depending on $A$, $M$, and the distributions of the entries) such that if $n^{1/2-t_0} \leq |z| \leq n^{1/2+t_0} $, then 
			\begin{align}
				\mathbf{P}\left\{\sigma_{1}\left(Y(z)\right)\leq {n^{-1/2-A}}
				\right\}\leq C n^{-KA}.
			\end{align}
			Here, the constant $C$  depends only on $M$, $A$, and the distributions of the entries; $K$ depends only on $M$ and the distributions of the entries. The requisite smallness of $A$ depends on $M$ and the distributions of the entries.  
				\end{theorem}
			
			Previous efforts to control the smallest singular value of this matrix only established a lower bound of the form $n^{-B}$, where $B$ was (potentially a quite large) positive constant which was not explicitly determined \cite{OS} \cite{OSVR} \cite{GT}. 
			
			The random matrix theory literature contains many other results concerning the least singular value and related singularity probability bounds for various models of random matrices including \cite{AR, BD, BVW, Cook, Edel, LLTTJ, LPRTJ, R, RVleast, SST, TV1, TV3, TVbern, TVlo, TVsing, V, V2} and references therein.  The approach to proving Theorem \ref{linsmallsig} is roughly modeled on the approach taken to prove a similar result for a single independent entry random matrix in \cite{TV3}, however the block structure introduces substantial new difficulties that must be overcome. Namely, to control the smallest singular value we would like to control the distance between a given row of $Y(z)$ (say, the last row) and the span of the remaining rows, exploiting the independence between the last row and others through a suitable anti-concentration estimate. The difficulty is that the last row of $Y(z)$ contains quite a lot of zeros placed there by the block structure itself, and if the normal to the hyperplane spanned by the remaining rows places too much weight away from coordinates on which the last row is supported then no anti-concentration estimate could possibly be applied and the entire argument collapses. On the one hand such a scenario does not sound particularly probable (as it would seem to require significant coordination between the independent factor matrices), on the other hand there does not appear to be a simple way to easily rule out such occurrences. The bulk of the argument is occupied with this problem, which is not encountered in the case of a single iid random matrix (where all entries of all rows are random) and whose resolution depends on careful consideration of the interplay between the block structure and the linear spaces spanned by the small singular vectors of the individual factor matrices.
				
			As a corollary, Theorem \ref{linsmallsig} implies a least singular value bound on the original product matrix:
			\begin{corollary}
			\label{smallsigvalcor}
			Let $X_{1},\ldots, X_{M}$ be independent $n$ by $n$ iid random matrices with subgaussian decay (featuring potentially different atom distributions). For any $z \in \mathbb{C}$, let $\sigma_{1}(X(z))$ denote the smallest singular value of the random matrix $X(z)=X_1 \dotsm X_M-z^MI$, where $I$ denotes the identity matrix.  For any sufficiently small constant $A> 0$, there exists $t_0 > 0$ (depending on $A$, $M$, and the distributions of the entries) such that if $n^{1/2-t_0} \leq |z| \leq n^{1/2+ t_0}$, then 
			\begin{align}
				\mathbf{P}\left\{ 
				\sigma_{1}(X(z)) \leq |z|^{M-1}n^{-1/2-A}
				\right\}
				\leq C n^{-KA}.
			\end{align}
			Here, the constant $C$  depends only on $M$, $A$, and the distributions of the entries; $K$ depends only on $M$ and the distributions of the entries.  The requisite smallness of $A$ depends on $M$ and the distributions of the entries.
			\end{corollary}
					
		The remainder of this paper is organized as follows. In Section \ref{singval_sec}, we present the proof of Theorem  \ref{linsmallsig} and the proof of its corollary for non-linearized product matrices. In Section \ref{linstat_sec}, we prove Theorem \ref{ginibre4mom} in the base case of Gaussian entries, and also provide the proof of Theorem \ref{unit4mom}, which uses similar techniques. In Section \ref{4mom_sec} we prove Theorem \ref{ginibre4mom} in full generality, and in Section \ref{corrunisec} we prove Theorems \ref{corruni} and \ref{corruni2}. The proofs of Theorems \ref{corruni} and \ref{corruni2} reference the proof to Theorem \ref{ginibre4mom} to some degree, otherwise these sections may be read independently of one another.
		
\subsection*{Acknowledgements} The authors thank Gernot Akemann for his comments on an earlier draft of this manuscript.  The authors also thank the anonymous referees for their careful reading of the manuscript and their many detailed comments and suggestions.  
						
		\section{Notations}
		
		In this section, we will collect some elementary notations which we will use throughout the paper.
		
		We will write $f_n=O(g_n)$, $f_n \ll g_n$, $g_n \gg f_n$ or $g_n = \Omega(f_n)$ if $|f_n|\leq C|g_n|$ holds for all $n > C$ and for some fixed $C>0$. If the value of the constant $C$ depends on some parameters, we will denote this with subscripts, e.g., if the constant depends on the parameters $a_1,\ldots,a_k$, we would write $f_n=O_{a_1,\ldots,a_k}(g_n)$.   We will write $f_n=o(g_n)$ to denote $f_n/g_n\to 0$ as $n\to\infty$. We will also make heavy use of the following two basic definitions from probability theory:
		\begin{definition}
			A sequence of events $E_n$ holds with \emph{overwhelming probability} if for every fixed $A>0$ there exists a constant $C_A>0$ such that the estimate $\mathbf{P}(E_n)\geq 1-C_An^{-A}$ holds for all $n$.
		\end{definition}
		\begin{definition}
			A sequence of events $E_n$ holds with \emph{exponential probability} (or \emph{exponentially high probability}) if there exist constants $C,c>0$ such that the estimate $\mathbf{P}(E_n)\geq 1-C\exp(-cn)$ holds for all $n$.
		\end{definition}
		
		We will also make extensive use of the $1-, 2-$ and $\infty$-norms for vectors. Letting $v=(v_1,\ldots,v_n) \in \mathbb{C}^n$:
		\begin{align}
			&\|v\|_{1}=|v_1|+|v_2|+\cdots+|v_n|,\\
			&|v| = \|v\|_{2}=\sqrt{|v_1|^2+\cdots+|v_n|^2},\\
			&\|v\|_{\infty}=\max_j |v_j|.
		\end{align}
		For a matrix $M$, we define the operator norm $\|M\|_{op}$ as follows:
		\begin{align}
			\|M\|_{op}=\sup_{\|v\|_{2}=1}\|Mv\|_{2}.
		\end{align}
		For complex functions $f$ defined on the measure space $(X, \Sigma, \mu)$, we define the usual $1-$, $2-$ and $\infty-$ norms:
		\begin{align}
		&\|f\|_{1}=\int_X |f(x)|d\mu,\\
		&\|f\|_{2}= \sqrt{\int_X |f(x)|^2 d\mu},\\
		&\|f\|_{\infty}=\mbox{ess }\mbox{sup}|f(x)|.
		\end{align}
		In the last definition, we use the essential supremum of $|f(x)|$ to denote the smallest constant which bounds $|f(x)|$ almost everywhere with respect to the appropriate measure. On the linear algebra side of things, we will let $V^\perp$ denote the vector space of all vectors orthogonal to $V$, and will frequently use $X-z$ to denote the matrix $X-zI$, where $X$ is some matrix, $I$ is the identity matrix of matching dimension, and $z\in\mathbb{C}$.
		For a $p \times n$ matrix $A$ with $p \leq n$, we let $\sigma_1(A) \leq \cdots \leq \sigma_p(A)$ denote the ordered singular values of $A$.  In particular, 
		\[ \sigma_1(A) = \inf_{\|v\|_2 = 1} \| A v \|_2. \]
		
		We let $[m]$ denote the discrete interval $\{1, \ldots, m\}$.  For a finite set $S$, $|S|$ is the cardinality of $S$.  We often use $\sqrt{-1}$ for the imaginary unit and reserve $i$ as an index.  
						
		\section{Smallest Singular Value}
		\label{singval_sec}
			
		\subsection{Preliminary Definitions and Lemmas}
		
		In this section, we provide the proof of Theorem \ref{linsmallsig} and Corollary \ref{smallsigvalcor}. We begin by collecting some elementary definitions and lemmas which we will require, and also make some preliminary reductions. Let $X^{(i)}$ for $1\leq i\leq M$ and $Y(z)$ be as in the statement of Theorem \ref{linsmallsig}, and let the small positive parameter $A>0$ be arbitrary.  In several parts of the proof, we assume $M > 1$ as Theorem \ref{linsmallsig} and Corollary \ref{smallsigvalcor} both follow from \cite[Theorem 3.2]{TV3} in the case $M=1$.

		Our argument will rely on the following estimate on the singular values of the factor matrices $X^{(i)}$, which follows from the proof of Lemma 4.11 in \cite{BD}.
		
		\begin{lemma}
			\label{smallsingval}
			There exists a constant $\gamma\in(0,1)$ and a constant $c_0>0$ such that with overwhelming probability, the following estimate hold simultaneously for the singular values $\sigma_1(X^{(i)})\leq\cdots\leq \sigma_n(X^{(i)})$ of each factor matrix $X^{(i)}$ and for all choices of $0<\tau<\gamma$ (such that $n^{1-\tau}$ is an integer):
			\begin{align}
				\label{smallsingvalest}
				\sigma_{n^{1-\tau}}\left(X^{(i)}\right)\geq c_0 n^{1/2-\tau}.
			\end{align}
			The constants $\gamma$ and $c_0$ depend only on the atom distributions of the matrices in question.
		\end{lemma}
		
		We will also need the small ball probability bound, which appears as Corollary 6.3 in \cite{TV3}:
		
		\begin{lemma}
			\label{smallballprob}
			Let $\xi_1$,\ldots,$\xi_n$ be i.i.d. random variables with mean zero and variance one. Then there exists a constant $c>0$ such that
			\begin{align*}
				\mathbf{P}\left\{|\xi_1v_1+\cdots+\xi_nv_n-z|<c \right\}\leq 1-c
			\end{align*}
			for all $z\in\mathbb{C}$ and all unit vectors $v=(v_1, \ldots, v_n)\in\mathbb{C}^n$.
		\end{lemma}
		
		It is technically possible to operate in greater generality here: the assumption of identical distribution can be weakened to the more general technical condition that the random variables in question be $\kappa$-controlled, but we will not seek such refinements here. For more information on $\kappa$-controlled distributions, the reader is referred to \cite{TV3}. 
		
		We will also need the classical Chernoff bound for indicator variables; see for example \cite{BLM}.  
		\begin{lemma} \label{lemma:chernoff}
			Let $\mathcal{X}_1$,\ldots,$\mathcal{X}_n$ be independent indicator variables, and let $\mu=\mathbf{E}\left[\sum_{i=1}^n \mathcal{X}_i\right]$ denote the expectation of the sum. Then for any $0<\delta\leq 1$:
			\begin{align}
				\mathbf{P}\left\{\sum_{i=1}^n \mathcal{X}_i \geq (1+\delta)\mu \right\}
				\leq
				e^{-\mu\delta^2/3}
			\end{align}
			and:
			\begin{align}
				\mathbf{P}\left\{\sum_{i=1}^n \mathcal{X}_i \leq (1-\delta)\mu \right\}
				\leq
				e^{-\mu\delta^2/2}.
			\end{align}
		\end{lemma}
		
		We will frequently use the Chernoff bound in conjunction with a well known cardinality bound on epsilon nets. A set of vectors $\mathcal{N}_\epsilon\subseteq\Omega\subseteq\mathbb{C}^n$ is said to constitute an $\epsilon$-net for $\Omega$ if for any vector $v\in\Omega$ there exists $u\in\mathcal{N}_{\epsilon}$ such that $\|u-v\|_{2}\leq \epsilon$.
		
		\begin{lemma}
			\label{epsilonnet}
			For any $\epsilon > 0$, the unit sphere of $\mathbb{C}^{n}$ contains an $\epsilon$-net $\mathcal{N}_\epsilon$ whose cardinality satisfies:
			\begin{align*}
				|\mathcal{N}_\epsilon| \leq \left(1+\frac{2}{\epsilon} \right)^{2n}.
			\end{align*}
		\end{lemma}
		
		The reader is referred to the excellent resource \cite{V} for a proof. We will also require the following estimate on the sizes of nets:
		
		\begin{lemma}
			\label{cordbound}
			For $a,b \in (0,1)$, the set of unit vectors in $\mathbb{C}^n$ with at most $an$ non-zero coordinates admits a $b$-net with cardinality at most:
			\begin{align*}
				{ n \choose an}\left(\frac{C}{b}\right)^{2an}
				\leq \left(\frac{e}{a}\frac{C}{b}\right)^{2an}.
			\end{align*}
			Here, $C>0$ is an absolute constant.
		\end{lemma}
		
		This follows from Lemma \ref{epsilonnet} along with the well-known bound on binomial coefficients ${ n \choose k}\leq (ne/k)^k$ (valid for all $k$ less than $n$). 
		
		During the course of the argument, we will find it necessary to work with so-called \textit{compressible} and \textit{incompressible} vectors separately. Intuitively, this distinction captures whether or not a vector is well-approximated by a sparse vector. Precisely, we define a compressible vector, with parameters $a, b\in(0,1)$, to be a vector $v^\prime$ such that there exists a vector $v$ supported on at most $an$ coordinates with $\|v-v^\prime\|_{2}\leq b$. Let $\mbox{Comp}(a,b)$ denote the set of compressible unit vectors with parameters $a$ and $b$, and let $\mbox{Incomp}(a,b)$, the set of incompressible unit vectors, denote its complement on the unit sphere.
		
		Without loss of generality, we now place some assumptions on the matrices we will be working with, ruling out some very unlikely pathologies.  By the bounds on the operator norm of iid random matrices with subgaussian entries \cite{V} we will often work on the event
		\begin{align}
			\label{genopbound}
			\max\left\{\|X^{(1)}\|_{op},\ldots,\|X^{(M)}\|_{op} \right\}
			\leq \sqrt{n} \left(1+r_0\right), 
		\end{align}
which holds with exponential probability for $r_0$ sufficiently large.  
		
		Let $\tilde{X}^{(i)}$ be the $n-1$ by $n$ rectangular matrix formed by deleting the last row of $X^{(i)}$, and let $\tilde{Y}(z)$ be the $Mn-1$ by $Mn$ rectangular matrix formed by deleting the last row of $Y(z)$. By the upper bound on the operator norm of $X^{(i)}$, for sufficiently large and fixed $r_0>0$, the following holds:
		\begin{align}
			\label{upboundnorm}
			\max_i\left\{\|X^{(i)}\|_{op},
			\|\tilde{X}^{(i)}\|_{op}
			\right\}\leq \sqrt{n}(1+r_0)
		\end{align}
		with exponential probability.  
		
		\subsection{The Compressible Case}
		Our starting observation is that it is quite unlikely for any fixed selection of unit vector to be orthogonal to the row space of $Y(z)$. 
		
		\begin{lemma}
			\label{singvect}
			The following estimate holds for all $z\in\mathbb{C}$ and any fixed unit vector u:
			\begin{align*}
				\mathbf{P}\left\{
				\|Y(z)u\|_2
				\leq c_1\sqrt{n}
				\right\}
				\leq O\left(\exp(-c_0n)\right).
			\end{align*}
			The implied constants featured in this bound, as well as $c_0>0$ and $c_1>0$, can be taken to depend only on $M$ and the atom distributions of the factor matrices.
		\end{lemma}
		
		\begin{proof}
		Begin by letting $u$ be any fixed unit vector, and write $u=(u^{(1)},u^{(2)},\ldots,u^{(M)})$, where each $u^{(\mathit{l})}\in\mathbb{C}^n$. The estimate $\|u^{(\mathit{l})}\|_{2}\geq M^{-1/2}$ must hold for at least one $\mathit{l}$: without loss of generality, we will take this estimate to hold for $\mathit{l}=2$. Let us then expand (for $1\leq k\leq n$):
		\begin{align*}
			\left|\left(Y(z)u\right)_{k}\right|
			=\left|\left(\sum_{j=1}^n X^{(1)}_{k,j}u^{(2)}_j\right)
			-zu^{(1)}_k\right|.
		\end{align*}
		Define $T_k$ to be the indicator variable associated with the event that the following estimate holds:
		\begin{align}
			\left|\left(\sum_{j=1}^n X^{(1)}_{k,j}u^{(2)}_j\right)
			-zu^{(1)}_k\right|\leq \frac{c}{M^{1/2}}.
		\end{align}
		Since by assumption $u$ is a vector whose first $n$ coordinates have 2-norm at least $M^{-1/2}$, we are able to choose the constant $c>0$ appropriately (in a manner that depends only on the distribution of $X^{(1)}_{j,k}$) such that Lemma \ref{smallballprob}, the small ball probability bound, implies that
		\begin{align}
			\label{probof1forchern}
			\mathbf{P}\left\{T_k=1
			\right\}\leq 1-c.
		\end{align}
		
		We would like to bound the probability of the following rare event:
		\begin{align*}
		\sum_{i=1}^{n} T_{i} \geq 
		n\left(1-\frac{c}{2}\right)
		\end{align*}
		Clearly this probability is maximized when the inequality in (\ref{probof1forchern}) is an equality, so we may as well assume that $T_k=1$ with probability $1-c$. By the Chernoff bound for indicator variables, we have that for any $\delta>0$:
		\begin{align*}
			\mathbf{P}\left\{\sum_{i=1}^{n} T_{i} \geq 
			n(1-c)(1 + \delta)
			\right\}
			\leq \exp(-n(1-c)\delta^2/3).
		\end{align*}
		Taking the parameter $\delta$ to be $c/(2(1-c))$ (we can safely assume $c<1/2$, as we can always choose smaller $c$) and $c_0=c^2/(12(1-c))$:
		\begin{align*}
			\mathbf{P}\left\{\sum_{i=1}^{n} T_{i} \geq 
			n\left(1-\frac{c}{2}\right)
			\right\}
			\leq \exp(-c_0n).
		\end{align*}
		This means that, with exponential probability, for at least $(cn/4)$ coordinates, the following estimate holds:
		\begin{align*} 
			|(Y(z)u)_k|\geq \frac{c}{M^{1/2}}.
		\end{align*}
		Since we can bound the norm of $Y(z)u$ from below by the norm of its first $cn/4$ coordinates:
		\begin{align*}
			\mathbf{P}
			\left\{\|Y(z)u^{(1)}\|^2_{2} \leq \frac{c^3n}{4M }\right\}
			\leq O\left(\exp(-c_0n)\right).
		\end{align*}
		Setting $c_1=\sqrt{ \frac{c^3}{4M } }$ concludes the argument. \end{proof}
		
		As a quick corollary, let us use this lemma to handle the case of compressible vectors. The thrust of the argument is that, since the set of compressible vectors can be well-approximated by a relatively small $\epsilon$-net,  the union bound may be applied and yield nontrivial bounds.
		
		\begin{lemma}
			\label{complemma}
			For every $\epsilon > 0$, there exists $t_0, \theta > 0$ (depending on $\epsilon$, $M$, the constant $r_0$ from \eqref{genopbound} and the atom distribution of the factor matrices $X^{(1)}$,\ldots,$X^{(M)}$) such that the following estimate holds for all $|z| \leq n^{1/2 + t_0}$, for $a=1/\log(n)$ and for $b = \theta n^{-\epsilon}$: 
			\begin{align}
				\mathbf{P}\left\{
				\min_{u\in\mathrm{Comp}(a,b)}\left\|Y(z)u\right\|_2
				\leq c_1\sqrt{n}
				\right\}\leq O\left(\exp(-c_0n)\right). 
			\end{align}
			The implied constants featured in this bound, and also the constants $c_0>0$ and $c_1>0$, can be taken to depend only on $\epsilon$, $M$ and the atom distributions of the matrices $X^{(i)}$.
		\end{lemma}
		
		\begin{proof} It suffices to prove the bound for all sufficiently small values of $\epsilon > 0$, as the bound then trivially holds for all larger values of $\epsilon$ (since $\mathrm{Comp}(a,b') \subset \mathrm{Comp}(a,b)$ for all $b' \leq b$).  Let $\Omega_{a,b}$ denote a $b$-net of the set of all vectors with unit length and at most $an$ nonzero coordinates. Using Lemma \ref{cordbound} we can bound the cardinality of this net:
		\begin{align*}
			|\Omega_{a,b} | \leq { n \choose an}\left(\frac{C}{b}\right)^{2an}
			\leq \left(\frac{e}{a}\frac{C}{b}\right)^{2an}
			= \exp\left(2an\log\left(\frac{Ce}{ba}\right)\right).
		\end{align*}
		
		Set $a=1/\log(n)$ and $b=\theta n^{-\epsilon}$ for $\epsilon$ sufficiently small (to be determined momentarily) and $\theta > 0$ to be chosen shortly. Our net can then be chosen to satisfy the cardinality bound:
		\begin{align*}
			\left|\Omega_{a,b}\right|\leq
			\exp\left(\frac{2n}{\log(n)}\log\left({Ce \theta^{-1} n^{\epsilon}\log(n)} \right) \right)
			\leq O\left(\exp\left({c_3n}\right)
			\right),
		\end{align*}
		where $c_3$ can be chosen arbitrarily small by taking $\epsilon$ sufficiently small (and because  $\log(\log(n))/\log(n)$ tends to zero).  
		By the union bound and Lemma \ref{singvect}, if we take $c_3$ to be small enough (and hence $\epsilon$ small enough), we have the following probability estimate:
		\begin{align}
			\label{compunionbd}
			\mathbf{P}
			\left\{\min_{u\in\Omega_{a,b}}\|Y(z)u\|_{2} \leq c_1\sqrt{n}\right\}
			\leq O\left(\exp(-cn)\right). 
		\end{align}
		
		This proves the result over the net, it remains to extend this bound to the rest of the set of compressible vectors.  Recall that a vector $u$ in $\mbox{Comp}(a,b)$ is distance at most $b$ away from a vector with at most $an$ nonzero coordinates, and therefore is at most distance $2b$ away from a vector $u^\prime$ in $\Omega_{a,b}$. Suppose there exists a vector $u \in \mathrm{Comp}(a,b)$ such that $\|Y(z) u\| \leq \frac{c_1}{10} \sqrt{n}$ and that $\|Y(z) \| \leq (2 + r_0) {n}^{1/2 + t_0}$. Then there exists a vector $u' \in \Omega_{a,b}$ such that $\|u - u'\| \leq 2b$.  Hence, by the triangle inequality, taking $\theta  = \frac{c_1}{30 (2 + r_0)}$ and $t_0 = \epsilon$ gives
\[ \|Y(z) u' \| \leq 2b (2 + r_0) {n}^{1/2+t_0} + \frac{c_1}{10} \sqrt{n} \leq c_1 \sqrt{n}. \]
In view of \eqref{compunionbd}, we conclude that
\begin{align*} 
	\mathbf{P}&\left\{
				\min_{u\in\mathrm{Comp}(a,b)}\left\|Y(z)u\right\|_2
				\leq \frac{c_1}{10} \sqrt{n} \text{ and } \|Y(z) \| \leq (2 + r_0) {n}^{1/2 + t_0}
				\right\} \\
				&\qquad \leq \mathbf{P}
			\left\{\min_{u\in\Omega_{a,b}}\|Y(z)u\|_{2} \leq c_1\sqrt{n}\right\} \\
			&\qquad \leq O\left(\exp(-cn)\right). 
\end{align*}
Since the norm bound $\|Y(z)\| \leq (2 + r_0) {n}^{1/2+ t_0}$ holds with exponential probability (due to \eqref{genopbound} and the assumption that $|z| \leq n^{1/2 + t_0}$), the claim follows.  \end{proof}
		
		We will utilize Lemma \ref{complemma} at the end of the section.  Notice that the proof to  Lemma \ref{complemma} goes over if we delete a row from either a factor matrix $X^{(i)}$ or ${Y}(z)$ (simply by replacing $n$ with $n-1$ as appropriate gives the equivalent results in this rectangular case). 
		
		\begin{lemma}
			\label{complemma2}
			For every $\epsilon > 0$, there exists $t_0, \theta > 0$ (depending on $\epsilon$, $M$, the constant $r_0$ from \eqref{genopbound} and the atom distribution of the factor matrices $X^{(1)}$,\ldots,$X^{(M)}$) such that the following estimate holds for all $|z| \leq n^{1/2 + t_0}$, for $a=1/\log(n)$ and for $b = \theta n^{-\epsilon}$: 
			\begin{align*}
				\mathbf{P}\left\{
				\min_{u\in\mathrm{Comp}(a,b)}\left\|\tilde{Y}(z)u\right\|_2
				\leq c_0\sqrt{n}
				\right\}\leq O\left(\exp(-cn)\right).
			\end{align*}
			The constants in this bound depend only on $\epsilon$, $M$, and the atom distributions of the matrices $X^{(i)}$.  
		\end{lemma}
		
		This extension will be useful for some of the arguments we will make later on, as a quick way of ruling out certain problematic scenarios without repeating all the details of this section.  
		
		\subsection{The Incompressible Case, Part I}
		
		We need to collect some preliminary lemmas in preparation for dealing with incompressible vectors. 
		
		First, an overview. Recall that $\tilde{Y}(z)$ denotes the rectangular matrix produced by deleting the last row of $Y(z)$. In this section, we will endeavor to show that, with high probability, any unit vector $v$ which is orthogonal to the rows of $\tilde{Y}(z)$ has the property that its first $n$ coordinates are not too close, in a certain sense, to the singular vectors corresponding to the small singular values of the product of the first $M-1$ factor matrices. The end game is to control the unit normal to the row space of $\tilde{Y}(z)$, and to show that it does not place the vast majority of its weight away from its leading $n$ coordinates. 
		
		Observe that if $v$ is orthogonal to the first $Mn-1$ rows of $Y(z)$, the block structure which results from linearization implies that 
			\begin{equation*} 
				v^{(1)} = \frac{1}{z^{M-1}} X^{(1)} \cdots X^{(M-1)} v^{(M)} 
			\end{equation*}
			and
			\begin{equation*} 
				\tilde{I} v^{(M)} = \frac{1}{z} \tilde{X}^{(M)} v^{(1)}, 
			\end{equation*}
			where $\tilde{I}$ is the $(n-1) \times n$ matrix formed from the $n \times n$ identity matrix by removing the last row.  
		In view of these identities, we would like to show that a specific class of unit vectors $v$ satisfy an estimate which looks like:
		\begin{align}
			\left\|\prod_{i=1}^{M-1}\left(\frac{1}{z}
			X^{(i)}\right)v
			\right\|_2 > n^{-\epsilon_0}
		\end{align}
		for an appropriate set of values of $z \in \mathbb{C}$ (this set of complex numbers will depend on $M$).  One of the main technical challenges will be showing that $\epsilon_0$ can be chosen independent of $M$.
		
		We will need to get a handle on the small singular values of the matrices $X^{(i)}$ to accomplish our goal, so let's start there. We will let $\sigma^{(i)}_1 \leq \sigma^{(i)}_2 \leq \cdots \leq \sigma^{(i)}_n$ denote the ordered singular values of $X^{(i)}$, and let $u_j^{(i)}$ denote some choice of associated unit singular vectors:
		\begin{align}
			\left(X^{(i)}\right)^{*}\left(X^{(i)}\right)u_j^{(i)}
			=\left(\sigma_j^{(i)}\right)^2 u_j^{(i)}.
		\end{align}
		Since we have a decent estimate on the number of small singular values of our factor matrices, we will try to isolate all vectors which interact with these singular values into a space of modest dimension. The advantage here is that we may now apply epsilon-net arguments which are not admissible when dealing with incompressible vectors more generally. The construction of this space is straightforward and is the occupation of the next lemma. The parameter $\tau$ determines exactly which singular values we will categorize as small, and an explicit selection of $\tau$ will be made later in the argument.
		
		\begin{lemma}
			\label{vtauexist}
			Let $c,\tau>0$ be sufficiently small constants. With overwhelming probability, there exists a linear subspace $V_\tau$ with dimension at most $O(n^{1-\tau})$ such that for all unit vectors $v$ orthogonal to $V_\tau$:
			\begin{align} \label{eq:tauexist}
			\left\|X^{(1)} \dotsm X^{(M-1)}v\right\|_2 \geq \frac{cn^{(1/2)(M-1)}}
			{n^{(M-1)\tau}}.
			\end{align} 
			In addition, $V_\tau$ can be taken to be the linear subspace spanned by the singular vectors corresponding to the $O(n^{1-\tau})$ smallest singular values of $X^{(1)}\dotsm X^{(M-1)}$.  
		\end{lemma}
		
		\begin{proof}
		We first construct a subspace $V_\tau$ with the desired properties; in the second half of the proof, we will show that $V_\tau$ can be taken to be the linear subspace spanned by the singular vectors corresponding to the $O(n^{1-\tau})$ smallest singular values of $X^{(1)}\dotsm X^{(M-1)}$.  
		Indeed, for each $i\in \{1,\ldots,M-1 \}$, let $W_\tau^{(i)}$ denote the vector space which is given by the span of $u_j^{(i)}$ for $1\leq j\leq n^{1-\tau}$. With overwhelming probability, we may assume that each of the factor matrices $X^{(1)},\ldots,X^{(M-1)}$ is both invertible (by Proposition 27 in  \cite{TV2}) and satisfies 
		(\ref{smallsingvalest}) (with the same $\tau$ as appears in the statement of the lemma, as we can take $\tau$ to be sufficiently small.)
		
		Define the space $\tilde{V}_\tau^{(M-1)}$ to simply be equal to $W_\tau^{(M-1)}$, 
		and define $Z_\tau^{(M-1)}$ to be the image of the vector space $\left(W_\tau^{(M-1)}\right)^\perp$ under the matrix  $\left(X^{(M-1)}\right)^{-1}$, which we may write $\left(X^{(M-1)}\right)^{-1}\left(W_\tau^{(M-1)}\right)^\perp$. Define $\tilde{V}_\tau^{(M-2)}$ to be the span of the union of $\tilde{V}_\tau^{(M-1)}$ with $\left(Z_\tau^{(M-1)}\right)^\perp$. Then the dimension of $\tilde{V}_\tau^{(M-2)}$ is at most $2n^{1-\tau}$ (because it is the span of the union of two spaces each with dimension at most $n^{1-\tau}$). 
		
		Because $\tilde{V}_\tau^{(M-2)}$ contains $\tilde{V}_\tau^{(M-1)}$, Lemma \ref{smallsingval} implies that every unit vector $v$ orthogonal to $\tilde{V}_\tau^{(M-2)}$ meets the following condition:
		\begin{align*}
			\|X^{(M-1)}v\|_{2}\geq cn^{1/2-\tau}.
		\end{align*}
		Further, the orthogonality between $v$ and $\tilde{V}_\tau^{(M-2)}$ ensures that $v\in Z_\tau^{(M-1)}$ by construction, and therefore that the vector $X^{(M-1)}v$ is orthogonal to the space $W_{\tau}^{(M-2)}$. Lemma \ref{smallsingval} then guarantees that:
		\begin{align*}
		\|X^{(M-2)}X^{(M-1)}v\|_{2}\geq cn^{1/2-\tau}\|X^{(M-1)}v\|_{2}	.
		\end{align*}
		
		Combining the preceding two estimates:
		\begin{align*}
			\left\|X^{(M-2)}X^{(M-1)}
			v\right\|_{2}\geq 
			c^2n^{1-2\tau}.
		\end{align*}	
		Continuing the construction iteratively, we construct a space $\tilde{V}_\tau^{(1)}$ with dimension at most $O(n^{1-\tau})$ such that for all unit vectors $v$ orthogonal to $\tilde{V}_\tau^{(1)}$:
		\begin{align}
		\label{orthogtovtauest}
		\left\|X^{(1)}\dotsm X^{(M-2)}X^{(M-1)}
		v\right\|_{2}\geq 
		c^{M-1}n^{(1/2)(M-1)-(M-1)\tau}.
		\end{align}
		
		This shows that the subspace $V_\tau$ has exactly the properties we require.  To complete the proof, we need to show that $V_\tau$ can be taken to be spanned by the singular vectors of $X^{(1)}\dotsm X^{(M-1)}$.  Indeed, we now use $V_\tau$ to construct another linear subspace $\Sigma_\tau$ which satisfies the same properties, but which is spanned by the singular vectors of the product.  Let $d$ be the dimension of $V_\tau$.  Fix a realization in which \eqref{eq:tauexist} holds and in which $d = O(n^{1-\tau})$.  Let $\Sigma_\tau$ denote the linear subspace spanned by the singular vectors of $X^{(1)}\dotsm X^{(M-1)}$ corresponding to singular values which are strictly smaller than $\frac{cn^{(1/2)(M-1)}}{n^{(M-1)\tau}}$ (if there are no such singular values, take $\Sigma_\tau$ to be the trivial subspace).  By the orthogonality of the singular vectors, it follows that if $v$ is a unit vector orthogonal to $\Sigma_\tau$, then
		\[ \| X^{(1)}\dotsm X^{(M-1)} v \|_2 \geq \frac{cn^{(1/2)(M-1)}}{n^{(M-1)\tau}}. \]
		It remains to show that the dimension of $\Sigma_\tau$ is $O(n^{1-\tau})$.  This follows from \eqref{eq:tauexist}.  Indeed, by the minimax principle for singular vectors (see, for instance, \cite[Problem III.6.1]{Bhatia}), it follows that
		\begin{align*}
			\sigma_{d+1}(X^{(1)}\dotsm X^{(M-1)}) &\geq \min_{v \in V_{\tau}^\perp, \|v \| = 1} \| X^{(1)}\dotsm X^{(M-1)} v \| \\
			&\geq \frac{cn^{(1/2)(M-1)}}{n^{(M-1)\tau}}. 
		\end{align*}
		This implies that $\dim(\Sigma_\tau) \leq d = O(n^{1-\tau})$, completing the proof of the lemma.  
		\end{proof}
		
		Next, we leverage the relatively small size of the space $V_\tau$ to control its behavior much in the same way we controlled compressible vectors. Specifically, we will see that it is unlikely for a unit vector in $V_\tau$ to be the first $n$ coordinates of a vector approximately normal to both the first $(M-1)n$ rows of $\tilde{Y}(z)$ (which we may condition on) and the remaining rows (which are still random). 
		
		\begin{lemma}
			\label{singspacelemma}
			Let $\tilde{I}_{n}$ be the $n-1$ by $n$ matrix formed by deleting the last row of the $n$-dimensional identity matrix, and assume that $z\in\mathbb{C}$ is such that $|z|\in [n^{1/2-\delta}, n^{1/2+\delta}]$ for some choice of $\delta\in(0,\frac{1}{4(M+1)})$. Then with overwhelming probability the matrices $X^{(1)}$,\ldots,$X^{(M-1)}$ are such that the following probability estimate holds with respect to the random matrix $\tilde{X}^{(M)}$:
			\begin{align*}
				\mathbf{P}\left\{\min_{\|v\|=1, v\in V_\tau}
				\left\|\left(z\tilde{I}_{n}-
				\tilde{X}^{(M)}
				\prod_{h=1}^{M-1}\left(\frac{1}{z}X^{(h)}\right)
				\right)v\right\|_2 \leq C_\delta n^{1/2-(M+1)\delta}
				\right\}\\
				\leq C_{\delta}^{-1} \exp(-c_\tau n).
			\end{align*}
			The constant $C_\delta>0$ depends on $\delta$, on $M$ and on the distribution of the matrix entries, while $c_\tau$ depends on the distribution of the matrix entries and also on the choice of $\tau$.
		\end{lemma}
				
		\begin{proof}
		By Lemma \ref{vtauexist}, the dimension of $V_\tau$ is at most $O(n^{1-\tau})$ with overwhelming probability. Assume this is so, and choose an $(1/\sqrt{n})$-net of the unit ball of $V_\tau$, which we will denote by $\Omega_\tau$. By Lemma \ref{epsilonnet} we can take this net to have cardinality at most:
			\begin{align} \label{eq:Omegasizebnd}
				\left| \Omega_\tau \right| \leq \left(1+2n^{1/2}\right)^{2n^{1-\tau}}
				\leq \exp\left([(3/2)n^{-\tau}\log(n)] n\right)
			\end{align}
			for $n$ sufficiently large.  
		Crucially, as long as $\tau>0$ this cardinality will grow slower than $\exp(cn)$ for any fixed choice of $c>0$. Now, fix an arbitrary choice of $v\in\Omega_\tau$ and, for notational simplicity, define the vector $y_v$ as follows:
			\begin{align*}
				y_v=\left(\frac{1}{z}X^{(1)} \right)\left(\frac{1}{z}X^{(2)} \right)
				\left(\frac{1}{z}X^{(3)} \right)\dotsm\left(\frac{1}{z}X^{(M-1)} \right)
				v.
			\end{align*}
		For $y_v \neq 0$, Lemma \ref{smallballprob}, the small ball probability bound, implies the following anti-concentration estimate for $1\leq j \leq n-1$:
			\begin{align*}
				\mathbf{P}\left\{ \left|\sum_{i=1}^{n} 
				\tilde{X}^{(M)}_{j,i}(y_v)_i-zv_j
				\right|
				\leq c\|y_v\|_{2}
				\right\}\leq 1-c.
			\end{align*}
		Let $T_j$, for $1\leq j \leq n-1$, be the indicator variable associated with the event that the following estimate holds:
			\begin{align*}
				\left|\sum_{i=1}^n \tilde{X}^{(M)}_{j,i}(y_v)_i-zv_j\right|^2
				\geq c^2\|y\|_{2}^{2}.
			\end{align*}
			By the small ball probability bound, each $T_j$ equals $1$ with probability at least $c$, and by the Chernoff inequality (Lemma \ref{lemma:chernoff}) for sums of indicator variables:
			\begin{align*}
				\mathbf{P}\left\{\sum_{j=1}^{n-1} T_{j} \leq \frac{c}{2}(n-1)
				\right\}
				\leq \exp(-c(n-1)/8).
			\end{align*}
			
			On the complement of this event at least $(c/2)(n-1)$ many of the $T_j$ must then be equal to 1. Consequently:
			\begin{align*}
				\mathbf{P}\left\{\left\| \tilde{X}^{(M)}y_v-z\tilde{I}_{n}v
				\right\|_{2}^2 \leq \frac{c^3}{2} \|y_v\|_{2}^{2}(n-1)\right\}\\
				\leq \exp(-cn/8)
			\end{align*}
			for a slightly different constant $c > 0$ and all $n$ sufficiently large.  To proceed further, we need to have an estimate of the quantity $\|y_v\|_{2}$, which we will handle by inspecting various cases. If we assume that $\|y_v\|_{2} \geq (n^{1/2}/|z|)^{M-1}$, then we can conclude that there exist constants $C_0,c_0>0$ depending only on the entry distributions such that for this choice of unit vector $v$:
			\begin{align}
				\label{ssv1}
				\mathbf{P}\left\{\left\| \tilde{X}^{(M)}y_v-z\tilde{I}_{n}v
				\right\|_{2} \leq C_0|z|(n^{1/2}/|z|)^{M}\right\}
				\leq C_0^{-1} \exp(-c_0n).
			\end{align}
			If we instead assume that $\|y_v\|_{2} \geq n^{-2\delta}$,  then we can similarly conclude that there exist constants $C_0,c_0>0$ depending only on the entry distributions such that for this choice of unit vector $v$:
			\begin{align}
				\label{ssv1alt}
				\mathbf{P}\left\{\left\| \tilde{X}^{(M)}y_v-z\tilde{I}_{n}v
				\right\|_{2} \leq C_0 n^{1/2-2\delta}\right\}
				\leq C_0^{-1} \exp(-c_0n).
			\end{align}
			
			These two cases handle the situations where $y_v$ is relatively large, we now consider the case where $y_v$ is relatively small. Specifically, assume that:
			\begin{align}
				\|y_v\|_{2} \leq \min\left\{(n^{1/2}/|z|)^{M-1},n^{-2\delta}\right\}. 
			\end{align}
			The Chernoff bound method employed above is not helpful for such $y_v$, as the bounds we would obtain would be too weak to establish the desired result. We therefore employ a different approach in this case, which takes advantage of the largeness of $|z|$ instead of the randomness of the rows of $\tilde{X}^{(M)}$. For this method to work we need to make sure that $v$ doesn't put too much mass on its last coordinate, so for now also assume that the following estimate holds:
			\begin{align}
				\label{normincomp}
				\|\tilde{I}_{n}v\|_{2}\geq n^{-\delta/2}.
			\end{align} 
			We will justify this assumption in due course, but first notice that the bound on the operator norm of $\tilde{X}^{(M)}$, along with the smallness of $\|y_v\|_{2}$, imply:
			\begin{align*}
				\|\tilde{X}^{(M)}y_v\|_{2}\leq (1+r_0)n^{1/2-2\delta}
			\end{align*}
			with exponential probability.  
			Comparing magnitudes, we have the estimate:
			\begin{align*}
				\|z\tilde{I}_{n}v-\tilde{X}^{(M)}y_v\|_{2}\geq 
				\left(|z|n^{-\delta/2}-(1+r_0)n^{1/2-2\delta}\right). 
			\end{align*}
			Using our assumptions on the magnitude of the parameter $|z|$ and also bounding $3\delta/2$ by $2\delta$, we obtain the following estimate:
			\begin{align}
				\label{ssv2}
				\|z\tilde{I}_{n}v-\tilde{X}^{(M)}y_v\|_{2}\geq 
				cn^{1/2-2\delta}
			\end{align}
			for a constant $c > 0$.  
			
			It remains to justify (\ref{normincomp}), which, as in the case of compressible vectors, boils down to the fact that the product of our random matrices with any fixed unit vector is not small with very high probability. To that end, let $e_n\in\mathbb{C}^n$ denote the unit Cartesian coordinate vector supported only on the $n$-th coordinate. By the same argument as we have used in the compressible case (specifically, by the Chernoff bound calculation in the proof of Lemma \ref{singvect}):
			\begin{eqnarray*}
				\mathbf{P}\left\{\|X^{(M-1)}e_n\|_2\leq C\sqrt{n}\right\}
				\leq O\left(\exp(-cn)\right).
			\end{eqnarray*}
			Iterating this argument and applying our assumptions on the magnitude of $|z|$, it is exponentially likely that:
			\begin{align*}
				\left\|\left(\frac{1}{z}X^{(1)} \right)\left(\frac{1}{z}X^{(2)} \right)
				\left(\frac{1}{z}X^{(3)} \right)\dotsm\left(\frac{1}{z}X^{(M-1)} \right)
				e_n\right\|_2\geq cn^{(M-1)/2}|z|^{-(M-1)}.
			\end{align*}
			On the other hand, if $v$ is a vector such that $\| \omega e_n-v\|_{2}\leq Cn^{-\delta/2}$ for some $\omega \in \mathbb{C}$ with $|\omega| = 1$, then (by the operator norm bound on the matrices $X^{(i)}$):
			\begin{align*}
				\left\|\left(\frac{1}{z}X^{(1)} \right)\left(\frac{1}{z}X^{(2)} \right)
				\left(\frac{1}{z}X^{(3)} \right)\dotsm\left(\frac{1}{z}X^{(M-1)} \right)
				(v-\omega e_n)\right\|_2\\
				\leq C(1+r_0)^{M-1}|z|^{-(M-1)}n^{(M-1)/2}n^{-\delta/2}.
			\end{align*}
			Using the fact that $C(1+r_0)^{M-1}n^{-\delta/2}\to 0$ as $n\to \infty$, as well as the triangle inequality:
			\begin{align*}
				\left\|\left(\frac{1}{z}X^{(1)} \right)\left(\frac{1}{z}X^{(2)} \right)
				\left(\frac{1}{z}X^{(3)} \right)\dotsm\left(\frac{1}{z}X^{(M-1)} \right)
				v\right\|_2\geq cn^{(M-1)/2}|z|^{-(M-1)}.
			\end{align*}
			This in turn implies (\ref{normincomp}), as we can assume that if $v$ is such that (\ref{normincomp}) fails then $\|y_v\|_{2}\geq c(n^{1/2}/|z|)^{M-1}$ necessarily, and we can take $c < 1$.
			
			Combining (\ref{ssv2}) with (\ref{ssv1}) and (\ref{ssv1alt}), we see that for any choice of unit vector $v$ in the net $\Omega_\tau$: 
			\begin{align*}
				\mathbf{P}\left\{\left\| \tilde{X}^{(M)}y_v-z\tilde{I}_{n}v
				\right\|_{2} \leq C_0
				\min\left\{|z|(n^{1/2}/|z|)^{M}
				,n^{1/2-2\delta}\right\}
				\right\}
				\leq \exp(-c_0n).
			\end{align*}
			Taking the union bound over all $v$ in our net and using \eqref{eq:Omegasizebnd}, we obtain
			\begin{align} \label{eq:unionbndstar}
				\mathbf{P}\left\{
				\min_{v\in\Omega_\tau}
				\left\| \tilde{X}^{(M)}y_v-z\tilde{I}_{n}v
				\right\|_{2} \leq C_0
				\min\left\{|z|(n^{1/2}/|z|)^{M}
				,n^{1/2-2\delta}\right\}
				\right\}\\
				\leq O\left(\exp(-c_0n)\right). \nonumber
			\end{align}
			To finish the proof, we approximate an arbitrary vector in $V_\tau$ with a vector in our net. For any unit $v^\prime\in V_\tau$ and $v\in\Omega_\tau$, the triangle inequality and the definition of $y_v$ imply:
			\begin{align*}
				\left\| \tilde{X}^{(M)}y_v-z\tilde{I}_{n}v
				\right\|_{2}
				- \left\|\left(z\tilde{I}_{n}-\tilde{X}^{(M)}
				\prod_{h=1}^{M-1}\left(\frac{1}{z}X^{(h)}\right)\right)\right\|_{op}
				\left\|v-v^\prime\right\|_{2}\\
				\leq\left\|\left(z\tilde{I}_{n}-\tilde{X}^{(M)}\prod_{h=1}^{M-1}\left(\frac{1}{z}X^{(h)}\right)\right)v^\prime\right\|_2.
			\end{align*}
			By construction of our net we may take $\left\|v-v^\prime\right\|_{2}\leq n^{-1/2}$, and by the triangle inequality again we also have the operator norm estimate:
			\begin{align*}
				\left\|\left(z\tilde{I}_{n}-\tilde{X}^{(M)}\prod_{h=1}^{M-1}\left(\frac{1}{z}X^{(h)}\right)\right)\right\|_{op}
				\leq |z|\left(1+(1+r_0)^M(n^{1/2}/|z|)^{M}\right),
			\end{align*}
			which holds with exponential probability by \eqref{genopbound}.  In view of \eqref{eq:unionbndstar}, it remains to bound the following quantity from below:
			\begin{align*}
				\min\left\{|z|(n^{1/2}/|z|)^{M}
				,n^{1/2-2\delta}\right\}-
				\frac{|z|}{\sqrt{n}}\left(1+(1+r_0)^M(n^{1/2}/|z|)^{M}\right). 
			\end{align*}
			To accomplish this, we will inspect both instances of the minimum. On one hand, we have:
			\begin{align*}
				&|z|(n^{1/2}/|z|)^{M}
				-
				\frac{|z|}{\sqrt{n}}\left(1+(1+r_0)^M(n^{1/2}/|z|)^{M}\right)
				\\
				&= |z|\left(
				\left(1-\frac{(1+r_0)^M}{\sqrt{n}}\right)
				(n^{1/2}/|z|)^M-\frac{1}{\sqrt{n}}
				\right)\\
				&\geq cn^{1/2-\delta}n^{-M\delta}=
				cn^{1/2-(M+1)\delta}
			\end{align*}
			for $n$ sufficiently large.  
			On the other hand:
			\begin{align*}
				n^{1/2-2\delta}-&
				\frac{|z|}{\sqrt{n}}
				\left(1+(1+r_0)^M(n^{1/2}/|z|)^{M}\right)\\
				&\geq n^{1/2-2\delta}-cn^{(M-1)\delta}\\
				&\geq cn^{1/2-2\delta}.
			\end{align*}
			To obtain the last inequality, we have used the assumption $\delta<\frac{1}{4(M+1)}$. Since $M+1 \geq 2$, we can therefore conclude:
			\begin{align*}
				cn^{1/2-(M+1)\delta}
				\leq\left\|\left(z\tilde{I}_{n}-\tilde{X}^{(M)}\prod_{h=1}^{M-1}\left(\frac{1}{z}X^{(h)}\right)\right)v^\prime\right\|_2. 
			\end{align*}
			Since $v^\prime$ was an arbitrary unit vector in $V_\tau$ this concludes the proof.
		\end{proof}
		
		We now have a level of control which we can live with over the space $V_\tau$, and it remains to take similar control over vectors which lie largely but not entirely in $V_\tau$. The epsilon-net methods which we have been employing won't work in this case, as the portion of such a vector which does not lie in $V_\tau$ may lie in one of any number of directions. Instead, our next lemma follows from an approximation argument.
		
		\begin{lemma}
			\label{approxsmallsig}
			Let $\epsilon_0>0$ be a sufficiently small constant, and
			suppose that the estimate $n^{1/2-\epsilon_0/16M} \leq |z| \leq n^{1/2+\epsilon_0/16M}$ holds. Let $V[\epsilon_0]$ denote the set of all unit vectors $v$ which satisfy:
			\begin{align*}
				\left\|\prod_{h=1}^{M-1}\left(\frac{1}{z}X^{(h)}\right)v \right\|\leq n^{-\epsilon_0}.
			\end{align*}
			Then there exists a positive constant $C_{\epsilon_0}$ such that with overwhelming probability $X^{(1)},\ldots,X^{(M-1)}$ are such that the following probability bound holds with respect to the random matrix $\tilde{X}^{(M)}$:
			\begin{align*}
				\mathbf{P}\left\{
				\min_{v\in V[\epsilon_0]}
				\left\|\left(z\tilde{I}_{n}-\tilde{X}^{(M)}\prod_{h=1}^{M-1}\left(\frac{1}{z}X^{(h)}\right)\right)
				v\right\|_2 \leq C_{\epsilon_0}\right\}\\
				\leq O\left(\exp(-c_{\epsilon_0}n)\right).
			\end{align*}
			The choice of $\epsilon_0$ depends only on the atom distributions of the factor matrices.  Here $C_{\epsilon_0}$ depends on $\epsilon_0$ and $M$, and $c_{\epsilon_0}$ depends on $\epsilon_0$, the atom distributions of the factor matrices, and $M$.  
		\end{lemma}
		
		\begin{proof}
		Choose $0<\tau \leq \epsilon_0/4M$. Applying Lemma \ref{singspacelemma} with $\delta=\epsilon_0/8(M+1)$, it is overwhelmingly probable that $X^{(1)}$,\ldots,$X^{(M-1)}$ are such that with probability at least $1-C_0\exp(-c_0n)$ (with respect to $\tilde{X}^{(M)}$) the following event occurs:
		\begin{align}
			\label{siglem}
			\min_{\|v\|=1, v\in V_\tau}
			\left\|\left(z\tilde{I}_{n}
			-\tilde{X}^{(M)}
			\prod_{h=1}^{M-1}\left(\frac{1}{z}X^{(h)}\right)
			\right)v\right\|_2 \geq Cn^{1/2-\epsilon_0/8}.
		\end{align}
		We will obtain our result by essentially just observing that (\ref{siglem}) implies:
		\begin{align}
		\label{siglem_aim}
		\min_{v\in V[\epsilon_0]}
		\left\|\left(z\tilde{I}_{n}-\tilde{X}^{(M)}\prod_{h=1}^{M-1}\left(\frac{1}{z}X^{(h)}\right)\right)
		v\right\|_2 \geq C_{\epsilon_0}.
		\end{align}
		Indeed, let $v$ be any unit vector such that:
		\begin{align}
			\label{smallprod}
			\left\|\prod_{h=1}^{M-1}\left(\frac{1}{z}X^{(h)}\right)v\right\|_{2} \leq n^{-\epsilon_0}.
		\end{align}
		We can decompose $v=v^{(1)}+v^{(2)}$, where $v^{(1)}\in V_\tau$ and $v^{(2)}$ is orthogonal to $V_\tau$ (and therefore to $v^{(1)}$ as well). We would like to show that $v^{(2)}$ must be small, and that therefore our vector $v$ is well approximated by its projection $v^{(1)}$. Using the fact that $v^{(2)}$ is orthogonal to $V_\tau$ (and therefore lies entirely in the span of the singular vectors associated with the ``large" singular values), as well as our assumptions about the magnitude of the complex parameter $z$, we obtain:
		\begin{align*}
			n^{-(M-1)\epsilon_0/(2M)}\|v^{(2)}\|_{2}
			\leq
			\left\|\prod_{h=1}^{M-1}\left(\frac{1}{z}X^{(h)}\right)v^{(2)}
			\right\|_2. 
		\end{align*}
		Following Lemma \ref{vtauexist}, we see that $V_{\tau}$ is spanned by the singular vectors corresponding to the smallest singular values of the product $X^{(1)} \cdots X^{(M-1)}$.  Thus, $v^{(1)}$ can be expressed as a linear combination of singular vectors in $V_\tau$, and $v^{(2)}$ can be expressed as linear combinations of singular vectors from $V_\tau^\perp$.  Hence, the orthogonality of singular vectors implies:
		\begin{align*}
		\left\|\prod_{h=1}^{M-1}\left(\frac{1}{z}X^{(h)}\right)v\right \|_{2} 
		\geq n^{-(M-1)\epsilon_0/(2M)}\|v^{(2)}\|_{2}. 
		\end{align*}
		Substituting in (\ref{smallprod}):
		\begin{align*}
			n^{-(M-1)\epsilon_0/(2M)}\|v^{(2)}\|_{2}
			\leq \left\|\prod_{h=1}^{M-1}\left(\frac{1}{z}X^{(h)}\right)v\right\|_{2}\leq n^{-\epsilon_0}.
		\end{align*}
		And therefore:
		\begin{align}
		\label{approxtau}
		\|v^{(2)}\|_{2}\leq
		 n^{-\epsilon_0/2}.
		\end{align}
		Applying the triangle inequality and $\|v-v^{(1)}\|_{2}=\|v^{(2)}\|_{2}$:
		\begin{align*}
			\left\|\left(z\tilde{I}_{n}-\tilde{X}^{(M)}
			\prod_{h=1}^{M-1}\left(\frac{1}{z}X^{(h)}\right)
			\right)v^{(1)}\right\|_2- \left\|\left(z\tilde{I}_{n}-\tilde{X}^{(M)}\prod_{h=1}^{M-1}\left(\frac{1}{z}X^{(h)}\right)\right)
			\right\|_{op}
			\left\|v^{(2)}\right\|_{2}\\
			\leq
			\left\|\left(z\tilde{I}_{n}-\tilde{X}^{(M)}
			\prod_{h=1}^{M-1}
			\left(\frac{1}{z}X^{(h)}\right)\right)v\right\|_2.
		\end{align*}
		By \eqref{upboundnorm}, \eqref{siglem}, and \eqref{approxtau}, with exponential probability, 
		\begin{align*}
			\left\|\left(z\tilde{I}_{n}-\tilde{X}^{(M)}
			\prod_{h=1}^{M-1}\left(\frac{1}{z}X^{(h)}\right)
			\right)v\right\|_2 
			\geq C_{\epsilon_0}n^{1/2-\epsilon_0/8}
			- Cn^{1/2-\epsilon_0/4}.
		\end{align*}
		
		This is what we wanted to show, as (\ref{siglem_aim}) is now established. \end{proof}
		
		We are now in a position to prove:
		
		\begin{lemma}
			\label{conkboundlem1}
			Suppose that $n^{1/2-\epsilon_0/16M}\leq |z|\leq n^{1/2+\epsilon_0/16M}$, and let $u$ be any unit vector orthogonal to the subspace spanned by the first $Mn-1$ rows of $Y(z)$. If we write $u=(u^{(1)},\ldots,u^{(M)})$, where each $u^{(i)}$ has $n$ entries, it is overwhelmingly probable that $X^{(1)},\ldots,X^{(M-1)}$ are such that the following probability estimate holds (with respect to $\tilde{X}^{(M)}$):
			\begin{align}
				\label{conkbound}
				\mathbf{P}\left\{\|u^{(1)}\|_{2}\leq n^{-2\epsilon_0} \right\}
				\leq \exp(-c_{\epsilon_0} n).
			\end{align}
			Here, $\epsilon_0>0$ is sufficiently small constant (where the maximal legal choice of $\epsilon_0$ depends on the atom distributions of the factor matrices).
		\end{lemma}
		
		\begin{proof}
		
		By Lemma \ref{approxsmallsig} we may safely work on the event that:
		\begin{align}
			\label{nonzeq}
			\min_{v\in V[\epsilon_{0}]}
			\left\|\left(z\tilde{I}_{n}-\tilde{X}^{(M)}\prod_{h=1}^{M-1}\left(\frac{1}{z}X^{(h)}\right)\right)
			v\right\|_2 \geq C_{\epsilon_0}.
		\end{align}
		This is because the first $M-1$ factor matrices are, with overwhelming probability, such that this event is exponentially likely with respect to $\tilde{X}^{(M)}$.	Since $u$ is normal to the span of the first $nM-1$ rows of $Y(z)$ by assumption, the block structure of $Y(z)$ implies that for $1\leq j\leq M-1$:
		\begin{align*}
			u^{(j)}=\left(\prod_{h=j}^{M-1}\frac{1}{z}X^{(h)}\right)u^{(M)}.
		\end{align*}
		And in particular:
		\begin{align} \label{eq:u1intermsofum}
			u^{(1)}=\left(\prod_{h=1}^{M-1}\frac{1}{z}X^{(h)}\right)u^{(M)}.
		\end{align}
		This is because, by the block structure and the orthogonality of $u$, we can express each $u^{(i)}$ in terms of $u^{(i+1)}$. Since $\|z^{-1}X^{(h)}\|_{op}$ is at most $(1+r_0)n^{\epsilon_0/16M}$, with exponential probability, the fact that $u$ is a unit vector implies that the 2-norm of $u^{(M)}$ can not be too small. Specifically, we can assume that:
		\begin{align}
			\label{umbound}
		n^{-\epsilon_0}\leq \|u^{(M)}\|_{2}.
		\end{align}
		
		By \eqref{eq:u1intermsofum} and the orthogonality of $u$ to the last $n-1$ rows of $\tilde{Y}(z)$:
		\begin{align*}
			\left\|\left(z\tilde{I}_{n}-\tilde{X}^{(M)}\prod_{h=1}^{M-1}\left(\frac{1}{z}X^{(h)}\right)\right)
			u^{(M)}\right\|_2=0.
		\end{align*}
		Comparing with (\ref{nonzeq}), we have that $u^{(M)}/\|u^{(M)}\|_{2}$ cannot lie in $V[\epsilon_0]$, and this (along with (\ref{umbound}) and \eqref{eq:u1intermsofum}) implies:
		\begin{align*}
			\|u^{(1)}\|_{2}=\|u^{(M)}\|_{2}\left\|
			\left(\prod_{h=1}^{M-1}\frac{1}{z}X^{(h)}\right)
			\frac{u^{(M)}}{\|u^{(M)}\|_{2}} \right\|_2
			\geq n^{-2\epsilon_0}.
		\end{align*}
		This is what we wanted to show. \end{proof}

		We will also need the following lemma, which follows from a line of reasoning similar to the one used in establishing Lemma \ref{singvect}.
		
		\begin{lemma}
			\label{nocompcomponent}
			For every sufficiently small constant $\theta > 0$ and for every sufficiently small $t_0 > 0$ the following holds with $a = 1/\log n$, $b = \theta n^{-100(M-1)t_0}$, and $n^{1/2-t_0} \leq |z| \leq n^{1/2+t_0}$.  Let $u \in\mathbb{C}^{Mn}$ be a unit vector orthogonal to the first $Mn-1$ rows of the random matrix $Y(z)$. Write $u=(u^{(1)},u^{(2)},\ldots,u^{(M)})$ where $u^{(j)}\in\mathbb{C}^n$ for $1\leq j\leq M$. Then, with overwhelming probability, any choice of $u$ must be such that $u^{(1)}$ is not identically zero and the normalized vector $u^{(1)}/\|u^{(1)}\|_{2}$ does not lie in $\mathrm{Comp}(a,b)$.  Here, sufficient smallness of $\theta$ depends only on $M$, the constant $r_0$ from \eqref{genopbound}, and the distributions of the entries of the factor matrices.  
		\end{lemma}
		\begin{proof}
			Let $\theta, t_0 \in (0,1)$ be a sufficiently small constants to be chosen later.  
			The claim that $u^{(1)}\neq 0$ is trivial, as the linear structure of $Y(z)$ and orthogonality assumption on $u$ would then imply that $X^{(1)}$ is noninvertible, which can be ruled out with exponentially high probability \cite{V}; alternatively one could also use Lemma \ref{conkboundlem1}.  
			
			It therefore suffices to show that $u^{(1)}/\|u^{(1)}\|_{2}$ is not compressible.  We will assume $M > 1$; in fact, the $M=1$ case can be deduced from Lemma \ref{complemma2}.  
			Let $v=(v^{(1)},\ldots,v^{(M)})$ be some nonzero vector in $\mathbb{C}^{Mn}$ with each $v^{(i)}\in\mathbb{C}^n$ and $\|v^{(1)}\|_{2} = 1$.  We will use a net argument to show that, with overwhelming probability, $v$ cannot be orthogonal to the first $Mn-1$ rows of $Y(z)$ if $v^{(1)} \in \mathrm{Comp}(a,b)$.  
			
			If $v$ is orthogonal to the first $Mn-1$ rows of $Y(z)$, the block structure which results from linearization implies that 
			\begin{equation} \label{eq:recident1}
				v^{(1)} = \frac{1}{z^{M-1}} X^{(1)} \cdots X^{(M-1)} v^{(M)} 
			\end{equation}
			and
			\begin{equation} \label{eq:recident2}
				\tilde{I} v^{(M)} = \frac{1}{z} \tilde{X}^{(M)} v^{(1)}, 
			\end{equation}
			where $\tilde{I}$ is the $(n-1) \times n$ matrix formed from the $n \times n$ identity matrix by removing the last row; these identities also appeared in the proof of Lemma \ref{conkboundlem1}.  In particular, since $\|v^{(1)} \|_2 = 1$, with exponentially high probability, we use \eqref{genopbound} and \eqref{eq:recident1} to deduce that
			\[ 1 \leq \left( \frac{ (1 + r_0) n^{1/2} }{ n^{1/2 - t_0} } \right)^{M-1} \|v^{(M)} \|_2 \]
			and hence
			\begin{equation} \label{eq:bndCr0}
				\|v^{(M)} \|_2 \geq C_{r_0} n^{-(M-1) t_0} 
			\end{equation}
			for some constant $C_{r_0} > 0$ which depends only on $r_0$ and $M$.  
			
			Inductively repeating the argument from the proof of Lemma \ref{singvect} (using  Lemma \ref{smallballprob} to control the size of each coordinate and then applying the Chernoff bound), we have
			\[ \mathbf{P} \left\{ \| z^{M-1} y - X^{(1)} \cdots X^{(M-1)} x \|_2 \leq c_1 n^{(M-1)/2 - 3(M-1) t_0} \right\} = O(\exp(-cn) ) \]
			for any fixed vectors $x,y$ with $\|x\|_2 \geq \frac{C_{r_0}}{2} n^{-3(M-1) t_0}$.  
						
			Let $a$ and $b$ be as in the statement of the lemma, and let $\Omega_{a,b}$ be a $3b$-net of $\mathrm{Comp}(a,b)$, the set of compressible unit vectors in $\mathbb{C}^n$.  By Lemma \ref{cordbound}, $\Omega_{a,b}$ can be chosen so that
			\[ |\Omega_{a,b}| \leq \exp \left( 2an \log \left( \frac{e C}{ab} \right) \right) = \exp \left( o(n) + 200 n (M-1) t_0 \right). \]
			Let $\mathcal{N}$ be a $b$-net of $\{z \in \mathbb{C} : |z| \leq n^{2} \}$.  A simple volume argument shows that $\mathcal{N}$ can be chosen so that 
			\[ l := |\mathcal{N}| \leq O(n^{O_M(1)}). \]
			Let $\omega_1, \ldots, \omega_l$ be an enumeration of the elements in $\mathcal{N}$.  
			To each $y \in \Omega_{a,b}$, we associate the vectors $x_{y,1}, \ldots, x_{y,l}$ such that
			\[ \tilde{I} x_{y,k} = \frac{1}{z} \tilde{X}^{(M)} y \]
			for each $1 \leq k \leq l$, and the last coordinate of $x_{y,k}$ is given by $\omega_k$.  The vectors $x_{y,1}, \ldots, x_{y,l}$ are random, but only depend on $\tilde{X}^{(M)}$.  In particular, these vectors are independent of $X^{(1)}, \ldots, X^{(M-1)}$.  For notational simplicity, define 
			\[\Omega_{a,b}' = \left\{ (y, x_{y,k}) : y \in \Omega_{a,b}, 1 \leq k \leq l, \|x_{y,k} \|_2 \geq \frac{C_{r_0}}{2} n^{-3(M-1)t_0} \right\}, \]  
			so 
			\[ |\Omega_{a,b}'| \ll \exp \left( o(n) + 200 n (M-1) t_0 \right). \]
			Taking $t_0$ sufficiently small and applying the union bound, we conclude that
			\begin{equation} \label{eq:unionbnd'}
				\mathbf{P} \left\{ \min_{(y,x) \in \Omega_{a,b}'} \| z^{M-1} y - X^{(1)} \cdots X^{(M-1)} x \|_2 \leq c_1 n^{(M-1)/2 - 3(M-1) t_0} \right\} = O(\exp(-c'n)). 
			\end{equation}
			Here, we have exploited the fact that $\Omega_{a,b}'$ depends only on $\tilde{X}^{(M)}$ and is independent of $X^{(1)}, \ldots, X^{(M-1)}$, while the probability above is only in terms of $X^{(1)}, \ldots, X^{(M-1)}$.  In particular, the probability bound above holds uniformly for any realization of $\tilde{X}^{(M)}$.  
			
			Now suppose $v = (v^{(1)}, \ldots, v^{(M)}) \in \mathbb{C}^{Mn}$ is normal to the first $Mn-1$ rows of $Y(z)$ with $\|v^{(1)} \|_2 = 1$, $v^{(1)} \in \mathrm{Comp}(a,b)$, and $C_{r_0} n^{-(M-1)t_0} \leq \|v^{(M)} \|_2 \leq n$.  Then clearly $v^{(1)}$ and $v^{(M)}$ must satisfy \eqref{eq:recident1} and \eqref{eq:recident2}.  In addition, there exists $y' \in \Omega_{a,b}$ such that $\|v^{(1)} - y'\|_2 \leq 3b$.  By \eqref{eq:recident2}, with exponentially high probability and for all $1 \leq k \leq l$, 
			\[ \| \tilde{I} v^{(M)} - \tilde{I} x_{y', k} \|_2 \leq 3 (1+r_0) n^{t_0} b. \]
			By the assumption that $\|v^{(M)} \|_2 \leq n$, the last coordinate of $v^{(M)}$ cannot be larger than $n$.  Thus, there exists $k$ such that $\|v^{(M)} - x_{y',k} \|_2 \leq 5 ( 1 + r_0) n^{ t_0} b$.  Taking $\theta$ sufficiently small (in particular, choosing $\sqrt{\theta}$ small enough), we obtain
			\[ \|v^{(M)} - x_{y',k} \|_2 < \frac{C_{r_0}}{2} \sqrt{\theta} n^{-99 (M-1) t_0}. \]
			In particular, by the lower bound assumption on $\| v^{(M)} \|_2$, this implies that $\|x_{y',k} \|_2 \geq \frac{C_{r_0}}{2} n^{-3 (M-1) t_0}$.  Hence, $(y', x_{y',k}) \in \Omega_{a,b}'$.  Applying \eqref{eq:recident1}, \eqref{genopbound}, and taking $\theta$ sufficiently small, we conclude that
			\begin{align*}
				\| z^{M-1} y' - X^{(1)} \cdots X^{(M-1)} x_{y',k} \|_2 &\leq |z|^{M-1} 3 b + [(1 + r_0) \sqrt{n}]^{M-1} \frac{C_{r_0}}{2} \sqrt{\theta} n^{-99 (M-1) t_0} \\
				&< c_1 n^{(M-1)/2 - 3 (M-1) t_0}
			\end{align*}
			with exponentially high probability.  Comparing to \eqref{eq:unionbnd'}, we obtain
			\begin{align*}
				\mathbf{P} &\left\{ \exists v \text{ orthogonal to rows of } \tilde{Y}(z) \text { with } v^{(1)} \in \mathrm{Comp}(a,b), C_{r_0} n^{-(M-1)t_0} \leq \|v^{(M)} \|_2 \leq n \right\} \\
				&\qquad\qquad= O(\exp(-c''n )). 
			\end{align*}
			
			To complete the proof, it remains to show, with overwhelming probability, that every vector $v$ orthogonal to the first $Mn-1$ rows of $Y(z)$ satisfies $C_{r_0} n^{-(M-1)t_0} \leq \|v^{(M)} \|_2 \leq n$.  Indeed, the lower bound follows, with exponentially high probability, due to \eqref{eq:bndCr0}.  We now prove the upper bound holds with overwhelming probability.  Indeed, in view of \eqref{eq:recident1} and \eqref{eq:recident2}, $v^{(M)}$ must satisfy
			\[ \left( \tilde{I} - \frac{1}{z^{M}} \tilde{X}^{(M)} X^{(1)} \cdots X^{(M-1)} \right) \frac{v^{(M)}}{\|v^{(M)} \|_2} = 0. \]
			(Note that $v^{(M)}$ cannot be zero by \eqref{eq:recident1} since $v^{(1)}$ is assumed to be a unit vector.)  
			Taking $t_0$ sufficiently small and applying Lemma \ref{approxsmallsig}, we find that, with overwhelming probability, $v^{(M)} / \|v^{(M)} \|_2 \not\in V[16 M t_0]$.  Hence, by definition of $V[16 M t_0]$ and \eqref{eq:recident1}, 
			\[ \left\| \frac{1}{z^{M-1}} X^{(1)} \cdots X^{(M-1)} \frac{v^{(M)} }{\|v^{(M)} \|_2} \right\|_2 = \frac{ \|v^{(1)} \|_2}{ \| v^{(M)} \|_2 } \geq n^{-16 M t_0}, \]
			which by rearranging (and since $\|v^{(1)} \|_2 = 1$) yields
			\[ \| v^{(M)} \|_2 \leq n^{16 M t_0} \leq n \]
			for $t_0$ sufficiently small.  
			The proof of the lemma is complete.  
			\end{proof}

		\subsection{The Incompressible Case, Part II}
		Let $\theta > 0$ be a sufficiently small constant satisfying Lemma \ref{nocompcomponent}.  We now prove the following result for incompressible vectors.  
		\begin{lemma}
			\label{decomplem}
			For any sufficiently small constant $A >0$, there exists a constant $t_0 > 0$ such that the following holds for $a = 1/\log n$ and $b = \theta n^{-100(M-1)t_0}$.  
			Suppose that $z$ lies in the following annulus:
			\begin{align}
				n^{1/2-t_0}\leq |z|\leq n^{1/2+t_0}.
			\end{align} 
			Then the following probability bound holds:
			\begin{align*}
				\mathbf{P}\left\{\min_{v\in \emph{Incomp}(a,b)}\|Y(z)v\|_{2}\leq
				c n^{-1/2-A}
				\right\}
				\leq O\left((\log n)n^{-KA}\right).
			\end{align*}
			Here $c$ and $K$ are positive constants depending only on the atom distributions of the factor matrices and on $M$ (as does in the implied constant on the right hand side of the probability estimate).
		\end{lemma}
			
		The argument presented here will in large part follow along the same lines as the argument in \cite{TV3}, and will be organized around the following result of Rudelson and Vershynin \cite{RVc}:
		
		\begin{lemma}
			\label{keyincomp}
			Let $\emph{dist}_k$ denote the distance between the $k$-th row of $Y(z)$ and the hypersurface spanned by the other $Mn-1$ rows. Then the following estimate holds for any constants $1 > a,b>0$:
			\begin{align}
				\label{rudverdist}
				\mathbf{P}\left\{\min_{v\in \emph{Incomp}(a,b)}\|Y(z)v\|_{2}\leq
				\epsilon' b (Mn)^{-1/2}
				\right\}\leq \frac{1}{aMn}\sum_{k=1}^{Mn} \mathbf{P}\left\{
				\emph{dist}_k\leq \epsilon'
				\right\},
			\end{align}
			where $\epsilon'>0$ is any arbitrary positive constant.
		\end{lemma}
		
		\begin{proof}[Proof of Lemma \ref{decomplem}]
		Our task is now  to estimate $\mbox{dist}_k$, which can be formulated as the magnitude of an inner product. It is sufficient to establish the result for the very last row of $Y(z)$, the changes required to establish the result for the other rows being purely notational. Make the following definitions, for constants $c_1,c_2>0$ to be determined later:
		\begin{align*}
			\epsilon={c_1}{b}n^{-A}, \qquad
			\rho=c_2n^{-A}.
		\end{align*}
		Define also $S_{\epsilon,\rho}$ to be the set of unit vectors $v\in\mathbb{C}^n$ satisfying the following inequality (where $X$ is a random vector with the same distribution as a row of $X^{(M)}$):
		\begin{align}
			\sup_{\xi\in\mathbb{C}} \left[\mathbf{P}\left(
			|X \cdot v -\xi|\leq \epsilon
			\right)\right]\geq \rho.
		\end{align}
		By Lemma 6.7 in \cite{TV3}, we have that for any $t_0, \theta > 0$, for $n$ large enough, and $A < 1/2$ (and choosing $c_1$ and $c_2$ appropriately):
		\begin{align*}
			S_{\epsilon,\rho}\subset \mbox{Comp}\left(
			O\left(\frac{1}{n\rho^2}\right), O\left(\frac{\epsilon}{\rho}\right)
			\right)
			\subset \mbox{Comp}\left(
			a,b
			\right).
		\end{align*}
		Let $\Phi$ be an $Mn$-dimensional vector $\Phi=(\Phi^{(1)},\ldots,\Phi^{(M)})$ (not necessarily of unit length) orthogonal to first $Mn-1$ rows of $Y(z)$, with each $\Phi^{(i)}\in\mathbb{C}^n$ and with $\|\Phi^{(1)}\|_{2}=1$.  Notice that $\Phi$ can be chosen to depend only on the first $Mn-1$ rows of $Y(z)$ and is independent of the last row of $X^{(M)}$.  By Lemma \ref{nocompcomponent}, with overwhelming probability (with respect to $X^{(1)}, \ldots, X^{(M-1)}$), we have that such a normal vector exists and $\Phi^{(1)}$ is not in $S_{\epsilon,\rho}$.  This implies that 
		\begin{align}
			\mathbf{P}\left\{
			\left\|X^{(M)}_{n}\Phi^{(1)}-z\Phi^{(M)} \right\|_{2}\leq \epsilon \right\}\leq \rho.  
		\end{align}
		Here, since $\Phi$ is orthogonal to the first $Mn-1$ rows of $Y(z)$, $\left\|X^{(M)}_{n}\Phi^{(1)}-z\Phi^{(M)} \right\|_{2}$ is simply the magnitude of the dot product of $\Phi$ with the last row of $Y(z)$.  For $A$ and $t_0$ sufficiently small, Lemma \ref{conkboundlem1}, with overwhelming probability, guarantees that 
	 $\left\|\left(\Phi^{(1)}/\|\Phi\|_{2}
	 \right)\right\|_{2}\geq n^{-A} $.  In other words, (since $\| \Phi^{(1)} \|_2 = 1$), $\| \Phi \|_2 \leq n^{A}$.  Rescaling, we obtain 
		\begin{align}
		\mathbf{P}\left\{
		\mbox{dist}_{Mn} \leq \epsilon n^{-A}
		\right\} \leq \mathbf{P} \left\{ \left\|X^{(M)}_{n}\Phi^{(1)}-z\Phi^{(M)} \right\|_{2} \leq \epsilon n^{-A} \| \Phi \|_2 \right\} \leq O\left(\rho \right). 
		\end{align}
		Applying Lemma \ref{keyincomp}, we conclude that
		\begin{align*}
			\mathbf{P}\left\{\min_{v\in \mathrm{Incomp}(a,b)}\|Y(z)v\|_{2}\leq
			c_1' \theta^2 n^{-2A - 1/2 - 200(M-1) t_0} 
			\right\}\leq O((\log n) \rho),
		\end{align*}
		where $c_1' = c_1 / \sqrt{M}$.  Taking $t_0$ sufficiently small (in terms of $A$ and $M$), we deduce that
		\[ \mathbf{P}\left\{\min_{v\in \mathrm{Incomp}(a,b)}\|Y(z)v\|_{2}\leq
			c_1' \theta^2 n^{-3A - 1/2 }
			\right\}\leq O((\log n) \rho).\]
		Since this is true for any sufficiently small positive constant $A$, replacing $A$ by $A/3$ completes the proof of Lemma \ref{decomplem}.  
		\end{proof}

		The proof of Theorem \ref{linsmallsig} can now be completed by combining Lemma \ref{complemma} with Lemma \ref{decomplem} and bounding the lower order terms (such as bounding $\log n$ above by $n^{KA/2}$ for $n$ sufficiently large).   
		We now turn to the proof of Corollary \ref{smallsigvalcor}.  From the block structure of $Y(z)$, one can compute each $n \times n$ block of the inverse $Y(z)^{-1}$.  Indeed, the top-left $n \times n$ minor of $Y(z)^{-1}$ is given by 
		\begin{align*}
		\left(\frac{1}{z^{M-1}}X_1\dotsm X_M-z\right)^{-1}.
		\end{align*}
		By Theorem \ref{linsmallsig}:
		\begin{align*}
		\mathbf{P}\left\{ 
		\|Y(z)^{-1}\|_{op} \geq n^{1/2+A}
		\right\}
		\leq C n^{-KA},
		\end{align*}
		and consequently (using the fact that the operator norm of a matrix bounds the operator norm of any sub-matrix):
		\begin{align*}
			\mathbf{P}\left\{ 
			 \left\| 
			\left(\frac{1}{z^{M-1}}X_1\dotsm X_M-z\right)^{-1}
			\right\|_{op} \geq n^{1/2+A}
			\right\}
			\leq C n^{-KA}.
		\end{align*}
		We conclude that with probability at least $1-C n^{-KA}$ the smallest singular value of the translated product matrix $X_1\dotsm X_M-z^M$ is no smaller than $|z|^{M-1}n^{-(1/2+A)}$. This establishes Corollary \ref{smallsigvalcor}.

		\section{Linear Statistics of Product Matrices}
		\label{linstat_sec}
		
		The main result of this section will be a central limit theorem for linear statistics of products of independent complex Ginibre matrices.  Recall that $\mathbb{U}$ represents the unit disk in the complex plane centered at the origin.  
		
		\begin{theorem}
			\label{cltforginibre}
			Let $f:\mathbb{C}\to\mathbb{R}$ be a function with continuous partial derivatives and at most polynomial growth at infinity. Additionally, let $\lambda_1,\ldots,\lambda_n$ denote the eigenvalues of the matrix product $n^{-M/2}X^{(1)}\dotsm X^{(M)}$, where each $X^{(i)}$ is an independent $n \times n$ complex Ginibre matrix, and let $N_n[f]$ denote the associated  centered linear statistic: $N_n[f]=\sum_{j=1}^n f(\lambda_j)-\mathbf{E}\left[\sum_{j=1}^n f(\lambda_j) \right]$.
			Then $N_n[f]$ converges in distribution to the mean zero normal distribution with the following limiting variance:
			\begin{align}
				\frac{1}{4\pi}\int_{\mathbb{U}}\left|\nabla f(z)\right| d^2z
				+\frac{1}{2}\|f\|^2_{H^{1/2}(\mathbb{\partial U})}.
			\end{align}
			\end{theorem}
		
		We will follow the approach developed by Rider and Vir\'ag \cite{RV} to prove a similar result for the Ginibre ensemble, which will come down to using rotary flow combinatorial machinery to control the cumulants of the statistic in the case of polynomial test functions, and then producing a variance bound to extend the result to the case of more general test functions. The key property which we are exploiting here is the fact that the Ginibre product matrix, like the Ginibre ensemble itself, is a rotationally invariant determinantal point process.  
		
		Theorem \ref{cltforginibre} will provide us with a base case to which we will then apply four moment universality. To finish off this section, we will also apply this same method to products of independent, truncated unitary matrices. 
		
		\subsection{On Combinatorics and Cumulants}
		
		We first introduce some preliminary combinatorial machinery which we will need to establish central limit theorems for linear statistics of product matrices. The development in this section is taken directly from \cite{RV}, on whose work our approach is modeled. The cumulants $\mathcal{C}_k(\xi)$ of a real-valued random variable $\xi$ are defined by the following formula:
		\begin{align*}
			\log\mathbf{E}[e^{\sqrt{-1}t \xi}]=\sum_{k=1}^\infty \frac{(\sqrt{-1}t)^k}{k!}\mathcal{C}_k(\xi).
		\end{align*}
		
		Define $X(g)=\sum g(z_k)$, where $g$ is some real-valued test function and $z_k$ are random points distributed according to the determinantal point process with kernel $K=K(-,-)$ and associated measure $\mu$ (where $K$ is a self adjoint trace class integral operator on $L^2(\mathbb{C},d\mu)$).  It is well-known that the $k$-th cumulant of $X(g)$ can be expressed according to the following Costin-Lebowitz formula \cite{CL} (see also \cite{Sosh, Sosh2}):
		\begin{align*}
			\mathcal{C}_k(X(g))&=\sum_{m=1}^k \frac{(-1)^m}{m}\sum_{k_1+\cdots +k_m=k} \frac{k!}{k_1!\dotsm k_m!}\\
			&\qquad\times \int_{\mathbf{C}^k}\left(\prod_{\mathit{l}=1}^m g(z_\mathit{l})^{k_\mathit{l}}\right)
			K(z_1,\overline{z}_2)\dotsm K(z_m,\overline{z}_1)d\mu(z_1)\dotsm
			d\mu(z_k).
		\end{align*} 
		
		To take advantage of the rotational invariance, we will need to cast this expression in a more combinatorial light. Following Rider and Vir\'ag, we define:
		\begin{align} \label{def:Phi_m}
			\Phi_{m}(f_1,\ldots,f_m)=
			\int_{\mathbb{C}^m}f_1(z_1)\dotsm f_m(z_m)
			K(z_1,\overline{z}_2)\dotsm K(z_m,\overline{z}_1)d\mu(z_1)\dotsm
			d\mu(z_m).
		\end{align}
		For a function $\sigma: [k] \to [m]$, and $f \in \mathcal{G}^k$, where $\mathcal{G}$ is an algebra of real-valued functions, define $\sigma f \in \mathcal{G}^m$ by
		\begin{equation} \label{def:sigma}
			(\sigma f)_j(z) = \prod_{i : \sigma(i) = j} f_i(z). 
		\end{equation}
		For $f \in \mathcal{G}$, we will occasionally abuse notation and write $\sigma f$ to denote $\sigma (f, f, \ldots, f)$.  
		
		This allows us to express the $k$-th cumulant as a sum over surjections:
		\begin{align*}
			\mathcal{C}_k(X(g))=\sum_{m=1}^k \frac{(-1)^m}{m}\sum_{\sigma: [k]\twoheadrightarrow [m]}
			\Phi_{m}\left( \sigma g \right).
		\end{align*}
		Here, $\sigma:[k]\twoheadrightarrow[m]$ represents a surjection from the set of $k$ objects to the set of $m$ objects. If we assume that $g$ is a polynomial in $z$ and $\bar{z}$, we see by multi-linearity that the $k$-th cumulant is a linear combination of terms of the form:
		\begin{align}
			\mathbb{Y}_{a,b}=
			\sum_{m=1}^k \frac{(-1)^m}{m}\sum_{\sigma: [k]\twoheadrightarrow [m]}
			\Phi_{m}\left( \sigma (z^{a_1} \bar{z}^{b_1}, \ldots, z^{a_k} \bar{z}^{b_k} ) \right).
		\end{align} 
		If we can show that all such summands vanish (either identically or in the limit as $n\to\infty$) for all cumulants after the second, we will be able to conclude that $X(g)$ converges in distribution to a Gaussian random variable. These formulas also provide us with a means to compute the limiting variance.
		
		\subsection{Products of Complex Ginibre Matrices}
		
		In this section, we will quickly collect some basic facts about the products of Ginibre random matrices and sketch our argument. The eigenvalues of the product of $M$ complex independent Ginibre matrices (with entries having variance $1/n$) constitute a determinantal point process with the following kernel \cite{AB2}: 
		\begin{align}
			\label{kerneleq}
			K_n(z,\bar{u}) = \tilde{K}_n(n^{M/2}z,n^{M/2}\bar{u})=
			\sum_{\iota=0}^{n-1}\frac{(z\bar{u})^\iota}
			{h_\iota}.
		\end{align}
		The normalization constants $h_\iota$ are given by:
		\begin{align*}
			h_\iota=\pi n^{-M\iota} \Gamma(\iota+1)^M.
		\end{align*}
		Here, $\Gamma$ denotes the standard gamma function. 
		
		The associated measures $\mu_n$ are given by the products of the Lebesgue measure on the complex plane with the following weight functions:
		\begin{align*}
			w_n(n^{M/2}z)=n^M\pi^{1-M} \int_{\mathbb{C}^M} \delta^2(n^{M/2}z-\xi_M\dotsm\xi_1)
			\prod_{j=1}^M e^{-|\xi_j|^2}d^2\xi_j.
		\end{align*} 
		Here, $\delta^2$ represents the standard delta function on $\mathbb{C}$; the factor of $n^M$ in front is due to the scaling of $z$ in the delta function.  Notice that this measure is rotationally invariant. The weight functions can also be expressed in terms of the Meijer G-function -- see \cite{AB2} for more details. We will need to be able to integrate polynomials against these measures. Applying Fubini's theorem, substituting $n^{M/2}z$, factoring the resulting product of integrals and switching to polar coordinates, we find:
		\begin{align} \label{eq:bimomentwn}
			\int_{\mathbb{C}}|z|^{2\iota}w_n(n^{M/2} z)d^2z
			&=\frac{\pi^{1-M}}{n^{M\iota}}\prod_{j=1}^M
			\left(\int_{\mathbb{C}}
			|\xi_j|^{2\iota} e^{-|\xi_j|^2}d^2\xi_j\right) \\
			&=\frac{\pi}{n^{M\iota}}\Gamma(\iota+1)^M. \nonumber
		\end{align}
		
		We turn now to establishing the asymptotic smallness of higher cumulants. Fixing some arbitrary nonnegative integers $\alpha_1,\ldots,\alpha_m$ and $\beta_1,\ldots,\beta_m$, we can plug the above into the general formula for $\Phi_{m}$ to obtain:
		\begin{align*}
			\Phi_{m}\left(z^{\alpha_1}\bar{z}^{\beta_1},\ldots,
			z^{\alpha_m}\bar{z}^{\beta_m}\right) &=
			\frac{1}{\pi^m}
			\sum_{q_1,\ldots,q_m=0}^{n-1}
			\int_{\mathbb{C}^m}\left(\prod_{\mathit{l}=1}^m 
			z_l^{\alpha_{\mathit{l}}}\bar{z}_l^{
				\beta_{\mathit{l}}}	\right)
			\\
			&\qquad \times
			\left[\prod_{j=1}^m \frac{z_j^{q_j}\bar{z}_{j+1}^{q_j}}
			{n^{-Mq_j} [\Gamma(q_j+1)]^M}
			w(n^{M/2}z_{j})\right] d^2z_1\dotsm d^2z_m,
		\end{align*}
		where we use the convention that $z_{m+1} = z_1$.  
		
		Introduce the notation:
			\begin{align*}
				\eta_j=\sum_{i=1}^j \left(\beta_i-\alpha_i \right).
			\end{align*}
		We will also use $\eta_{max}$ and $\eta_{min}$ to refer to the maximal and minimal values of $\eta_j$, respectively. The following combinatorial lemma, whose proof will be furnished shortly, will be crucial for us.
		\begin{lemma}
		\label{comblemma}
		If the quantity $s=\alpha_1+\cdots +\alpha_m$ is equal to the quantity $\beta_1+\cdots +\beta_m$ for some choice of nonnegative integer vectors $\alpha$ and $\beta$, then the following approximation holds:
		\begin{align*}
			\frac{1}{\pi^m}\int_{\mathbb{C}^m}
			\left(z_1^{\alpha_1}\dotsm z_m^{\alpha_m}\right)
			\overline{\left(z_1^{\beta_1}\dotsm z_m^{\beta_m}\right)}
			\sum_{q_1,\ldots,q_m=0}^{n-1}
			\prod_{j=1}^m \frac{z_j^{q_j}\bar{z}_{j+1}^{q_j}}
			{n^{-Mq_j}\Gamma(q_j+1)^M}
			w_n(n^{M/2}z_j) d^2z_j\\
			=
			\frac{n}{Ms+1}
			-(1+\eta_{max})
			+\frac{1}{2}+\frac{1}{Ms}\sum_{j=1}^m
			\left(\alpha_j\eta_j
			+\frac{\alpha_j(\alpha_j+1)}{2}\right)
			+O\left(n^{-1} \right).
		\end{align*}
		If on the other hand $\alpha_1+\cdots +\alpha_m$ differs from $\beta_1+\cdots +\beta_m$, then the integral in question vanishes identically (by rotational invariance).
		\end{lemma}
				
		Setting $\alpha_j = \sum_{i : \sigma(i) = j} a_i$ for each $j$ and some selection of nonnegative integers $a_1,\ldots,a_m$, and $\beta_j= \sum_{i : \sigma(i) = j} b_i$ for each $j$ and some selection of nonnegative integers $b_1,\ldots,b_m$, the non-vanishing condition is equivalent to assuming $a_1+\cdots+a_k$ is equal to $b_1+\cdots+b_k$. We then obtain a tractable formula for $\mathbb{Y}_{a,b}$ by way of Lemma \ref{comblemma}:
		\begin{align*}
			\sum_{m=1}^k \frac{(-1)^m}{m}&\sum_{\sigma: [k]\twoheadrightarrow [m]}
			\Phi_{m}\left( \sigma (z^{a_1} \bar{z}^{b_1}, \ldots, z^{a_k} \bar{z}^{b_k} ) \right)\\
			&=\sum_{m=1}^k\frac{(-1)^m}{m}\sum_{\sigma:[k]\twoheadrightarrow[m]}
			\Bigg[\frac{n}{Ms+1}
			-(1+\eta_{max})
			+\frac{1}{2}\\ 
			&\qquad +\frac{1}{Ms}\sum_{j=1}^m
			\left(\alpha_j\eta_j
			+\frac{\alpha_j(\alpha_j+1)}{2}\right)
			\Bigg]
			+O(1/n).
		\end{align*}  
		(Note that $\alpha_j$ and $\beta_j$ depend on $\sigma$, and so $\eta_j$ and $\eta_{max}$ also depend on $\sigma$.)
		
		This summation is handled by the combinatorial machinery put together by Rider and Vir\'ag. By Lemma 5.1 in \cite{RV}, for $k\geq 3$:
		\begin{align}
			\label{van1}
			\sum_{m=1}^k\frac{(-1)^m}{m}\sum_{\sigma:[k]\twoheadrightarrow[m]}
			\left(\frac{n}{Ms+1}
			-1/2
			\right)=0.
		\end{align}  
		By Lemma 5.5 in \cite{RV} (again assuming that $k\geq 3$):
		\begin{align}
			\label{van2}
			\sum_{m=1}^k\frac{(-1)^m}{m}\sum_{\sigma:[k]\twoheadrightarrow[m]}
			\eta_{max}=0.
		\end{align}  
		On the other hand, by Lemmas 5.3 and 5.4 in \cite{RV}, for $k\geq 3$:
		\begin{align}
			\label{van3}
			\sum_{m=1}^k\frac{(-1)^m}{m}\sum_{\sigma:[k]\twoheadrightarrow[m]}
			\Bigg[\frac{1}{Ms}\sum_{j=1}^m
			\left(\alpha_j\eta_j
			+\frac{\alpha_j(\alpha_j+1)}{2}\right)
			\Bigg]=0.
		\end{align}  
		Putting together (\ref{van1}), (\ref{van2}), and (\ref{van3}), we can conclude that $\mathbb{Y}_{a,b}\to 0$ as $n \to \infty$, and therefore as discussed above we have that all the cumulants after the second vanish asymptotically.
		
		The linear statistic of a real valued polynomial test function therefore converges in distribution to a Gaussian, and all that remains is to compute the limiting variance.
		
		We now provide the proof of Lemma \ref{comblemma}.
		\begin{proof}[Proof of Lemma \ref{comblemma}]
		
		First, consider the following expression:
		\begin{align*}
			\frac{1}{\pi^m}\sum_{q_1,\ldots,q_m=0}^{n-1} \int_{\mathbb{C}^m}\prod_{j=1}^m
			\left[
			n^{Mq_j}\Gamma(q_j+1)^{-M}(z_j)^{\alpha_j}\overline{(z_j)^{\beta_j}} 
			(z_j\overline{z_{j+1}})^{q_j}\right]\\
			\times w_n(n^{M/2}z_1)d^2z_1\ldots w_n(n^{M/2}z_m)d^2z_m.
		\end{align*}
		By rotational invariance of the weight function, the integrals which we are summing vanish unless the following condition (which may be viewed as a rotary flow condition -- see the discussion in \cite{RV}) is satisfied for each index $j$:
		\begin{align*}
			\alpha_j+q_j=\beta_j+q_{j-1}.
		\end{align*}
		Following \cite{RV}, we set:
		\begin{align*}
			\gamma_j&=\beta_j-\alpha_j\\
			\eta_j&=\gamma_1+\cdots +\gamma_j.
		\end{align*}
		Under this new notation, the rotary flow condition which we must meet to have a nonvanishing integral summand is expressed as:
		\begin{align*}
			q_j=q_m+\eta_j\\
			\sum_{j=1}^m \alpha_j =\sum_{j=1}^m \beta_j.
		\end{align*}
		Since the $q_j$ must satisfy $0\leq q_j \leq n-1$ by definition (from the equation for the correlation kernel -- see (\ref{kerneleq})), we have that $q_m$ must satisfy the bounds $-\eta_{min}\leq q_m \leq n-1-\eta_{max}$. Setting $\mathit{l}=q_m$ (which specifics the other values of $q_j$, by the discussion above) and using the fact that we have $2\mathit{l}+2\eta_j-\gamma_j+\alpha_j+\beta_j=
		2(\mathit{l} + \eta_j+\alpha_j)$, our expression is now:
		\begin{align*}
			\frac{1}{\pi^m}\sum_{\mathit{l}=-\eta_{min}}^{n-1-\eta_{max}}
			\prod_{j=1}^m n^{M(\mathit{l}+\eta_j)}\Gamma(\mathit{l} + \eta_j +1)^{-M}
			\times\left(\int_{\mathbb{C}}
			|z|^{2(\mathit{l} + \eta_j+\alpha_j)}
			w_n({n}^{M/2} z)d^2z\right)^M.
		\end{align*}
		
		Plugging in the moments of the weight function (see \eqref{eq:bimomentwn}), this last display can be rewritten as:
		\begin{align*}
			\sum_{\mathit{l}=-\eta_{min}}^{n-1-\eta_{max}}
			\prod_{j=1}^m
			\frac{n^{M(\mathit{l}+\eta_j)}}{n^{M(\mathit{l}+\eta_j+\alpha_j)}}
			\left[\frac{ \Gamma(\alpha_j+(\mathit{l}+\eta_j+1))}
			{\Gamma(\mathit{l}+\eta_j+1)}\right]^M.
		\end{align*}
		We will need to simplify this expression a little further. First, rewriting the gamma function in terms of a factorial:
		\begin{align*}
			\left[\frac{ \Gamma(\mathit{l}+\eta_j+1+\alpha_j)}
			{\Gamma(\mathit{l}+\eta_j+1)}\right]^M
			&=\Big((\mathit{l}+\eta_j+\alpha_j)\dotsm (\mathit{l}+\eta_j+1)
			\Big)^M\\
			&=\Big(\mathit{l}^{\alpha_j}+\mathit{l}^{\alpha_j-1}
			\left(\frac{\alpha_j(\alpha_j+1)}{2}+\alpha_j \eta_j\right)+O\left(\mathit{l}^{\alpha_j-2}\right)
			\Big)^M
			\\
			&=\mathit{l}^{M\alpha_j}+M\mathit{l}^{M\alpha_j-1}
			\left(\alpha_j\eta_j+\frac{\alpha_j(\alpha_j+1)}{2}\right)
			+O\left(\mathit{l}^{M\alpha_j-2} \right).
		\end{align*}
		Taking the product over all $j$ yields:
		\begin{align}
			\label{combo_form}
			\sum_{\mathit{l}=-\eta_{min}}^{n-1-\eta_{max}}
			\frac{1}{n^{Ms}}\left[
			\mathit{l}^{Ms}+M\mathit{l}^{Ms-1}\left(
			\sum_{j=1}^m \alpha_j\eta_j+\frac{\alpha_j(\alpha_j+1)}{2} 
			\right)
			+O\left(\mathit{l}^{Ms-2}\right)\right].
		\end{align}
		
		To simplify this further we will use the fact that, for any constants $c$ and $t$, the following approximation holds (see Section 6 in \cite{RV}):
		\begin{align*}
			\sum_{\mathit{l}=c}^n \mathit{l}^t=\frac{n^{t+1}}{t+1}
			+\frac{n^t}{2}+O\left(n^{t-1}\right).
		\end{align*}
		This allows us to write:
		\begin{align*}
			\frac{1}{n^{Ms}}\sum_{\mathit{l}=-\eta_{\min}}^{n-1-\eta_{\max}}
			\mathit{l}^{Ms}
			=\frac{n}{Ms+1}+\left[\frac{1}{2}-(1+\eta_{max})\right]+O(n^{-1}).
		\end{align*}
		And:
		\begin{align*}
			\frac{1}{n^{Ms}}\sum_{\mathit{l}=-\eta_{\min}}^{n-1-\eta_{\max}}
			\mathit{l}^{Ms-1}
			=\frac{1}{Ms}+O\left(n^{-1}\right).
		\end{align*}
		
		Inserting both of these approximations into (\ref{combo_form}):
		\begin{align*}
			\sum_{q_1,\ldots,q_m=0}^{n-1} \int_{\mathbb{C}^m} \prod_{j=1}^m
			\left[
			n^{Mq_j}\Gamma(q_j+1)^{-M}(z_j)^{\alpha_j}\overline{(z_j)^{\beta_j}} 
			(z_j\overline{z_{j+1}})^{q_j}
			w_n(n^{M/2}z_j)d^2z_j\right]
			\\
			=
			\Bigg[\frac{n}{Ms+1}
			-(1+\eta_{max})
			+\frac{1}{2}+\frac{1}{s}\sum_{j=1}^m
			\left(\alpha_j\eta_j
			+\frac{\alpha_j(\alpha_j+1)}{2}\right)
			+O\left(n^{-1} \right)\Bigg].
		\end{align*}
		This is exactly what we wanted to prove. \end{proof} 
		
		\subsection{Variance for Polynomial Test Functions}
		
		We have now established that the higher cumulants of the polynomial linear statistic associated with the product of Ginibre matrices vanish in the limit, and so our next item of business is the computation of the limiting variance, which will establish the Gaussian limit.
		As in the case of computing the higher cumulant limits, the variance computations closely follow (and often altogether mirror) those in \cite{RV}.
		
		Consider two monomials, let's say $z^{a_1}\bar{z}^{b_1}$ and $z^{a_2}\bar{z}^{b_2}$, with $a_1+a_2=b_1+b_2$ (we will denote this common quantity by $s$, for brevity).  We now compute the covariance of the linear statistics associated with each monomial, $\mbox{Cov}\left(z^{a_1}\bar{z}^{b_1},z^{a_2}\bar{z}^{b_2}\right)$.  Indeed, it follows that 
		\[ \mbox{Cov}(z^{a_1} \bar{z}^{b_1}, z^{a_2} \bar{z}^{b_2}) = \Phi_1(z^{a_1 + a_2} \bar{z}^{b_1 + b_2}) - \Phi_2(z^{a_1} \bar{z}^{b_1}, z^{a_2} \bar{z}^{b_2}). \] 
		Thus, by the proof of Lemma \ref{comblemma}, we can estimate the covariance  $\mbox{Cov}\left(z^{a_1}\bar{z}^{b_1},z^{a_2}\bar{z}^{b_2}\right)$ 
		up to $O(1/n)$ error which we neglect, by the following expression:
		\begin{align*}
			&\Bigg(\frac{n}{Ms+1}
			-1
			+\frac{1}{2}+\frac{1}{s}\frac{(a_1+a_2)(a_1+a_2+1)}{2}
			\Bigg)\\
			&\qquad -\Bigg(\frac{n}{Ms+1}
			-(1+\max(0,b_1-a_1))
			+\frac{1}{2}\\
			&\qquad +\frac{1}{s}
			\left(a_1(b_1-a_1)
			+\frac{a_1(a_1+1)}{2}
			+\frac{a_2(a_2+1)}{2}\right)
			\Bigg).
		\end{align*}
		Simplifying, the covariance is:
		\begin{align*}
			\max(0,b_1-a_1)+\frac{1}{s}
			&\Bigg[\frac{(a_1+a_2)(a_1+a_2+1)}{2}
			-
			\left(a_1(b_1-a_1)
			+\frac{a_1(a_1+1)}{2}
			+\frac{a_2(a_2+1)}{2}\right)\Bigg]\\
			&=\max(0,b_1-a_1))+\frac{1}{s}\left(a_1a_2-a_1(b_1-a_1) \right)\\
			&=\max(0,b_1-a_1)+\frac{a_1b_2}{s}\\
			&=\max(0,a_1-b_1)+\frac{b_1a_2}{s}.
		\end{align*}
		The last equality follows because:
		\begin{align*}
			\max(0,b_1-a_1)-\max(0,a_1-b_1)
			&=\frac{(b_1-a_1)(a_1+a_2)}{s}\\
			&=\frac{b_1(a_2+a_1)-a_1(b_1+b_2)}{s}\\
			&=\frac{b_1a_2}{s}-\frac{a_1b_2}{s}.
		\end{align*}
		On the other hand, we can compute:
		\begin{align*}
			\frac{1}{\pi}\int_{|z|\leq 1}\overline{\partial}(z^{a_1}\bar{z}^{b_1})\times
			\overline{\bar{\partial}(z^{b_2}\bar{z}^{a_2})}d^2z
			&=\frac{b_1a_2}{\pi}\int_{|z|\leq 1} z^{a_1+a_2-1}
			\bar{z}^{b_1+b_2-1}d^2z\\
			&=\delta(a_1+a_2,b_1+b_2)\frac{b_1a_2}{a_1+a_2}.
		\end{align*}
		Similarly:
		\begin{align*}
			\frac{1}{\pi}\int_{|z|\leq 1}\partial(z^{a_1}\bar{z}^{b_1})\times
			\overline{\partial(z^{b_2}\bar{z}^{a_2})}d^2z
			=\delta(a_1+a_2,b_1+b_2)\frac{a_1b_2}{a_1+a_2}.
		\end{align*}
		We may also compute:
		\begin{align*}
			\frac{1}{2}\left<z^{a_1}\bar{z}^{b_1},z^{a_2}\bar{z}^{b_2}\right>_{H^{1/2}}&= \sum_{k>0}\left(k\times(z^{a_1}\bar{z}^{b_1})\wedge (k)\times
			\overline{(z^{b_2}\bar{z}^{a_2})\wedge (k)}\right)\\&=
			\delta(a_1+a_2,b_1+b_2)\times\max(0,a_1-b_1).
		\end{align*}
		Here, $f\wedge(k)$ represents the $k$-th Fourier coefficient with respect to the $H^{1/2}$ Sobolev norm on the unit circle (here we are extending the definition of the Fourier coefficient to include potentially complex-valued functions, in a slight abuse of notation), and $\delta$ denotes the usual Kronecker delta. Letting $f$ be a real valued polynomial, we see that $|\bar{\partial}f|^2=|\partial f|^2=|\nabla f|^2/4$.  Thus, expanding the expression for covariance into terms which involve only integrals of monomials, we see that the second cumulant is given by:
		\begin{align}
			\label{lim_var}
			\frac{1}{4\pi}\int_{|z|\leq 1} |\nabla f|^2 d^2z
			+\frac{1}{2}\|f\|^2_{H^{1/2}(|z|=1)}+o(1), 
		\end{align}
		which gives precisely the correct limiting variance when $n \to \infty$.
		
		\subsection{Variance in the General Case and the Proof of Theorem \ref{cltforginibre}}
		
		We have now established Theorem \ref{cltforginibre} in the case of real-valued polynomial test functions; to complete the proof it remains to extend the result to real valued test functions with continuous partial derivatives and at most polynomial growth at infinity. It will suffice, by the density of polynomials and the arguments in Section 7 of \cite{RV}, to show that the previously obtained formula for the limiting variance in the case of a real valued test function carries over to more general test functions. 
		
		The argument once again follows along the same lines as in \cite{RV}. We begin by showing that
		\begin{equation} \label{eq:growthbound}
			\lim_{n \to \infty} \int_{|z| > 1 + \epsilon} |z|^C K_n(z, \bar{z}) d \mu_n(z) = 0
		\end{equation}
		for any fixed $\epsilon > 0$ and $C > 0$.  In the forthcoming calculations, we will exploit the fact that 
		\begin{equation} \label{eq:gammabnd}
			\frac{1} {\pi} \int_{\mathbb{C}} |\xi|^{2(t-1)} e^{-|\xi|^2} d^2 \xi = \Gamma(t) 
		\end{equation}
		for $t > 0$, 
		which follows by changing to polar coordinates and applying a simple substitution.  We have
		\begin{align*}
			\int_{|z| > 1 + \epsilon} |z|^C K_n(z, \bar{z}) d \mu_n(z) &\leq \sum_{\iota = 0}^{n-1} \int_{|z| > 1 + \epsilon} |z|^C \frac{|z|^{2 \iota}}{\pi \Gamma(\iota + 1)^M} n^{M \iota} w_n(n^{M/2} z) d^2 z \\
			&\leq n \int_{|z| > 1 + \epsilon} |z|^C \frac{|z|^{2 (n-1)}}{\pi \Gamma(n)^M} n^{M (n-1)} w_n(n^{M/2} z) d^2 z \\
			&\leq \frac{n}{\pi^M \Gamma(n)^M} \int_{\mathbb{C}^M} 1_{|\xi_1 \cdots \xi_M| \geq (1+ \epsilon) n^{M/2}}  \left( \prod_{j=1}^M |\xi_j|^{2(n-1) + C} e^{-|\xi_j|^2} d^2 \xi_j \right). 
		\end{align*}
		If $|\xi_1 \cdots \xi_M| \geq (1+ \epsilon) n^{M/2}$, then there exists $j$ such that $|\xi_j| \geq (1 + \epsilon)^{1/M} n^{1/2}$.  Hence,
		\[ 1_{|\xi_1 \cdots \xi_M| \geq (1+ \epsilon) n^{M/2}} \leq \sum_{j=1}^M 1_{|\xi_j| \geq (1 + \epsilon)^{1/M} n^{1/2} }. \]
		Thus, using \eqref{eq:gammabnd} we obtain
		\begin{align*}
			\int_{|z| > 1 + \epsilon} |z|^C K_n(z, \bar{z}) d \mu_n(z) &\ll \frac{n}{\Gamma(n)} \left[ \frac{ \Gamma(n + C/2)}{\Gamma(n)} \right]^{M-1} \int_{|\xi| > (1+ \epsilon)^{1/M} n^{1/2}} |\xi|^{2(n-1) + C} e^{-|\xi|^2} d^2 \xi \\
			&\ll \frac{n^{O(1)}}{\Gamma(n)} \int_{|\xi| > (1+ \epsilon)^{1/M} n^{1/2}} |\xi|^{2(n-1) + C} e^{-|\xi|^2} d^2 \xi. 
		\end{align*}
		By changing to polar coordinates and applying a substitution, we see that
		\[  \frac{1}{\pi} \int_{|\xi| > (1+ \epsilon)^{1/M} n^{1/2}} |\xi|^{2(n-1) + C} e^{-|\xi|^2} d^2 \xi = \int_{(1 + \epsilon)^{2/M} n}^\infty t^{n-1 + C/2} e^{-t} dt \]
		is the upper incomplete gamma function.  Applying standard asymptotic expansions for the upper incomplete gamma function (see, for instance, \cite{Temme1, Temme2}), we conclude that
		\[ \frac{1}{\Gamma(n)} \frac{1}{\pi} \int_{(1+ \epsilon)^{1/M} n^{1/2}}^\infty |\xi|^{2(n-1) + C} e^{-|\xi|^2} d^2 \xi \ll n^{O(1)} \exp(-\Omega_\epsilon(\sqrt{n}) ). \]
		Combining the bounds above, we obtain \eqref{eq:growthbound} 
		for any fixed $C, \epsilon > 0$.

		From the arguments in \cite{RV}, it follows from \eqref{eq:growthbound} and the growth assumption on $f$ that $\Var(N_n[f])$ is asymptotically the same as $\Var(N_n[f \varphi])$), where $\varphi$ is a smooth function taking values in $[0,1]$, which equals $1$ on $\mathbb{U}$ and vanishes outside of $|z| \leq 1 + \epsilon/2$.  Thus, it suffices to assume that the test functions $f$ have continuous partial derivatives and are supported on $|z| \leq 1 + \epsilon/2$.  We work under these assumptions for the remainder of the proof.    
	
		We now turn to computing the limiting variance of $N_n[f]$ when $f$ has continuous partial derivatives and is supported on $|z| \leq 1 + \epsilon/2$.  Begin by making the definition:
		\begin{align*}
		\phi_n(\nu,\eta)&=\int_{\mathbb{C}}\frac{1}{z-\nu}
		\overline{\left(\frac{1}{z-\eta}\right)}K_n(z,\bar{z})d\mu_n(z)\\
		&\qquad -\int_{\mathbb{C}}\int_{\mathbb{C}}
		\frac{1}{z-\nu}
		\overline{\left(\frac{1}{w-\eta}\right)}K_n(z,\bar{w})K_n(w,\bar{z})
		d\mu_n(z)d\mu_n(w).
		\end{align*}
		We next need the following lemma. 
		\begin{lemma}
			\label{firstlimlem}
			Let $f,g$ be real valued test functions with continuous partial derivatives which are supported on $|z| \leq 1 + \epsilon /2$. The following limit holds: 
			\begin{align*}
				\lim_{n\to\infty}\frac{1}{\pi^2}\int_{|\nu|<1}\int_{
					|\eta|<1+
					\epsilon}
				\overline{\partial} f(\nu) \overline{\overline{\partial} g(\eta)}
				\phi_n(\nu,\overline{\eta})d^2\nu d^2\eta
				=\frac{1}{\pi}\int_{\mathbb{U}}
				\overline{\partial} f(\nu) \overline{\overline{\partial} g(\nu)}
				d^2\nu.
			\end{align*}
		\end{lemma}
		This will allow us to quickly extend the limiting variance formula for more general test functions and prove Theorem \ref{cltforginibre}. 
		
		\begin{proof}[Proof of Theorem \ref{cltforginibre}]
			Following \cite{RV}, we expand the covariance:
			\begin{align*}
				\mbox{Cov}(N_n[f],N_n[g]))&=\int_\mathbb{C} f(z)\overline{g(z)}
				K_n(z,\bar{z})d\mu_n(z)\\
				&\qquad -\int_\mathbb{C}\int_\mathbb{C}f(z)\overline{g(w)}
				|K_n(z,\overline{w})|^2 d\mu_n(z)d\mu_n(w), 
			\end{align*}
			where (as discussed above) we assume $f,g$ are test functions supported on $|z| \leq 1 + \epsilon/2$ with continuous partial derivatives.  
			For test functions $f$ which are once continuously differentiable, the well known Cauchy-Pompeiu formula states (where we continue to let $\mathbb{U}$ denote the disc $|z|<1$):
			\begin{align*}
				f(\zeta)=
				-\frac{1}{\pi }\int_{\mathbb{U}}\frac{\bar{d}f(w)}{w-\zeta}d^2w
				+\frac{1}{2\pi i}\int_{\partial\mathbb{U}}\frac{f(w)}{w-\zeta}dw
			\end{align*}
			for $\zeta \in \mathbb{U}$.  
			Substituting this in and simplifying as in \cite{RV}, we can read the covariance of the linear statistic from:
			\begin{align*}
				\frac{1}{\pi^2}\int_{|v|<1+\epsilon}\int_{\nu<1+\epsilon}
				\partial f(\nu)\overline{\partial g(\eta)}\times\Bigg\{
				\int_{\mathbb{C}}\frac{1}{z-\nu}\overline{\frac{1}{z-\eta}}
				K_n(z,\bar{z})d\mu_n(z)\\-\int_{\mathbb{C}}\int_{\mathbb{C}}
				\frac{1}{z-nu}\overline{\frac{1}{z-\eta}}|K_n(z,\bar{w})|^2
				d\mu_n(z)d\mu_n(w)
				\Bigg\}d^2\nu d^2\eta.
			\end{align*}
			In our notation, this is just:
			\begin{align*}
				\frac{1}{\pi^2}\int_{|\nu|<1+\epsilon}\int_{|\eta|<1+\epsilon}
				\partial f(\nu)\overline{\partial g(\eta)}
				\phi_n(\nu,\eta)d^2\nu d^2\eta.
			\end{align*}
			By Lemma \ref{firstlimlem}: 
			\begin{align*}
				\lim_{n\to\infty}\frac{1}{\pi^2}\int_{|\nu|<1}\int_{
					|\eta|<1+
					\epsilon}
				\overline{\partial} f(\nu) \overline{\overline{\partial} g(\eta)}
				\phi_n(\nu,\overline{\eta})d^2\nu d^2\eta
				=
				\frac{1}{\pi}\int_{\mathbb{U}}
				\overline{\partial} f(\nu) \overline{\overline{\partial} g(\nu)}
				d^2\nu
			\end{align*}
			Setting $f=g$, this quantity becomes:
			\begin{align*}
				\frac{1}{\pi}\int_{\mathbb{U}}|\partial f(\nu)|^2d^2\nu
				=
				\frac{1}{4\pi}\int_{\mathbb{U}}|\nabla f(\nu)|^2d^2\nu.
			\end{align*}
			By the proof of Lemma 13 in \cite{RV} (with additional factors of $M$, which do not affect the resulting asymptotics), we also have the limit:
			\begin{align*}
				\lim_{n\to\infty}\frac{1}{\pi^2}\iint_{1<|\nu|,|\eta|<1+\epsilon}
				\overline{\partial} f(\nu) \overline{\overline{\partial} g(\eta)}
				\phi_n(\nu,\overline{\eta})d^2\nu d^2\eta
				=\frac{1}{2}\left<f,g\right>_{H^{1/2}(|z|=1)}.  
			\end{align*}
			Here, we have exploited the fact that our test functions are supported on the disk $|z| \leq 1 + \epsilon/2$.  
			Setting $f=g$ here as well, this quantity becomes:
			\begin{align*}
				\frac{1}{2}\|f \|^2_{H^{1/2}(\partial \mathbb{U})}.
			\end{align*}
			We can take the limit of the covariance as $n\to\infty$ to obtain:
			\begin{align*}
				\lim_{n\to\infty}\mbox{Var}(N_n[f])=
				\frac{1}{4\pi}\int_{\mathbb{U}}|\nabla f(\nu)|^2d^2\nu+
				\frac{1}{2}\|f\|^2_{H^{1/2}(\partial \mathbb{U})}.
			\end{align*}
			This is what we wanted to show. \end{proof} 
		
		It remains to prove Lemma \ref{firstlimlem}, which we turn to presently.
		\begin{proof}[Proof of Lemma \ref{firstlimlem}] Fix $\nu\in\mathbb{U}$, with $|\nu|=a<1$. Let $\mathbb{B}$ be the disc in the complex plane about the origin with radius equal to $b >|\nu|$. Our first task is to establish the following limit:
		\begin{align}
			\label{first_eq_rv}
			\lim_{n\to\infty}\int_{\mathbb{B}} \phi_n(\nu,\overline{\eta})d^2\eta=\pi.
		\end{align}
		To do this, we must investigate the integral:
		\begin{align*}
		I_{\mathbb{B}}&=\int_{\mathbb{B}}\phi_n(\nu,{\eta})d^2\eta \\
		&=
		\int_{\mathbb{B}}\int_{\mathbb{C}}\frac{1}{z-\nu}
		\overline{\left(\frac{1}{z-\eta}\right)}K_n(z,\bar{z})d\mu_n(z)
		d^2\eta\\
		&\qquad-\int_{\mathbb{B}}\int_{\mathbb{C}}\int_{\mathbb{C}}
		\frac{1}{z-\nu}
		\overline{\left(\frac{1}{w-\eta}\right)}K_n(z,\bar{w})K_n(w,\bar{z})
		d\mu_n(z)d\mu_n(w)
		d^2\eta.
		\end{align*}
		Our strategy will be to first rewrite $I_{\mathbb{B}}$ in a more tractable form, which is straightforward if somewhat lengthy to transcribe, and then, second, to send $n$ to infinity to obtain the desired limiting value of $\pi$. 
		
		Suppose $|z|>b$. By series expansion:
		\begin{align}
		\int_{|\eta|\leq b}\frac{1}{z-\eta}d^2\eta
		=\int_{|\eta|\leq b}\frac{1}{z}\left(1+\frac{\eta}{z}+\left(\frac{\eta}{z}\right)^2+\cdots 
		\right)
		d^2\eta
		=\frac{\pi b^2}{z}.
		\end{align}
		Suppose instead that $|z|<b$. By the Cauchy-Pompeiu formula (and using an appropriate partial fraction decomposition to show that the relevant line integral vanishes):
		\begin{align}
		\int_{|\eta|\leq b}\frac{1}{z-\eta}d^2\eta=\pi \bar{z}-
		\frac{1}{2i}\int_{|\eta|=b}\frac{\bar{\eta}}{\eta-z}d\eta
		=\pi \bar{z}.
		\end{align}
		We can expand:
		\begin{align*}
		\int_{0\leq |z| \leq a}\frac{z}{z-\nu}K_n(z,\bar{z})d\mu_n(z)
		=\int_{0\leq |z| \leq a}\frac{z}{\nu}
		\left(\sum_{j=0}^\infty\frac{z^j}{\nu^j} \right)K_n(z,\bar{z})d\mu_n(z).\\
		\end{align*}
		This quantity vanishes identically, due to rotational invariance. 
		Substituting in these formulas, $I_{\mathbb{B}}$ becomes:
		\begin{align*}
			\pi \int_{a\leq |z| \leq b}\frac{z}{z-\nu}
			K_n(z,\bar{z})d\mu_n(z)
			+\pi\int_{|z|>b}\frac{ b^2}{\bar{z}}\frac{1}{z-\nu}
			K_n(z,\bar{z})d\mu_n(z)
			\\
			-\pi \int_{|z|>a}\int_{a\leq |w| \leq b}
			\frac{w}{z-\nu}
			K_n(z,\bar{w})K_n(w,\bar{z})
			d\mu_n(z)d\mu_n(w)\\
			-\pi \int_{|z|>a}\int_{|w|>b}
			\frac{b^2}{\bar{w}}\frac{1}{z-\nu}
			K_n(z,\bar{w})K_n(w,\bar{z})
			d\mu_n(z)d\mu_n(w).
		\end{align*}
		
		By rotational invariance again:
		\begin{align*}
		\int_{|z|>b}\frac{ b^2}{\bar{z}}\frac{1}{z-\nu}
		K_n(z,\bar{z})d\mu_n(z)&=
		\int_{|z|>b}\frac{b^2}{|z|^2}
		\left(1+\frac{z}{\nu}+\cdots  \right)
		\left(\sum_{j=0}^{n-1}\frac{|z|^{2j}}{h_j} \right)d\mu_n(z)\\
		&=b^2\int_{|z|>b}
		\left(\sum_{j=0}^{n-1}\frac{|z|^{2j-2}}{h_j} \right)d\mu_n(z).
		\end{align*}
		Combining:
		\begin{align*}
		\pi \int_{a\leq |z| \leq b}\frac{z}{z-\nu}
		K_n(z,\bar{z})d\mu_n(z)
		+\pi\int_{|z|>b}\frac{ b^2}{\bar{z}}\frac{1}{z-\nu}
		K_n(z,\bar{z})d\mu_n(z)\\
		=\int_{|z| \leq b}
		\left(\sum_{j=0}^{n-1}\frac{|z|^{2j}}{h_j} \right)
		d\mu_n(z)
		+b^2\int_{|z|>b}
		\left(\sum_{j=0}^{n-1}\frac{|z|^{2j-2}}{h_j} \right)d\mu_n(z).
		\end{align*}
		
		Turning now to the double integral terms in $I_{\mathbb{B}}$, we can write:
		\begin{align*}
		&\int_{|z|\geq a}\int_{a\leq |w| \leq b}
		\frac{w}{z-\nu}
		K_n(z,\bar{w})K_n(w,\bar{z})
		d\mu_n(z)d\mu_n(w)\\
		&=\int_{|z|\geq a}\int_{a\leq |w| \leq b}
		\frac{w}{z}\left(\sum_{j=0}^\infty \frac{\nu^j}{z^j}\right)
		\left(\sum_{j=0}^{n-1}\frac{z^j\bar{w}^j}{h_j} \right)
		\left(\sum_{j=0}^{n-1}\frac{w^j\bar{z}^j}{h_j} \right)
		d\mu_n(z)d\mu_n(w).
		\end{align*}
		By orthogonality, this last display is:
		\begin{align*}
		\sum_{j=0}^{n-2}\int_{|z|\geq a}\int_{a\leq |w| \leq b}
		\left(\frac{|w|^{j+1}|z|^j}{h_jh_{j+1}} \right)
		d\mu_n(z)d\mu_n(w).
		\end{align*}
		The other double integral may be handled similarly:
		\begin{align*}
		&\int_{|z|>a}\int_{|w|>b}
		\frac{b^2}{\bar{w}}\frac{1}{z-\nu}
		K_n(z,\bar{w})K_n(w,\bar{z})
		d\mu_n(z)d\mu_n(w)\\
		&=\int_{|z|\geq a}\int_{|w|>b}
		\frac{b^2}{z\bar{w}}\left(\sum_{j=0}^\infty
		\frac{\nu^j}{z^j}\right)
		\left(\sum_{j=0}^{n-1}\frac{z^j\bar{w}^j}{h_j} \right)
		\left(\sum_{j=0}^{n-1}\frac{w^j\bar{z}^j}{h_j} \right)
		d\mu_n(z)d\mu_n(w)
		\\&=
		b^2\sum_{j=0}^{n-2}\int_{|z|\geq a}\int_{|w|>b} 
		\left(\frac{|z|^{2j}|w|^{2j}}{h_j h_{j+1}} \right)
		d\mu_n(z)d\mu_n(w).
		\end{align*}
		Combining, we obtain:
		\begin{align*}
		&\pi \int_{|z|>a}\int_{a\leq |w| \leq b}
		\frac{w}{z-\nu}
		K_n(z,\bar{w})K_n(w,\bar{z})
		d\mu_n(z)d\mu_n(w)\\
		&+\pi \int_{|z|>a}\int_{|w|>b}
		\frac{b^2}{\bar{w}}\frac{1}{z-\nu}
		K_n(z,\bar{w})K_n(w,\bar{z})
		d\mu_n(z)d\mu_n(w)\\
		&=
		\pi\sum_{j=0}^{n-2}\int_{|z|\geq a}\int_{a\leq |w| \leq b}
		\left(\frac{|w|^{j+1}|z|^j}{h_jh_{j+1}} \right)
		d\mu_n(z)d\mu_n(w)
		\\
		&+\pi b^2\sum_{j=0}^{n-2}\int_{|z|\geq a}\int_{|w|>b} 
		\left(\frac{|z|^{2j}|w|^{2j}}{h_j h_{j+1}} \right)
		d\mu_n(z)d\mu_n(w).
		\end{align*}
		Substituting into the formula for $I_{\mathbb{B}}$ provides:
		\begin{align*}
		I_{\mathbb{B}}&=\pi \int_{|z| \leq b}
		\left(\sum_{j=0}^{n-1}\frac{|z|^{2j}}{h_j} \right)
		d\mu_n(z)
		+b^2 \pi \int_{|z|>b}
		\left(\sum_{j=0}^{n-1}\frac{|z|^{2j-2}}{h_j} \right)d\mu_n(z)\\
		&-\pi\sum_{j=0}^{n-2}\int_{|z|\geq a}\int_{0\leq |w| \leq b}
		\left(\frac{|w|^{2(j+1)}|z|^{2j}}{h_jh_{j+1}} \right)
		d\mu_n(z)d\mu_n(w)
		\\
		&-\pi b^2\sum_{j=0}^{n-2}\int_{|z|\geq a}\int_{|w|>b} 
		\left(\frac{|z|^{2j}|w|^{2j}}{h_j h_{j+1}} \right)
		d\mu_n(z)d\mu_n(w).
	\end{align*}
		After recombination:
		\begin{align*}
			&\pi \sum_{j=0}^{n-1}\left[\int_{|z|>a}
			\left(\frac{|z|^{2j}}{h_j} \right)
			d\mu_n(z)
			-\int_{|z|>b}\left(1-\frac{b^2}{|z|^2}\right)
			\left(\frac{|z|^{2j}}{h_j} \right)d\mu_n(z)\right]\\
			&-\pi\sum_{j=0}^{n-2}\int_{|z|\geq a}\int_{\mathbb{C}}
			\left(\frac{|w|^{2(j+1)}|z|^{2j}}{h_jh_{j+1}} \right)
			d\mu_n(z)d\mu_n(w)
			\\
			&+\pi \sum_{j=0}^{n-2}\int_{|z|\geq a}\int_{|w|>b}
			\left(1-\frac{b^2}{|w|^2} \right) 
			\left(\frac{|z|^{2j}|w|^{2j}}{h_j h_{j+1}} \right)
			d\mu_n(z)d\mu_n(w).
		\end{align*}
		By the definition of $h_{j+1}$ and cancellation:
		\begin{align*}
			I_{\mathbb{B}}&=\pi
			\int_{|z|>a}\left(\frac{|z|^{2(n-1)}}{h_{n-1}}\right)d\mu_n(z) -
			\pi\int_{|z|>b}\left(1-\frac{b^2}{|z|^2}\right)
			d\mu_n(z)
			\\
			&\qquad -\pi \sum_{j=0}^{n-2}\int_{|z|\leq a}\int_{|w|>b}
			\left(1-\frac{b^2}{|w|^2} \right) 
			\left(\frac{|z|^{2j}|w|^{2j}}{h_j h_{j+1}} \right)
			d\mu_n(z)d\mu_n(w)
		\end{align*}
		We have now derived the expression for $I_{\mathbb{B}}$ which we wanted, and what remains is to show that this expression goes to the claimed limit for large values of $n$. As we shall see, in the large $n$ limit the first integral goes to one while the second and third vanish, so the claimed limiting expression holds.
		
		To obtain the estimates we require, we need to plug in the definition of $d\mu_n(z)$. 
		We begin with 
		\begin{align*}
		&\pi^{1-M}n^M \int_{|z|<a}\int_{\mathbb{C}^M}
		|z|^{2(n-1)}
		\delta^2(n^{M/2}z-\xi_M\dotsm \xi_1)\left(\prod_{j=1}^M e^{-|\xi_j|^2}d^2\xi_j\right)
		d^2z\\
		&=\pi^{1-M}n^{-M(n-1)}\int_{\mathbb{C}^M}
		|\xi_M\dotsm \xi_1|^{2(n-1)}1_{|\xi_M\dotsm \xi_1|\leq n^{M/2}a}
		\left(\prod_{j=1}^M e^{-|\xi_j|^2}d^2\xi_j\right). 
		\end{align*}
		Observe that if $|\xi_1 \cdots \xi_M| \leq n^{M/2} a$, then there exists $j$ such that $|\xi_j| \leq n^{1/2} a^{1/M}$.  Thus, we find that
		\[ 1_{|\xi_M\dotsm \xi_1|\leq n^{M/2}a} \leq 1_{|\xi_1| \leq n^{1/2} a^{1/M}} + \cdots + 1_{|\xi_M| \leq n^{1/2} a^{1/M}} , \]
		and hence (applying \eqref{eq:gammabnd})
		\begin{align*}
		 \frac{1}{\pi^M} \int_{\mathbb{C}^M}
		&|\xi_M\dotsm \xi_1|^{2(n-1)}1_{|\xi_M\dotsm \xi_1|\leq n^{M/2}a}
		\left(\prod_{j=1}^M e^{-|\xi_j|^2}d^2\xi_j\right) \\
		&\ll [\Gamma(n)]^{M-1} \frac{1}{\pi} \int_{\mathbb{C}} 1_{|\xi| \leq n^{1/2} a^{1/M}} |\xi|^{2(n-1)} e^{-|\xi|^2} d^2 \xi \\
		&=  [\Gamma(n)]^{M-1} \gamma(n, a^{2/M} n), 
		\end{align*}
		where the last step follows by rewriting the integral in polar coordinates and making a substitution.  Here, $\gamma$ is the incomplete gamma function defined by
		\[ \gamma(x, y) = \int_0^y t^{x-1} e^{-t} dt. \]
		Since $a < 1$, we can apply a standard asymptotic expansion for the incomplete gamma function (see, for instance, \cite{Temme1, Temme2}) to conclude that
		\begin{equation} \label{eq:incompgamma}
			\gamma(n, a^{2/M} n) = o(\Gamma(n)). 
		\end{equation}
		
		Combining the bounds above allows us to estimate
		
		\begin{align*}
		\left(\frac{n^{n-1}}{(n-1)!}\right)^M
		& \left(
		\pi^{1-M}n^M \int_{|z|<a}\int_{\mathbb{C}^M}
		|z|^{2(n-1)}
		\delta^2(n^{M/2}z-\xi_M\dotsm \xi_1)\left(\prod_{j=1}^M e^{-|\xi_j|^2}d^2\xi_j\right)
		d^2z
		\right)
		\\
		&\ll \left(\frac{n^{n-1}}{(n-1)!}\right)^M n^{-M(n-1)} [\Gamma(n)]^{M-1} \gamma(n, a^{2/M} n). 
		\end{align*}
		This expression goes to zero by \eqref{eq:incompgamma}, which implies:
		\begin{align*}
			\lim_{n\to\infty}\int_{|z|<a}\left(\frac{|z|^{2(n-1)}}{h_{n-1}}
			\right)d\mu_n(z)=0.
		\end{align*}
		Since  $h_{n-1}$ is the normalization constant, this is exactly equivalent to:
		\begin{align}
		\int_{|z|>a}\left(\frac{|z|^{2(n-1)}}{h_{n-1}}\right)
		d\mu_n(z)\to 1.
		\end{align}
		We similarly obtain:
		\begin{align*}
		\int_{|z|<b}
		d\mu_n(z)&=
		\pi^{1-M}n^M \int_{|z|<b}\int_{\mathbb{C}^M}
		\delta^2(n^{M/2}z-\xi_M\dotsm \xi_1)\left(\prod_{j=1}^M e^{-|\xi_j|^2}d^2\xi_j\right)
		d^2z\\
		&= \pi^{1-M} \int_{\mathbb{C}^M}1_{|\xi_1\dotsm\xi_M|<n^{M/2}b}
		\left(\prod_{j=1}^M e^{-|\xi_j|^2}d^2\xi_j\right)\\
		&\geq \pi^{1-M} \left(\int_{|\xi|<n^{1/2}b^{1/M}}
		 e^{-|\xi|^2}d^2\xi\right)^M\\&=
		\pi\left(1- o(1) \right).
		\end{align*}
		Consequently:
		\begin{align}
		\int_{|z|>b}\left(1-\frac{b^2}{|z|^2}\right)
		d\mu_n(z)\to 0.
		\end{align}
		It remains to investigate:
		\begin{align*}
		\sum_{j=0}^{n-2} &\int_{|z|\leq a}\int_{|w|>b}
		\left(1-\frac{b^2}{|w|^2} \right) 
		\left(\frac{|z|^{2j}|w|^{2j}}{h_j h_{j+1}} \right)
		d\mu_n(z)d\mu_n(w)\\
		&\leq C\sum_{j=0}^{n-2}\left(\int_{|z|\leq a}
		\left(\frac{|z|^{2j}}{h_j} \right)d\mu_n(z)\right)
		\left(\int_{|w|>b}
		\left(\frac{|w|^{2j}}{h_{j+1}} \right)
		d\mu_n(w)\right).
		\end{align*}
		This goes to zero by a similar argument as in \cite{RV} (in fact, this reduces to essentially their Equation (7.10) raised to the $M$-th power), which concludes the proof of (\ref{first_eq_rv}).
				
		Now, by the same argument as in the proof of Lemma 7.2 in \cite{RV} (with some terms raised to the $M$-th power which, as in the preceding argument, does not change the resulting asymptotics) we obtain, for any region $A$ contained in $|z|<1+\epsilon$ and any $\delta > 0$, the following limit:
		\begin{align}
		\label{second_eq_rv}
		\lim_{n\to\infty}\int_{A\cap |\nu-\eta|>\delta}
		\left|\phi_n(\nu,\overline{\eta})\right|d^2\eta=0.
		\end{align}
		Following the arguments in \cite{RV}, we combine (\ref{first_eq_rv}) and (\ref{second_eq_rv}) along with the assumptions on $f$ and $g$ to see that 
		\begin{align*}
				\lim_{n\to\infty}\frac{1}{\pi}\int_{|\nu|<1}\int_{
					|\eta|<1+
					\epsilon}
				\overline{\partial} f(\nu) \overline{\overline{\partial} g(\eta)}
				\phi_n(\nu,\overline{\eta})d^2\nu d^2\eta
				=\int_{\mathbb{U}}
				\overline{\partial} f(\nu) \overline{\overline{\partial} g(\nu)}
				d^2\nu,
			\end{align*}
		which completes the proof of the lemma.\end{proof}
		
		\subsection{Products of Truncated Unitary Matrices: Proof of Theorem \ref{unit4mom}}
		
		The methods developed in \cite{RV} also extend to products of truncated unitary matrices. The eigenvalues of the product of $M$ truncated unitary matrices (that is, the product of the $n$ by $n$ matrices formed by taking the $n$ by $n$ minors of $M$ independent $K$ by $K$ random unitary matrices, distributed according the uniform Haar measure) form a determinantal point process with the following kernel \cite{AB}:
		\begin{align*}
			K_{n}(z,\bar{u})=\sum_{t=0}^{n-1} \frac{(z\overline{u})^t}{h_t}, 
		\end{align*}
where the normalization constants $h_t$ will be defined below (these differ from the normalization constants appearing in the previous subsection).  
		We define $\kappa=K-n$ and assume that $\kappa=\lfloor \tau n \rfloor$ holds for some fixed constant $\tau\in(1/2,1)$. The associated measures $\mu_n$ are simply the product of the Lebesgue measure on the complex plane with the following weight function:
		\begin{align*}
			w(z)=\pi^{1-M} \int_{\mathbb{U}^M} \prod_{m=1}^M \frac{1}{\Gamma(\kappa)}
			\Bigg[
			(1-|\xi_m|^2)^{\kappa-1}
			 \delta^2(z-\xi_M\dotsm\xi_1)d^2\xi_m\Bigg].
		\end{align*}
		The normalization here is slightly different than in \cite{AB}, simply for ease of presentation. Notice that this weight is rotationally invariant, so once again we can apply the approach of \cite{RV} and aim to show that the following expression vanishes for $k \geq 3$:
		\begin{align}
			\sum_{m=1}^k \frac{(-1)^m}{m}\sum_{\sigma: [k]\twoheadrightarrow [m]}
			\Phi_{m}\left( \sigma (z^{a_1} \bar{z}^{b_1}, \ldots, z^{a_k} \bar{z}^{b_k} ) \right).
		\end{align} 
		Recall that $\Phi_m$ is defined in \eqref{def:Phi_m} and $\sigma (z^{a_1} \bar{z}^{b_1}, \ldots, z^{a_k} \bar{z}^{b_k} )$ is defined using \eqref{def:sigma}.  
		Using Fubini's theorem and polar coordinates, we can integrate against the weight function:
		\begin{align*}
			\int_{\mathbb{C}}|z|^{2t}w(z)d^2z
			&=\pi^{1-M}
			\prod_{m=1}^M\left[\frac{1}{\Gamma(\kappa)} \int_{\mathbb{U}}
			|\xi_{m}|^{2t}(1-|\xi_m|^2)^{\kappa-1}d^2\xi_m\right]\\
			&=\pi^{1-M}
			\left[\frac{2\pi}{\Gamma(\kappa)} \int_{0}^1
			r^{2t+1}(1-r^2)^{\kappa-1}d r\right]^M.
		\end{align*}
		In other words:
		\begin{align} \label{def:h_tunit}
			h_t=\int_{\mathbb{C}}|z|^{2t}w(z)d^2z=\pi
			\left[\frac{\Gamma(t+1)}{\Gamma(t+\kappa+1)}\right]^M.
		\end{align}
		
		First, let's compute the expectation of the normalized linear statistic of a monomial. We have:
		\begin{align*}
			\mathbf{E}\left[\frac{1}{n}\sum_{j=1}^n \left|\lambda_j\right|^{2L}\right]&=\frac{1}{n}
			\sum_{t=0}^{n-1}\int_{\mathbb{C}} |z|^{2L+2t}
			\frac{1}{\pi}\left(\frac{\Gamma(t+\kappa+1)}
			{\Gamma(t+1)} \right)^M
			w(z) d^2z\\
			&=\frac{1}{n}
			\sum_{t=0}^{n-1}
			\left(\frac{\Gamma(t+\kappa+1)}{\Gamma(t+1)}
			\frac{\Gamma(L+t+1)}{\Gamma(L+t+\kappa+1)}
			\right)^M\\
			&=\frac{1}{n}
			\sum_{t=0}^{n-1}
			\left(\frac{(L+t)\dotsm(1+t)}
			{(\kappa+[L+t])\dotsm(\kappa+[1+t])}
			\right)^M.
		\end{align*}
		At leading order:
		\begin{align*}
			\frac{1}{n}
			\sum_{t=0}^{n-1}
			\left(\frac{t^L}
			{(\kappa+t)^L}
			\right)^M
			=\frac{1}{n}
			\sum_{t=0}^{n-1}
			\left(\frac{t/n}
			{\kappa/n+t/n}
			\right)^{LM}
			\to \int_{0}^1 \left(\frac{x}{\tau+x}\right)^{LM}dx.
		\end{align*}
		We identify this limit with the moments of the following limiting density for the eigenvalues of products of truncated unitary matrices derived in \cite{AB} (which vanishes on $|z|>(1/(\tau+1))^{M/2}$):
		\begin{align}
			\phi(z)=\frac{1}{\pi M}\frac{\tau}{|z|^{2(1-1/M)}(1-|z|^{2/M})^2}
			\mbox{  }
			\mbox{  on }
			|z|\leq \left(\frac{1}{\tau+1}\right)^{M/2}.
		\end{align}
		Indeed, we have (with $\mathbb{I}(z)$ denoting the indicator function):
		\begin{align*}
			\frac{\tau}{\pi M}\int_{\mathbb{C}}|z|^{2L}	
			\left( \mathbb{I}(z)_{ \{ |z|\leq (1-\tau/(\tau+1))^{M/2} \}}
			\frac{1}{|z|^{2(1-1/M)}(1-|z|^{2/M})^2}
			\right)d^2z
			\\
			=\frac{2\tau}{M}
			\int_{0}^{(1-\tau/(\tau+1))^{M/2})} r^{2L}\times\left(
			\frac{r^{2/M-1}}{(1-r^{2/M})^2}
			\right)dr.
		\end{align*}
		Setting $x=\tau [1/(1-r^{2/M})-1]$, so that $r^{2/M}=1-\tau/(\tau+x)$ and also $dx=2r\tau/(1-r^2)^2 dy$, and substituting:
		\begin{align*}
			\int_{0}^1 \left(1-\frac{\tau}{\tau+x}\right)^{ML}dx
			=\int_{0}^1 \left(\frac{x}{\tau+x}\right)^{ML} dx.
		\end{align*}

		To deal with the higher cumulants, we apply the same argument as in the case of Ginibre products. We fix an arbitrary choice of nonnegative integers $\alpha_1,\ldots,\alpha_m$ and $\beta_1,\ldots,\beta_m$, and taking $f_i(z) = z^{\alpha_i} \bar{z}^{\beta_i}$ we calculate:
		\begin{align*}
			\Phi_m(f_1,\ldots,f_m)&=
			\int_{\mathbb{C}^m} z_1^{\alpha_1}\overline{z_1^{\beta_1}}\dotsm
			z_m^{\alpha_m}\overline{z_m^{\beta_m}}
			\left(\sum_{t=0}^{n-1}\frac{(z_1\overline{z_2})^t}{h_t} \right)\cdots 
			\left(\sum_{t=0}^{n-1}\frac{(z_m\overline{z_1})^t}{h_t} \right) \prod_{j=1}^m w({z_j}) d^2 z_j \\
			&=\sum_{q_1,\ldots,q_m=0}^n\int_{\mathbb{C}^m}\prod_{j=1}^m
			\left[
			h_{q_j}^{-1}(z_j)^{\alpha_j}\overline{(z_j)^{\beta_j}}
			(z_j\overline{z_{j+1}})^{q_j}\right]
			w(z_1)d^2z_1\dotsm w(z_m)d^2z_m
		\end{align*}
		again with the convention that $z_{m+1} = z_1$.  Again, we define:
		\begin{align*}
			\gamma_j&=\beta_j-\alpha_j\\
			\eta_j&=\gamma_1+\cdots +\gamma_j.
		\end{align*}
		We will continue to use $\eta_{max}$ and $\eta_{min}$ to denote the maximal and minimal values of $\eta_j$. The nonvanishing condition again becomes:
		\begin{align*}
			q_j=q_m+\eta_j\\
			\sum_{j=1}^m \alpha_j =\sum_{j=1}^m \beta_j.
		\end{align*}
		Setting $\mathit{l}=q_m$, we obtain:
		\begin{align*}
			\Phi_m(f_1,\ldots,f_m)=\sum_{\mathit{l}=-\eta_{min}}
			^{n-1-\eta_{max}}
			\prod_{j=1}^m h_{\mathit{l}+\eta_j}^{-1}
			\times\left[\int_{\mathbb{C}}
			|z|^{2 (\textit{l} + \eta_j + \alpha_j) }w(|z|)d^2z\right].
		\end{align*}
		Substituting the moments of the weight:
		\begin{align*}
			\Phi_m(f_1,\ldots,f_m)
			=\pi^m\sum_{\mathit{l}=-\eta_{min}}^{n-1-\eta_{max}}
			\prod_{j=1}^m h_{\mathit{l}+\eta_j}^{-1}
			\left[\frac{(\mathit{l}+\eta_j+\alpha_j)!}
			{(\kappa+\mathit{l}+\eta_j+\alpha_j)!}
			\right]^M.
		\end{align*}
		By definition (see \eqref{def:h_tunit}):
		\begin{align*}
			h^{-1}_{\mathit{l}+\eta_j}=\frac{1}{\pi}
			\left[\frac{(\kappa+\mathit{l}+\eta_j)!}
			{(\mathit{l}+\eta_j)!}\right]^M.
		\end{align*}
		Substituting this into the previous expression:
		\begin{align*}
			\Phi_m(f_1,\ldots,f_m)
			&=
			\sum_{\mathit{l}=-\eta_{min}}^{n-1-\eta_{max}}
			\prod_{j=1}^m
			\left[
			\frac{(\kappa+\mathit{l}+\eta_j)!}
			{(\mathit{l}+\eta_j)!}
			\frac{(\mathit{l}+\eta_j+\alpha_j)!}
			{(\kappa+\mathit{l}+\eta_j+\alpha_j)!}
			\right]^M \\
			&=
			\sum_{\mathit{l}=-\eta_{min}}^{n-1-\eta_{max}}
			\prod_{j=1}^m
			\left[
			\frac{(\mathit{l}+\eta_j+1)\dotsm (\mathit{l}+\eta_j+\alpha_j)}
			{(\mathit{l}+\eta_j+\kappa+1)\dotsm (\mathit{l}+\eta_j+\kappa+\alpha_j)}
			\right]^M\\
			&=\sum_{\mathit{l}=-\eta_{min}}^{n-1-\eta_{max}}
			\prod_{j=1}^m
			\left[\prod_{t=1}^{\alpha_j} \frac{\mathit{l}+\eta_j+t}{\mathit{l}+\kappa}
			\left(1
			-\frac{\eta_j+t}{\mathit{l}+\kappa}
			+O\left(\frac{1}{\kappa^2}\right)
			\right)
			\right]^M.
		\end{align*}
		Here we have used the identity (for $\xi=\eta_j+t$):
		\begin{align}
			\frac{\mathit{l}+\xi}{\mathit{l}+\xi+\kappa}
			=\frac{\mathit{l}+\xi}{\mathit{l}+\kappa}\left(
			1-\frac{\xi}{\mathit{l}+\kappa}
			+\frac{\xi^2}{(\mathit{l}+\kappa)^2}
			\frac{1}{1+\xi/(\mathit{l}+\kappa)}
			\right).
		\end{align}
		We may therefore further rewrite $\Phi_{m}(f_1,\ldots,f_m)$ as:
		\begin{align*}
			\sum_{\mathit{l}=-\eta_{min}}^{n-1-\eta_{max}}
			\prod_{j=1}^m\left(\frac{1}{\mathit{l}+\kappa}\right)^{M\alpha_j}
			\left(\mathit{l}^{M\alpha_j}
			+\mathit{l}^{M\alpha_j-1}\left(\frac{M\kappa}{\mathit{l}+\kappa}
			\right)
			\sum_{t=1}^{\alpha_j}(\eta_j+t)+O\left(\mathit{l}^{M\alpha_j-2}\right)\right).
		\end{align*}
		
		Multiplying across all indices $j$ (and writing $s=\alpha_1+\cdots +\alpha_m$), we see that $\Phi_{m}(f_1,\ldots,f_m)$ is:
		\begin{align*}
			&\sum_{\mathit{l}=-\eta_{min}}^{n-1-\eta_{max}}
			\left(\frac{1}{\mathit{l}+\kappa}\right)^{Ms}
			\left(\mathit{l}^{Ms}
			+\mathit{l}^{Ms-1}\left(\frac{M\kappa}{\mathit{l}+\kappa}\right)
			\sum_{j=1}^m\sum_{t=1}^{\alpha_j}(\eta_j+t)+O\left(\mathit{l}^{Ms-2}\right)\right)\\
			&=\sum_{\mathit{l}=-\eta_{min}}^{n-1-\eta_{max}}
			\left(\frac{\mathit{l}}{\mathit{l}+\kappa}\right)^{Ms}+M\kappa
			\sum_{\mathit{l}=-\eta_{min}}^{n-1-\eta_{max}}
			\frac{\mathit{l}^{Ms-1}}{(\mathit{l}+\kappa)^{Ms+1}}
			\left(\sum_{j=1}^m\left[\alpha_j\eta_j+\frac{\alpha_j(\alpha_j+1}{2}\right]\right)+O\left(\frac{1}{\kappa} \right).
		\end{align*}
		We will now quickly estimate the terms in our expression for $\Phi_m$ separately. By the fundamental theorem of calculus:
		\begin{align*}
			\left|\left(\frac{\mathit{n}}
			{\mathit{n}+\kappa} \right)^{Ms}
			-\left(\frac{\mathit{l}}
			{\mathit{l}+\kappa} \right)^{Ms}\right|\leq O\left(\frac{n-\mathit{l}}{\kappa} \right).
		\end{align*}
		This estimate affords us the expansion:
		\begin{align*}
			\sum_{\mathit{l}=-\eta_{min}}^{n-1-\eta_{max}}
			\left(\frac{\mathit{l}}
			{\mathit{l}+\kappa} \right)^{Ms}
			&=\sum_{\mathit{l}=1}^{n}
			\left(\frac{\mathit{l}}
			{\mathit{l}+\kappa} \right)^{Ms}
			-\sum_{\mathit{l}=n-1-\eta_{max}}^n \left(\frac{n}
			{n+\kappa} \right)^{Ms}\\
			&+
			\sum_{\mathit{l}=n-1-\eta_{max}}^n\left( 
			\left(\frac{\mathit{n}}
			{\mathit{n}+\kappa} \right)^{Ms}
			-\left(\frac{\mathit{l}}
			{\mathit{l}+\kappa} \right)^{Ms}\right)
			+O\left(\frac{1}{\kappa}\right).
		\end{align*}
		This last display is just:
		\begin{align*}
			\sum_{\mathit{l}=1}^{n}
			\left(\frac{\mathit{l}}
			{\mathit{l}+\kappa} \right)^{Ms}
			-\left(1+\eta_{max}\right) \left(\frac{n}
			{n+\kappa} \right)^{Ms}+O\left(\frac{1}{\kappa} \right).
		\end{align*}
		By a similar argument:
		\begin{align*}
			\sum_{\mathit{l}=-\eta_{min}}^{n-1-\eta_{max}}
			\frac{\kappa\mathit{l}^{Ms-1}}
			{\left(\mathit{l}+\kappa\right)^{Ms+1}}&=
			\sum_{\mathit{l}=1}^{n}\frac{\kappa \mathit{l}^{Ms-1}}
			{(\mathit{l}+\kappa)^{Ms+1}}-\left(1+\eta\right)
			\left(\frac{\kappa n^{Ms-1}}{(n+\kappa)^{Ms+1}}\right)
			+O\left(\frac{1}{\kappa^2}\right)\\
			&=\sum_{\mathit{l}=1}^{n}\frac{\kappa \mathit{l}^{Ms-1}}
			{(\mathit{l}+\kappa)^{Ms+1}}
			+O\left(\frac{1}{\kappa}\right)
			.
		\end{align*}
		
		Plugging these approximations back into our equation for $\Phi_m$, we obtain:
		\begin{align*}
			\Phi_m(f_1,\ldots,f_m)
			&=
			\sum_{\mathit{l}=1}^{n}
			\left(\frac{\mathit{l}}
			{\mathit{l}+\kappa} \right)^{Ms}
			-\left(1+\eta_{max}\right) \left(\frac{1}
			{1+\tau} \right)^{Ms}\\
			&+\left(\sum_{\mathit{l}=1}^{n}\frac{M\kappa \mathit{l}^{Ms-1}}
			{(\mathit{l}+\kappa)^{Ms+1}} \right)
			\left(\sum_{j=1}^m\left[\alpha_j\eta_j+\frac{\alpha_j(\alpha_j+1)}{2}\right]\right)+O\left(\frac{1}{\kappa} \right).
		\end{align*}
		Therefore, setting $\alpha_j = \sum_{i : \sigma(i) = j} a_i$ and $\beta_j= \sum_{i : \sigma(i) = j} b_i$ for each $j$, in order to show that the higher cumulants vanish we must show that the following expression vanishes in the large dimensional limit:
		\begin{align*}
			\sum_{m=1}^k \frac{(-1)^m}{m}\sum_{\sigma: [k]\twoheadrightarrow [m]}\Bigg[
			\sum_{\mathit{l}=1}^{n}
			\left(\frac{\mathit{l}}
			{\mathit{l}+\kappa} \right)^{Ms}
			-\left(1+\eta_{max}\right) \left(\frac{1}
			{1+\tau} \right)^{Ms}\\
			+\left(\sum_{\mathit{l}=1}^{n}\frac{M\kappa \mathit{l}^{Ms+1}}
			{(\mathit{l}+\kappa)^{Ms+1}} \right)
			\left(\sum_{j=1}^m\left[\alpha_j\eta_j
			+\frac{\alpha_j(\alpha_j+1)}{2}\right]\right)\Bigg].
		\end{align*} 
		By Lemma 5.1 in \cite{RV}, if $k\geq 2$:
		\begin{align*}
			\sum_{m=1}^k \frac{(-1)^m}{m}\sum_{\sigma: [k]\twoheadrightarrow [m]}\Bigg[
			\sum_{\mathit{l}=1}^{n}
			\left(\frac{\mathit{l}}
			{\mathit{l}+\kappa} \right)^{Ms}
			- \left(\frac{1}
			{1+\tau} \right)^{Ms}\Bigg]=0.
		\end{align*}
		By Lemma 5.5 in \cite{RV}, if $k\geq 3$:
		\begin{align*}
			\left(\frac{1}
			{1+\tau} \right)^{Ms}
			\sum_{m=1}^k \frac{(-1)^m}{m}\sum_{\sigma: [k]\twoheadrightarrow [m]}
			\eta_{max}=0.
		\end{align*}
		By Lemma 5.4 in \cite{RV}, if $k\geq 3$:
		\begin{align*}
			\left(\sum_{\mathit{l}=1}^{n}\frac{M\kappa \mathit{l}^{Ms+1}}
			{(\mathit{l}+\kappa)^{Ms+1}} \right)
			\sum_{m=1}^k \frac{(-1)^m}{m}\sum_{\sigma: [k]\twoheadrightarrow [m]}\Bigg[
			\left(\sum_{j=1}^m\left[\alpha_j\eta_j
			+\frac{\alpha_j(\alpha_j+1)}{2}\right]\right)\Bigg]=0.
		\end{align*} 
		We can then conclude that the higher cumulants vanish in the limit, and the linear statistic converges to a normal random variable. 
		
		It remains to compute the variance. Let $z^{a_1}\bar{z}^{b_1}$ and $z^{a_2}\bar{z}^{b_2}$ be monomials such that $a_1+a_2 = b_1+b_2$ (we will again denote this common quantity by $s$, for brevity). We have:
		\begin{align*}
		\mbox{Cov}(z^{a_1}\bar{z}^{b_1},z^{a_2}\bar{z}^{b_2})
		=\Phi_1(z^{a_1+a_2}\bar{z}^{b_1+b_2})-
		\Phi_2(z^{a_1}\bar{z}^{b_1},z^{a_2}\bar{z}^{b_2}),
		\end{align*}
		where:
		\begin{align*}
		\Phi_1(z^{a_1+a_2}\bar{z}^{b_1+b_2})&=
		\sum_{\mathit{l}=1}^{n}
		\left(\frac{\mathit{l}}
		{\mathit{l}+\kappa} \right)^{Ms}
		-\left(\frac{1}{1+\tau} \right)^{Ms}\\
		&+\left(\sum_{\mathit{l}=1}^{n}\frac{M\kappa \mathit{l}^{Ms+1}}
		{(\mathit{l}+\kappa)^{Ms+1}} \right)
		\left(
		\frac{(a_1+a_2)(a_1+a_2+1)}{2}\right)+o(1).
		\end{align*}
		And also:
		\begin{align*}
		\Phi_2(z^{a_1}\bar{z}^{b_1},z^{a_2}\bar{z}^{b_2})&=
		\sum_{\mathit{l}=1}^{n}
		\left(\frac{\mathit{l}}
		{\mathit{l}+\kappa} \right)^{Ms}
		-(1+\max(0,b_1-a_1))\left(\frac{1}
		{1+\tau} \right)^{Ms}\\
		&+\left(\sum_{\mathit{l}=1}^{n}\frac{M\kappa \mathit{l}^{Ms+1}}
		{(\mathit{l}+\kappa)^{Ms+1}} \right)
		\left(a_1(b_1-a_1)+\frac{a_1(a_1+1)}{2}
		+\frac{a_2(a_2+1)}{2}\right)+o(1).
		\end{align*}
		Canceling:
		\begin{align*}
			\mbox{Cov}(z^{a_1}\bar{z}^{b_1},z^{a_2}\bar{z}^{b_2})
			=\left(\sum_{\mathit{l}=1}^{n}\frac{M\kappa \mathit{l}^{Ms-1}}
			{(\mathit{l}+\kappa)^{Ms+1}} \right)a_1b_2+\max(0,b_1-a_1))\left(\frac{1}
			{1+\tau} \right)^{Ms}+o(1).
		\end{align*}
		Here we have used the identity:
		\begin{align*}
		\frac{(a_1+a_2)(a_1+a_2+1)-a_1(a_1+1)-a_2(a_2+1)}{2}
		&-a_1(b_1-a_1)\\
		&=a_1a_2-a_1(b_1-a_1)\\
		&=a_1b_2.
		\end{align*}
		To take the limit of the covariance, we will apply the Riemann sum approximation:
		\begin{align*}
			\sum_{\mathit{l}=1}^{n}\frac{\kappa \mathit{l}^{Ms-1}}
			{(\mathit{l}+\kappa)^{Ms+1}}&=
			\frac{\kappa}{n}\left(\frac{1}{n}
			\sum_{\mathit{l}=1}^{n}
			\frac{\left(\mathit{l}/n\right)^{Ms-1}}
			{(\mathit{l}/n+\kappa/n)^{Ms+1}}\right)\\
			&\to\tau\int_{0}^1 \left(\frac{x^{Ms-1}}{(x+\tau)^{Ms+1}}\right)dx.
		\end{align*}
		The limiting covariance is then:
		\begin{align*}
		M\tau a_1b_2\int_0^1 \left(\frac{x^{Ms-1}}{(x+\tau)^{Ms+1}}\right)dx
		+\max(0,b_1-a_1))\left(\frac{1}
		{1+\tau} \right)^{Ms}.
		\end{align*}
		To interpret the first term of the limiting covariance, notice that:
		\begin{align*}
			\frac{1}{\pi}\int_{|z|\leq (1/(1+\tau))^{M/2}}\left(\frac{\partial}{\partial z}z^{a_1}\bar{z}^{b_1}\right)
			\overline{\left(\frac{\partial}{\partial z}\bar{z}^{a_2}z^{b_2}\right)}
			d^2z&=
			\frac{ a_1 b_2}{\pi}\int_{|z|\leq (1/(1+\tau))^{M/2}}|z|^{2s-2}	
			d^2z\\
			&=
			{2 a_1b_2} \int_0^{(1/(1+\tau))^{M/2}}
			r^{2s-1}dr.
		\end{align*}
		We make the following substitution:
		\begin{align*}
		 x=&\tau \left(\frac{ r^{2/M}}{1-r^{2/M}}\right),\\
		 dx
		 =&\frac{2}{M\tau}x(\tau+x)r^{-1}dr.
		\end{align*}
		Our integral is now:
			\begin{align*}
			M \tau a_1b_2\int_{0}^1 \left(\frac{x}{\tau+x}\right)^{Ms}
			\left(\frac{1}{x(\tau+x)}\right)dx.
		\end{align*}
		To interpret the second term of the limiting covariance, one notices instead (arguing as we did to obtain (\ref{lim_var})):
		\begin{align*}
			\max(0,b_1-a_1))\left(\frac{1}
			{1+\tau} \right)^{Ms}&=
			\max(0,b_1-a_1)\left(\frac{1}
			{1+\tau} \right)^{(M/2)(a_1+b_1+a_2+b_2)}
			\\&=\frac{1}{2}
			\left<\left(z^{a_1}\bar{z}^{b_1}\right),
			\left(\bar{z}^{a_2}z^{b_2}\right)\right>
			_{H^{1/2}(|z|=(1+\tau)^{-M/2})}.
		\end{align*}
		We can conclude that the limiting variance (at least for real valued polynomials) is:
		\begin{align}
			\label{limvar_truncunit}
			\frac{1}{4\pi}\int_{|z|\leq (1/(1+\tau))^{M/2}} \left|\nabla f(z)\right|^2
			d^2z 
			+\frac{1}{2}
			\|f\|^2_{H^{1/2}(|1+\tau|^{-M/2})}.
		\end{align}
		Here we have again used $\left| \partial f(z)\right|^2=\frac{1}{4}|\nabla f(z)|^2$. This expression is as desired, and the proof of Theorem \ref{unit4mom} is finished.
		
		\section{Four Moment Universality}
		\label{4mom_sec}
		
		In this section, we prove Theorem \ref{ginibre4mom} by way of four moment universality (with the bound on the smallest singular value, Theorem \ref{linsmallsig}, playing a crucial role in the argument). The development in this section is based on an approach previously employed in \cite{PK} for independent entry matrices, which was itself based on the argument put forward in \cite{TV2}. 		
		
		We will on several occasions appeal to technical results obtained by Nemish during the course of his proof of the local $M$-fold circular law \cite{Nemish}:
		
		\begin{theorem}[Nemish]
			\label{nemish}
			 Let $f:\mathbb{C}\to \mathbb{R}$ be a fixed smooth function with compact support. Let $\lambda_1,\ldots,\lambda_{n}$ be the eigenvalues of $n^{-M/2}X^{(1)}\cdots X^{(M)}$, where each jointly independent factor $X^{(i)}$ is an $n \times n$ iid random matrix. If $|z_0|,(1-|z_0|)\geq \tau_0$ for some $\tau_0>0$, then for any $d\in(0,1/2]$:
			\begin{align}
				\left(\frac{1}{n}\sum_{j=1}^{n}f_{z_0}(\lambda_j)
				-\frac{1}{M\pi}\int_{|z|\leq 1}f_{z_0}(z)|z|^{2/M-2} \right)
				\prec n^{-1+2d}\|\Delta f\|_{1}.
			\end{align}
			Here, $f_{z_0}$ is the $n^{-d}$ rescaling of $f(z)$ around $z_0$:
			\begin{align*}
				f_{z_0}(z)=n^{2d}f(n^d(z-z_0)).
			\end{align*}
		\end{theorem}
		
		Similar local law results for independent-entry matrices have also been obtained in \cite{BYY, BYY2, EKJ, TV2, Ylc}.  The notation $A\prec B$ appearing in the local $M$-fold circular law denotes stochastic domination:
		
		\begin{definition}
			Let $A_n, B_n\in\mathbb{C}$ be two sequences of random variables for $n\in\mathbb{N}$. The sequence $B_n$ is said to stochastically dominate $A_n$ (written $A_n\prec B_n$) if, for any $\epsilon>0$ and $D>0$, the following holds:
			\begin{align}
				\mathbf{P} \left\{|A_n|\geq n^\epsilon |B_n| \right\} \leq C_{D}n^{-D} .
			\end{align}
		\end{definition}
		
		This explains the presence of the spectral bulk condition $\tau_0<|z|<1-\tau_0$ in Theorem \ref{ginibre4mom}; if the arguments in \cite{Nemish} could be extended to the spectral edge then our argument would be extended to the spectral edge as well. Notice that this is asking more than the extension of the \textit{result} in \cite{Nemish}; indeed this has recently been accomplished at the origin in \cite{GNT}, but the arguments used do not appear to translate into the context of the argument presented here.
		
		\subsection{Overview of the Argument}
		\label{girkooverview}
		
		Recall that the linearization matrix of the product $n^{-M/2}X^{(1)}\dotsm X^{(M)}$, denoted $Y$, is defined as follows: 
		\[ Y = \frac{1}{\sqrt{n}} \left( \begin{array}{ccccc}
			0 & X^{(1)} &  0 & \cdots & 0 \\
			0 & 0 & X^{(2)} & \cdots & 0 \\
			\vdots & \vdots & \vdots & \ddots & \vdots \\
			0 & \cdots & 0 & 0 & X^{(M-1)} \\
			X^{(M)} & 0 & \cdots & 0 & 0 
			\end{array} \right).\]
				
		For $\beta=1,2$, let $Y^{(\beta)}$ be the linearization matrix associated with the product $n^{-M/2}X^{(\beta,1)} \cdots X^{(\beta, M)}$ (with factor matrices obeying the assumptions in the statement of Theorem \ref{linlinstatthm} below). For any choice of complex $z$ in the spectral bulk (defined as in the statement of Theorem \ref{linlinstatthm}), define $Y^{(\beta)}(z)=Y^{(\beta)}-zI$, and define also:
		\begin{align}
			\label{wbz}
		W^{(\beta)}(z)= \left( \begin{array}{cc}
		0 & Y^{(\beta)}(z) \\
		\left(Y^{(\beta)}(z)\right)^{*} & 0 
		\end{array} \right). 
		\end{align}
		We let $\lambda_j^{(\beta)}(z)$, for $1\leq j \leq 2Mn$, denote the eigenvalues of $W^{(\beta)}(z)$. 
			
		Our argument will rest in large part on the classical Girko Hermitization trick, which we will use to get around various complications which stem from the failure of Hermiticity. The trick relies on the following identity for twice continuously differentiable test functions with compact support:
		\begin{align}
			f(\lambda)=\frac{1}{2\pi}\int_{\mathbb{C}}\Delta f(z)\log|\lambda-z|d^2z.
		\end{align}
		If $\iota_{1},\ldots,\iota_{Mn}$ are the eigenvalues of $Y^{(\beta)}$, then this formula becomes:
		\begin{align*}
			\sum_{j=1}^{Mn}f(\iota_j)&
			=\frac{1}{2\pi}\int_{\mathbb{C}}\Delta f(z)\log |\det(Y^{(\beta)}-z)|d^2z
			\\&=\frac{1}{4\pi}\int_{\mathbb{C}}\Delta f(z)\log |\det W^{(\beta)}(z)|d^2z.
		\end{align*}
		The upshot here is that the matrix $W^{(\beta)}(z)$ is Hermitian, and can therefore be analyzed by the tools of Hermitian random matrix theory. The cost we have incurred is the presence of the integral, which we will need to deal with. Using this formulation, we prove the following four moment universality result for the linear statistics of the linearization matrix.  
		\begin{theorem}
			\label{linlinstatthm}
			Suppose $Y^{(1)}$ is the linearization matrix associated with the product matrix $n^{-M/2}X^{(1,1)}\dotsm X^{(1,M)}$  and $Y^{(2)}$ is the linearization matrix associated with the product $n^{-M/2}X^{(2,1)}\dotsm X^{(2,M)}$ (where all factor matrices are mutually independent $n$ by $n$ iid matrices), and suppose that the atom distributions of the factors $X^{(1,i)}$ and $X^{(2,i)}$ match to four moments for $1\leq i \leq M$.  Let $f:\mathbb{C} \to \mathbb{R}$ be a fixed function with two continuous derivatives, supported in the spectral bulk $\tau_0<|z|<1-\tau_0$ for some fixed $\tau_0>0$.  If the linear statistic generated by the eigenvalues of $Y^{(1)}$ and $f$, denoted $N^{(1)}_n[f]$, converges in distribution to some limiting distribution $\chi$, then the linear statistic  generated by the eigenvalues of $Y^{(2)}$ and $f$, $N^{(2)}_n[f]$, necessarily converges in distribution to $\chi$ as well.
		\end{theorem}
		
		We can connect this linear statistic with the linear statistic of the product matrix $n^{-M/2} X^{(\beta,1)}\cdots X^{(\beta,M)}$ (with eigenvalues $\mu_1,\ldots,\mu_n$) through the formula (see for instance \cite{OS} for a complete derivation)
		\begin{align}
			\sum_{j=1}^{n}f(\mu_j)=\sum_{j=1}^{Mn}\left[\frac{1}{M} f(\iota_j^{M})\right],
		\end{align}
		where $\iota_1, \ldots, \iota_{Mn}$ are the eigenvalues of the linearization matrix of the product  $n^{-M/2} X^{(\beta,1)}\cdots X^{(\beta,M)}$.  
		Applying universality to the test function $M^{-1}f(z^{M})$ immediately extends fourth moment universality for linear statistics of linearization matrices to fourth moment universality for linear statistics of product matrices. Theorem \ref{linlinstatthm} then implies the following corollary:
		
		\begin{theorem}
			\label{fmlthm}
			Suppose that the product matrices $X^{(1)}=n^{-M/2}X^{(1,1)}\dotsm X^{(1,M)}$ and  $X^{(2)}=n^{-M/2}X^{(2,1)}\dotsm X^{(2,M)}$ are both products of mutually independent $n$ by $n$ iid random matrices, and suppose that the atom distributions of the factors $X^{(1,i)}$ and $X^{(2,i)}$ match to four moments for $1\leq i \leq M$.  Let $f:\mathbb{C} \to \mathbb{R}$ be a fixed function with two continuous derivatives, supported in the spectral bulk $\tau_0<|z|<1-\tau_0$ for some fixed $\tau_0>0$. If the linear statistic generated by the eigenvalues of $X^{(1)}$ and $f$, denoted $N^{(1)}_n[f]$, converges in distribution to some limiting distribution $\chi$, then the linear statistic generated by the eigenvalues of $X^{(2)}$ and $f$, $N^{(2)}_n[f]$, converges in distribution to $\chi$ as well.
		\end{theorem}
		
		Combining Theorem \ref{fmlthm} with the Gaussian result, Theorem \ref{cltforginibre}, one immediately obtains Theorem \ref{ginibre4mom}, which is our objective. It remains, then, to prove Theorem \ref{linlinstatthm}, which the rest of this section is dedicated.  Throughout the proof, we will assume that all matrices under consideration feature exclusively real entries; the same proof goes forward for matrices with complex entries but with slightly more cumbersome notations.
					
		\subsection{Preliminaries}
			
		Here we collect some preliminary results and definitions which we will need in order to prove Theorem \ref{linlinstatthm}. Define an \textit{elementary matrix} to be a Hermitian matrix featuring one or at most two entries equal to 1, and all the other entries set to zero.  Therefore, adding a multiple of an elementary matrix to a Hermitian matrix $H$ changes either a single diagonal entry or two conjugate off-diagonal entries of $H$, and leaves the other entries undisturbed.
			
		For an $n \times n$ Hermitian matrix $H$ and an elementary matrix $V$, define:
			\begin{align}
				H_t&=H+\frac{1}{\sqrt{n}}tV,\\
				R_0(\zeta)&=(H-\zeta)^{-1},\\
				R_t(\zeta)&=(H_t-\zeta)^{-1},\\
				s_t(\zeta)&=\frac{1}{n}\mbox{Tr}R_t(\zeta).
			\end{align}
		We will also need to define an appropriate matrix norm:
			\begin{align*}
				\|A\|_{(\infty,1)}=\max_{1\leq i,j\leq n}|A_{ij}|.
			\end{align*}
		The following Taylor expansion type lemma is due to Tao and Vu (see Proposition 13 in \cite{TVWig}), and is proven by iterating the classical resolvent identity:
			\begin{lemma}\label{p2}
				Let $H$ be a Hermitian matrix, $V$ an elementary matrix, $t$ and $E$ real numbers, and $\eta >0$.  Take $\zeta = E + \sqrt{-1} \eta$. Let $k \geq 0$ be fixed.  Suppose we have:
				\begin{align*}
					|t|\times\|R_0(\zeta)\|_{(\infty,1)}=o(\sqrt{n}).
				\end{align*}
				Then we have the following Taylor expansion to order $k$ of the quantity $s_t(\zeta)$:
				\begin{align*}
					s_0+\sum_{j=1}^k n^{-j/2}c_jt^j+O\left(n^{-(k+1)/2}|t|^{k+1}
					\|R_0(\zeta)\|^{k+1}_{(\infty,1)}
					\min(\|R_0(\zeta)\|_{(\infty,1)},\frac{1}{n\eta})
					\right).
				\end{align*}
				The coefficients $c_j$ are independent of $t$ and obey the following estimate:
				\begin{align*}|c_j|\ll
					\|R_0(\zeta)\|_{(\infty,1)}^j\min \left(\|R_0(\zeta)\|_{(\infty,1)},\frac{1}{n\eta} \right).
				\end{align*}
			\end{lemma}
		 We will also need the following Monte Carlo sampling lemma (see Lemma 6.1 in \cite{TVPol}), a consequence of Chebyshev's inequality.  
		\begin{lemma}\label{p1}
		Let $(X,\mu)$ be a probability space and $F$ a square integrable function from $(X,\mu)$ to the real line. For $m$ independent $x_i$, distributed according to $\mu$,  define the empirical average:
		\begin{align*}
		S_m=\frac{1}{m}\sum_{i=1}^m F(x_i).
		\end{align*}
		Then for any $\delta >0$ the following estimate holds with probability at least $1-\delta$:
		\begin{align}
		\left|S_m-\int_X F d\mu \right| \leq \frac{1}{\sqrt{\delta m}}\left(\int_X (F-\int_X Fd\mu)^2d\mu\right)^{1/2}.
		\end{align}
		\end{lemma}
			
		We will require the following technical lemma, which will take the place of Propositions 29 and 31 in \cite{TV2}:
			\begin{lemma}
				\label{techlem}
				Let $N^{(\beta)}_{I}=\emph{Card}\left\{i, 
				\lambda^{(\beta)}_i\in I\right\}$ be the counting function of the number of eigenvalues in an interval $I$ of $W^{(\beta)}(z)$, let $R_0(\zeta)$ denote the matrix $(W^{(\beta)}(z)-\zeta)^{-1}$ for any $\zeta\in\mathbb{C}$, and suppose that $\tau_0\leq |z| \leq 1-\tau_0$ for some fixed $\tau_0 > 0$. Then we have, uniformly and with overwhelming probability, the following bounds:
				\begin{align}
					\label{techlem1}
					N_{I}\leq n^{o(1)}(1+Mn|I|)
				\end{align}
				for all intervals $I$ (where $|I|$ here denotes the length of the interval $I$) and
				\begin{align}
					\label{techlem2}
					\left|R_0(\sqrt{-1}\eta)_{i,j}\right|\leq n^{o(1)}\left(1+\frac{1}{Mn\eta}\right) 
				\end{align} 
				for all $\eta > 1/n$ and $1\leq i,j \leq 2Mn$.  In addition, for any sufficiently small constant $c_0$, there exists an event which holds with probability at least $1 - O(n^{-\Omega(c_0) + o(1)})$ such that conditioned on this event, 
				\begin{equation} \label{eq:uniformetabnd}
					\sup_{\eta > 0} \|R_0(\sqrt{-1} \eta) \|_{(\infty, 1)} \leq O(n^{O(c_0)})
				\end{equation}
				with overwhelming probability.  
			\end{lemma}
			
		\begin{proof}
			
		Both \eqref{techlem1} and \eqref{techlem2} are immediate from the proof of the local $M$-fold circular law \cite{Nemish}. First, (\ref{techlem1}) follows from the fact that the number of classical eigenvalue locations of $W^{(\beta)}(z)$ in an interval is proportional to the length of the interval, and from the eigenvalue rigidity argument which was used to obtain (26) from Theorem 5 in \cite{Nemish} (see also Lemma 5.1 in \cite{BYY}). Second, (\ref{techlem2}) follows immediately from (ii) in the proof of Lemma 17 in \cite{Nemish}.
		
		\eqref{eq:uniformetabnd} follows from the proof of Lemma 46 in \cite{TV2}, where instead of applying Proposition 31 from \cite{TV2}, one applies \eqref{techlem2}.  We omit the details.  
		\end{proof}
			
		\subsection{Proof of Theorem \ref{linlinstatthm}}
		In this section, we will prove Theorem \ref{linlinstatthm} by way of the following, somewhat more general, result.
			
		\begin{theorem}
		\label{4moments}
		Let $G: \mathbb{R} \to \mathbb{C}$ be any smooth function with five bounded derivatives, and let $f:\mathbb{C} \to \mathbb{R}$ be a function with two continuous derivatives supported in the spectral bulk $\tau_0<|z|<1-\tau_0$ for some fixed $\tau_0>0$.  Let $N^{(\beta)}_n[f]$ (for $\beta=1,2$) denote the linear statistics of the linearization matrices corresponding to two products of $M$ independent $n$ by $n$ iid random matrices: $n^{-M/2} X^{(1,1)}\dotsm X^{(1,M)}$ and $n^{-M/2} X^{(2,1)}\dotsm X^{(2,M)}$. Assume furthermore that the entry distributions of $X^{(1,i)}$ match the entry distributions of $X^{(2,i)}$ to four moments for $1\leq i\leq M$.  Then there exist constants $C,A>0$ such that:
		\begin{align}
		\left|\mathbf{E}G(N^{(1)}_n[f])
		-\mathbf{E}G(N^{(2)}_n[f]) \right|	
			\leq C n^{-A}.
		\end{align}
		\end{theorem}
			
		Notice that Theorem \ref{linlinstatthm} follows as a simple corollary by the Fourier inversion formula (as we can take $G$ such that $\mathbf{E}[G(N^{(\beta)}_n[f])]$ is the characteristic function of $N^{(\beta)}_n[f]$ -- see \cite{TV2} or \cite{PK}). The following proof of Theorem \ref{4moments} is a combination of the proof of a similar result in \cite{TV2} with a Monte Carlo sampling argument.
			
		\begin{proof}[Proof of Theorem \ref{4moments}]
		The proof is divided into three steps: the first is a preprocessing step which reformulates the statement we wish to prove into a statement about finite sums. The second step reduces the problem into a statement about Stieltjes transforms, and the third uses resolvent swapping and a Taylor expansion to conclude the argument.
		
		\textit{Step 1: Reformulating the Problem}.
		We will first need a variance bound which follows from the proof of the local $M$-fold circular law \cite{Nemish}. By (2.9) in \cite{Nemish}, we have that, for any $D>0$ and any $\epsilon>0$, with probability at least $1-O_{D,\epsilon}(n^{-D})$, the following estimate holds:
			\begin{align}
				\label{varestmate}
				\int_{\mathbb{C}}|\Delta f(z)|^2 \left|\sum_{j=1}^{2Mn} \log |\lambda_j^{(\beta)}(z)|-\log |\gamma_j(z)| \right|^2d^2z \leq O_{D,\epsilon}\left( n^{\epsilon}\right).
			\end{align}
		Here, $\gamma_j(z)$ represent the classical locations of the eigenvalues of $W^{\beta}(z)$, as defined in \cite{Nemish} (the exact definition of $\gamma_j(z)$ and properties thereof will not be essential to our argument, so we do not provide an overview of this material here; the important feature will simply be that these are deterministic quantities).
			
		By the Girko Hermitization trick, to prove the desired result it suffices to establish the following estimate (see the discussion in Subsection \ref{girkooverview}):
			\begin{align*}
				\Bigg| \mathbf{E}G\left(\int_{\mathbb{C}} \Delta f(z)\sum_{j=1}^{2Mn}[\log|\lambda_j^{(1)}(z)|]d^2z \right)-\mathbf{E}
				G\left(\int_{\mathbb{C}} \Delta f(z)\sum_{j=1}^{2Mn}[\log|\lambda_j^{(2)}(z)|]d^2z \right)\Bigg|\\
				\leq Cn^{-A}.
			\end{align*}
		Fix some choice of positive constant $k_0>0$ to be determined later, and for $K=\left \lceil n^{k_0}\right \rceil$, let $z_1,\ldots,z_K$ denote independent random elements selected uniformly at random from the support of $f$, independent of the product matrices.  Let $L > 0$ be the Lebesgue measure of the support of $f$; since $f$ is supported in the spectral bulk, it follows that $L=O(1)$.  Define the two stochastic Riemann sums $S^{(1)}_K$ and $S^{(2)}_K$:
			\begin{align*}
				S^{(\beta)}_{K}=\frac{L}{K}\sum_{i=1}^K\left(\sum_{j=1}^{2Mn}\Delta f (z_i)\left[
				\log|\lambda_j^{(\beta)}(z_i)|-\log|\gamma_j(z_i)|\right] \right).
			\end{align*}
		By (\ref{varestmate}) (with some choice of $D>0$ to be determined) and the Monte Carlo sampling lemma (Lemma \ref{p1}):
			\begin{align*}
				\mathbf{P}\left\{\left|\left(\int_{\mathbb{C}} \Delta f(z) \sum_{j=1}^{2Mn} [\log|\lambda_j^{(\beta)}(z)|-\log|\gamma_j(z)|]d^2z\right)-S^{(\beta)}_K
				\right|^2
				\leq O\left(\frac{n^{\epsilon}}{{\delta K}}\right)\right\}
				\\ \geq 1-O(n^{-D})-\delta .
			\end{align*}
		Choosing $\epsilon=k_0/4$ and $\delta=n^{-k_0/4}$ and $D$ sufficiently large, this becomes:
			\begin{align*}
				\mathbf{P}\left\{\left|\left(\int_{\mathbb{C}} \Delta f(z) \sum_{j=1}^{2Mn} [\log|\lambda_j^{(\beta)}(z)|-\log|\gamma_j(z)|]d^2z\right)
				-S^{(\beta)}_K\right|^2
				\leq O\left(\frac{1}{n^{k_0/8}}\right)\right\}\\
				\geq 1 - O(n^{-k_0/8}).
			\end{align*}
		Since we are dealing with the expectations of bounded functions, this estimate will suffice for our purposes. Indeed, it follows that we may write:
			\begin{align*}
				&\mathbf{E} G\left(\int_{\mathbb{C}} \Delta f(z)\sum_{j=1}^{Mn}[\log|\lambda_j^{(1)}(z)|
				-\log|\gamma_j(z)|]d^2z \right)\\&-
				\mathbf{E}G\left(\int_{\mathbb{C}} \Delta f(z)\sum_{j=1}^{Mn}[\log|\lambda_j^{(2)}(z)|-\log|\gamma_j(z)|]d^2z \right)\\
				&=\mathbf{E}G(S^{(1)}_K)- \mathbf{E}G (S^{(2)}_K)+O_G\left(\frac{1}{n^{k_0/8}} \right).
			\end{align*}
			
		\textit{Step 2: Additional Reductions.}
		We have now replaced the integral which resulted from the Girko Hermitization trick with a sum of $K$ terms; this more or less reduces the problem to the one solved in \cite{TV2} (as we may deal with each summand separately and just add the resulting errors), and the remainder of the proof just follows the argument made in \cite{TV2}. We may condition on the precise choice of points $z_1,\ldots,z_K$, and do so now.
			
		Expanding the logarithm by way of the fundamental theorem of calculus, one obtains (where we use $s_\beta(z,\sqrt{-1}\eta)$ to denote quantity $\frac{1}{2Mn} \mbox{Tr}[W^{(\beta)}(z)-\sqrt{-1}\eta]^{-1}$) for any $z\in\mathbb{C}$ with $|z| \leq n^2$ (say):
			\begin{align*}
				\log |\det(W^\beta(z)) |&=
				\log\left|\det(W^\beta(z)-in^{100})\right|
				-2Mn\mbox{Im}\int_0^{n^{100}}
				s_\beta(z,\sqrt{-1}\eta)d\eta\\
				&=200Mn\log(n)+O(n^{-10})-2Mn\mbox{Im}\int_0^{n^{100}}
				s_\beta(z,\sqrt{-1}\eta)d\eta.
			\end{align*}
			Since $G$ has a bounded first derivative and may be translated, in order to prove Theorem  \ref{4moments} it is sufficient to show that:
			\begin{align*}
				&\Bigg|\mathbf{E} G\left(\frac{nL}{K}\sum_{i=1}^K \mbox{Im}\int_0^{n^{100}}
				s_1(z_i,\sqrt{-1}\eta)d\eta \right)\\
				&\qquad-\mathbf{E} G\left(\frac{nL}{K}\sum_{i=1}^K \mbox{Im}\int_0^{n^{100}}
				s_2(z_i,\sqrt{-1}\eta)d\eta \right)\Bigg|
				\leq Cn^{-A}.
			\end{align*}
			
			\textit{Step 3: Resolvent Swapping}.
			We would like to simplify this last expression via the resolvent swapping lemma, Lemma \ref{p2}, however we must first make sure that the matrices $W^{(1)}(z_i)$ and $W^{(2)}(z_i)$ satisfy the assumptions of the lemma for each $z_i$, which is the one part of the proof where we will require control over the smallest singular value of linearization matrices. The argument here is the exact clone of the same argument in \cite{TV2}, and we do not reproduce it here. Indeed, by the arguments given in Section 8 of \cite{TV2} (with our Theorem \ref{linsmallsig} in the place of their Proposition 27 and our Lemma \ref{techlem} in place of their Propositions 29 and 31) we can conclude that, with probability  $1-O\left(n^{-\Theta_0 c_0}\right)$  the matrices $W^{(1)}(z_i)$ and $W^{(2)}(z_i)$, for any fixed $i$, are such that resolvent swapping lemma applies. Here the constant $c_0 > 0$ is sufficiently small and $\Theta_0>0$ is absolute.  Since $K=O(n^{k_0})$, by taking $k_0$ small enough and using the union bound we may safely assume that this condition is satisfied  for each $z_i$, and thus every summand in $S_K^{(\beta)}$.
			
			We may now safely swap entries. We demonstrate this process by first swapping the very first entry in the very first pair of factor matrices, $X_{1,1}^{(1,1)}$ and $X_{1,1}^{(2,1)}$. This will require some additional notations, which we develop presently. Let $s_{(1,1)}(z_i,\sqrt{-1}\eta)$ denote the Stieltjes transform of $W^{(1)}(z_i)$, and let $s_{(2,1)}(z_i,\sqrt{-1}\eta)$ denote the Stieltjes transform of the matrix formed by taking $W^{(1)}(z_i)$ and replacing the distribution of the entry $X^{(1,1)}_{1,1}$ (which we will denote $\xi^{(1,1)}_{1,1}$) with the distribution of the entry $X^{(2,1)}_{1,1}$ (which we will denote $\xi^{(2,1)}_{1,1}$). Also let $s^{\prime}_{(1,1)}(z_i,\sqrt{-1}\eta)$ denote the Stieltjes transform of the matrix formed by taking $W^{(1)}(z_i)$ and replacing the distribution of the entry $X^{(1,1)}_{1,1}$ with the distribution of the random variable which is identically equal to zero.
			
			An application of Lemma $\ref{p2}$ (with $W^{(\beta)}(z_i)$ in the role of the Hermitian matrix $H$) produces the expansion (for $\beta=1,2$):
			\begin{align*}
				s_{(\beta,1)}(z_i,\sqrt{-1}\eta)
				&=s^{\prime}_{(1,1)}(z_i,\sqrt{-1}\eta)
				+\sum_{j=1}^4 
				\left(\xi^{(\beta,1)}_{1,1}\right)^j
				n^{-j/2}c_j(\eta)
				\\&+O\left(
				n^{-5/2+O(c_0)}\min\left(1, \frac{1}{n\eta} \right)
				\right).
			\end{align*}
			Define the constants $\tilde{c}_j$ as follows:
			\begin{align*}
				\tilde{c}_j=n\mbox{Im}\int_0^{n^{100}}c_j(\eta)d\eta
			\end{align*}
			By Lemma $\ref{p2}$ and \eqref{eq:uniformetabnd}, the coefficients $\tilde{c}_j$ satisfy:
			\begin{align} \label{eq:cjbound}
				\left|\tilde{c}_j\right|\leq O\left(n^{O(c_0)}\right).
			\end{align}
			By a trivial integration, we have:
			\begin{align}
				n\int_0^{n^{100}}\min\left(1,\frac{1}{n\eta}\right)d\eta
				\leq O\left(\log(n)\right). 
			\end{align}
			And therefore we may write the expansion:
			\begin{align*}
				\frac{n}{K}&\sum_{i=1}^{K}\mbox{Im}\int_{0}^{n^{100}}
				s_{(\beta,1)}(z_i,\sqrt{-1}\eta)d\eta \\
				&=
				\frac{n}{K}\sum_{i=1}^{K}
				\mbox{Im}\int_{0}^{n^{100}}s^{\prime}_{(1,1)}(z_i,\sqrt{-1}\eta)d\eta
				+\sum_{j=1}^4
				\left( \left(\xi^{(\beta,1)}_{1,1}\right)^j n^{-j/2}\tilde{c}_j\right)
				\\ &\qquad\qquad+O\left(
				n^{-5/2+O(c_0)}
				\right).
			\end{align*}	
			Computing the fourth order Taylor expansion of the function $G$, one obtains:
			\begin{align*}
				\mathbf{E} G&\left(\frac{nL}{K}\sum_{i=1}^{K}\mbox{Im}\int_{0}^{n^{100}}
				s_{(\beta,1)}(z_i,\sqrt{-1}\eta)d\eta \right) \\
				&=
				\mathbf{E} G\left(\frac{nL}{K}\sum_{i=1}^{K}
				\mbox{Im}\int_{0}^{n^{100}}s^{\prime}_{(1,1)}(z_i,\sqrt{-1}\eta)d\eta \right) \\
				&\qquad+
				\mathbf{E} \sum_{k=1}^4
				\frac{L^k}{k!}{G^{(k)} \left(\frac{nL}{K}\sum_{i=1}^{K}
				\mbox{Im}\int_{0}^{n^{100}}s^{\prime}_{(1,1)}(z_i,\sqrt{-1}\eta)d\eta \right)}
				\left(\sum_{j=1}^4\left(\xi_{11}^{(\beta,1)}\right)^j n^{-j/2}\tilde{c}_j\right)^k \\
				&\qquad +O_G\left(
				n^{-5/2+O(c_0)}
				\right).
			\end{align*}
			Using the moment matching condition (i.e., the fact that the random variables $\xi_{11}^{(2,1)}$ and $\xi_{11}^{(1,1)}$ match to four moments) and bounding the remaining terms using \eqref{eq:cjbound}:
			\begin{align*}
				\left|\mathbf{E}G\left(\frac{nL}{K}\sum_{i=1}^{K}\mbox{Im}\int_{0}^{n^{100}}
				s_{(1,1)}(z_i,\sqrt{-1}\eta)d\eta \right)
				-\mathbf{E}G\left(\frac{nL}{K}\sum_{i=1}^{K}\mbox{Im}\int_{0}^{n^{100}}
				s_{(2,1)}(z_i,\sqrt{-1}\eta)d\eta \right) \right|\\
				\leq O\left(
				n^{-5/2+O(c_0)}
				\right).
			\end{align*}
			Repeating this process for all $Mn^2$ entries, summing and applying the triangle inequality concludes the proof.
			 \end{proof}
		
		\section{Proof of Theorems \ref{corruni} and \ref{corruni2}}
		\label{corrunisec}
		
		In this section, we will use Theorem \ref{linsmallsig} to prove Theorem \ref{corruni}, which establishes local universality for the correlation functions. The proof of Theorem \ref{corruni2} is virtually identical, except one swaps the call to Theorem 2.1 in \cite{TVPol} (which deals with polynomials with complex coefficients) with a call to Theorem 3.1 in \cite{TVPol} (which deals with polynomials with real coefficients), and is therefore omitted.
		
		As in the statement of Theorem \ref{corruni}, for $\beta\in\{1,2\}$ we let $Z^\beta_n$ denote the product of $M$ independent $n$ by $n$ iid matrices, with $p^{(k)}_\beta$ denoting the associated $k$-point correlation function for the $M$-th root eigenvalue process, and we assume the factor matrices of $Z^1_n$ and $Z^2_n$ match to four moments. Let 
		\[ f^\beta_n(z) = \det(z^M I - Z_n^\beta) \]
		denote the polynomial whose roots form the $M$-th root eigenvalue process associated with $Z^\beta_n$. Additionally, we let $z_1,\ldots,z_k$ be complex numbers (which are allowed to depend on $n$) located in the spectral bulk $\tau_0\leq n^{-1/2}|z_i|\leq 1-\tau_0$, and let $G:\mathbb{C}^k\to\mathbb{C}$ denote a smooth function supported on the polydisc $B(0,r_0)^k$, where $r_0$ is a small constant which is allowed to depend on $\tau_0$.
		
		The method of proof consists mainly of an application of Theorem 2.1 in \cite{TVPol}, the statement of which we reproduce as Theorem \ref{randpollem} below, a wide ranging result concerning the universality of zeros of random polynomials. We will use $N_{B(z, \rho)}(f)$ to denote the number of zeros of $f$ in $B(z,\rho)$, the disk centered at $z$ with radius $\rho$.
		
		\begin{theorem} [Tao--Vu]
			\label{randpollem}
			Let $C_1, r_0\geq 1, 1\geq c_0 \geq 0$ be real constants and let $a_0, k\geq 1$ be integer constants.  Set $A=100ka_0/c_0$. Let $f^1_n$ and $f^2_n$ be random polynomials of degree at most $n$, and let $z_1,\ldots,z_k$ be $k$ points in the complex plane (which are allowed to depend on $n$). 
			Assume that three conditions holds: 
			\begin{enumerate}
				\item (Non-degeneracy) The probability that either polynomial is identically zero is at most $C_1 n^{-A}$.  
				\item (Non-clustering) For $r \geq 1$, one has $N_{B(z_i,r)}(f^1_n)\leq C_1n^{1/A}r^2$ with probability at least $1-C_1 n^{-A}$, and similarly for $f^2_n$.
				\item (Comparability of log-magnitudes) Given any $1\leq k^\prime \leq n^{c_0}$ and complex number $z^\prime_1,\ldots,z^\prime_{k'} \in \cup_{i=1}^k B(z_i,20r_0)$, and any smooth function $F: \mathbb{C}^{k'} \to \mathbb{C}$ with the derivative bound $|\nabla^a F(w)|\leq n^{c_0}$ for all $0\leq a\leq a_0$ and $w\in \mathbb{C}^{k^\prime}$, one has:
				\begin{align*}
					\left|\mathbf{E}F\left(\log(|f^1_n(z_1^\prime)|), \dotsm, 
					\log(|f^1_n(z_{k^\prime}^\prime)|)\right)-
					\mathbf{E}F\left(\log(|f^2_n(z_1^\prime)|),\dotsm, 
					\log(|f^2_n(z_{k^\prime}^\prime)|)\right)\right|\\
					\leq C_1n^{-c_0},
				\end{align*}
				with the convention that $F$ vanishes when one or more of its arguments is undefined.
			\end{enumerate}
			Let $G: \mathbb{C}^k \to \mathbb{C}$ be a smooth function supported on the polydisc $B(0,r_0)^k$ such that, for some $C_2 > 0$, $G$ obeys the derivative bound $|\nabla^a G(w)|\leq C_2$  for all $0\leq a\leq a_0 + 2k + 1$ and all $w\in \mathbb{C}^k$. Then (letting $p_{\beta,n}^{(k)}$ be the $k$-point correlation function of the zeros of the random polynomial $f_{n}^\beta$):
			\begin{align*}
				\Bigg|\int_{\mathbb{C}^k} G(w_1 ,\ldots, w_k)p_{1,n}^{(k)}(z_1+w_1,\ldots, 
				z_k+w_k)d^2w_1\ldots d^2w_k\\
				-\int_{\mathbb{C}^k} G(w_1 ,\ldots, w_k)p_{2,n}^{(k)}(z_1+w_1,\ldots, 
				z_k+w_k)d^2w_1\ldots d^2w_k
				\Bigg|\\
				\leq O\left(n^{-c_0/4}\right),
			\end{align*}
			where the implied constant depends only on $C_1, r_0, c_0, k, a_0$, and linearly on $C_2$.  
		\end{theorem}
		\begin{remark} \label{rem:scaling}
		Theorem \ref{randpollem} is designed to handle cases where the mean spacing between zeros is on the order of a constant. This is consistent with the $M$-th root eigenvalue process, where the mean spacing between points in the bulk is comparable to $1$.  To handle cases where the mean spacing is not on the order of a constant, one would need to generalize Theorem \ref{randpollem}; see Remark 2.5 in \cite{TVPol} for further details.  If one wishes to prove a version of Theorem \ref{corruni} using a different scaling convention (rather than the $M$-th root eigenvalue process), one would need to utilize this more general version of Theorem \ref{randpollem}.  
		\end{remark}
		
		The proof of Theorem \ref{corruni} will boil down to making sure that the three conditions in this preceding theorem hold.  Let us show instead that the analogous result holds for the characteristic polynomials of the linearization matrices, $Y^1(0)$ and $Y^2(0)$ defined by
		\[ Y^\beta(z) = \left( \begin{array}{ccccc}
			-zI & X^{(\beta,1)} &  0 & \cdots & 0 \\
			0 & -zI & X^{(\beta,2)} & \cdots & 0 \\
			\vdots & \vdots & \vdots & \ddots & \vdots \\
			0 & \cdots & 0 & -zI & X^{(\beta,M-1)} \\
			X^{(\beta,M)} & 0 & \cdots & 0 & -zI 
			\end{array} \right).\] 
		Indeed, since we have the identity $|f_n^\beta(z)|=|\det(Y_n^\beta(z))|$ (which follows from induction and the Schur determinant identity), we may assume that $f^1_n$ and $f^2_n$ are the characteristic polynomials of $Y^1(0)$ and $Y^2(0)$, respectively.  Choose $r_0 > 0$ so small that $B(z_i,20r_0)$ is still in the spectral bulk $\tau^\prime_0\leq n^{-1/2}|z|\leq 1-\tau^\prime_0$ for some positive $\tau_0^\prime$ smaller than $\tau_0$, and let $c_0$ denote an arbitrarily small positive constant.  We now verify the three conditions of Theorem \ref{randpollem}.  
		
		The first condition, the non-degeneracy condition, is immediate since these are both characteristic polynomials of matrices. Second, we need the non-clustering property, ensuring that $N_{B(z_i,\rho)}(f^\beta_n)\leq C_1n^{1/A}\rho^2$ for $\rho\geq 1$ with probability at least $1-O(n^{-A})$, but this is just a consequence of Nemish's local circular law for the linearization matrices $Y^\beta(0)$, which follows from the results in \cite{Nemish} if we choose the constant $C_1$ sufficiently large (for instance, by approximating the indicator function of the relevant disc by smooth functions). We therefore focus on establishing the third condition: comparability of log-magnitudes. We will demonstrate this using some of the same arguments we used in establishing Theorem \ref{ginibre4mom}, which were originally introduced in \cite{TV2}.  Because the arguments here are similar to those in \cite{TV2}, we mostly provide a sketch of the details and explain those portions which differ from \cite{TV2}.  
		
		We need to show that given any $k^\prime\leq n^{c_0}$ and any collection of complex numbers $z^\prime_1,\ldots,z_{k'}^\prime$ in the spectral bulk, and for any function $F$ satisfying the estimate $\left|\Delta^a F\right|\leq n^{c_0}$ for $0 \leq a\leq 5$, we have:
		\begin{align*}
			\left|\mathbf{E} F
			\left(\log|f_n^1(z^\prime_1)|,
			,\ldots,\log|f_n^1(z^\prime_{k^\prime})|
			\right)
			- \mathbf{E}F\left(
			\log|f_n^2(z^\prime_1)|,\ldots,\log|f_n^2(z^\prime_{k^\prime})|
			\right) \right|\\
			\leq C_1n^{-c_0}.
		\end{align*}
		
		We will also first assume that $k^\prime=1$ to keep the presentation and the notations simple, but the argument generalizes easily.
		We will deal with the linearization matrices using the same strategy employed during the proof of Theorem \ref{4moments}, albeit without the Monte Carlo sampling step. Define the matrices $W^{(\beta)}(z_1)$, for $\beta\in\{1,2\}$, as in \eqref{wbz}. Let $s_\beta(z^\prime_1,\sqrt{-1}\eta)$ denote the Stieltjes transform of $W^{(\beta)}(z^\prime_1)$, which is defined as $\frac{1}{2Mn}\mbox{tr}(W^{(\beta)}(z^\prime_1)-\sqrt{-1}\eta)^{-1}$. We have (by the argument presented in Step 2 of the proof of Theorem \ref{4moments}):
		\begin{align*}
			&\left|\mathbf{E}
			{F}\left(\log \left|\det Y^1(z^\prime_1)\right|\right) 
			- \mathbf{E} {F}\left(\log \left|\det Y^2_n(z^\prime_1)\right|\right)
			\right|\\&=
			\left|\mathbf{E}
			{F}\left(\frac{1}{2}\log \left|\det W^1(z^\prime_1)\right|\right) 
			- \mathbf{E} {F}\left(\frac{1}{2}\log \left|\det W^2_n(z^\prime_1)\right|\right)
			\right|
			\\
			&=\Bigg|
			\mathbf{E}{\tilde F}\left({Mn}
			\mbox{Im}\int_0^{n^{100}}s_1(z^\prime_1,\sqrt{-1}\eta)d\eta
			\right)
			\\
			&\qquad-\mathbf{E}{\tilde F}\left({Mn}
			\mbox{Im}\int_0^{n^{100}}s_2(z^\prime_1,\sqrt{-1}\eta)d\eta
			\right)
			\Bigg|+O\left(n^{-10}\right)
		\end{align*}
		by taking $c_0$ sufficiently small.  Here, $\tilde{F}$ is a translation of $F$ (exactly as was done in Step 2 of the proof of Theorem \ref{4moments}).  
		Applying the same argument as in Step 3 of the proof of Theorem \ref{4moments} (and also as in \cite{TV2}), and in particular using Theorem \ref{linsmallsig}, we see that Lemma \ref{p2}, the resolvent swapping lemma, applies, and that $Mn^2$ separate applications of the the resolvent swapping lemma (combined with a Taylor expansion of $\tilde{F}$, again exactly as in Step 3 of the proof of Theorem \ref{4moments}), provides:
		\begin{align*}
			\left|\mathbf{E}
			{F}\left(\log \left|\det Y^1(z^\prime_1)\right|\right) 
			-\mathbf{E} {F}\left(\log \left|\det Y^2_n(z^\prime_1)\right|\right)
			\right|&\leq C_{{F}}Mn^2n^{-5/2+O(c_0)}\\
			&\leq O\left(n^{-c_0}\right)
		\end{align*}
		by taking $c_0$ sufficiently small, 
		which establishes all three conditions from Theorem \ref{randpollem}. If $k^\prime>1$, we simply apply this argument $k^\prime$ times (one for each argument of ${F}$), and use the upper bound $k^\prime\leq O(n^{c_0})$ (by taking $c_0$ smaller if necessary) to show that the error is still sufficiently small. By Theorem \ref{randpollem}, the result is then established.
				
\bibliography{products}
\bibliographystyle{abbrv}
\end{document}